\documentclass[a4paper,11pt,twoside]{amsart}

\usepackage{latexsym}
\usepackage{amssymb}
\usepackage{amscd}
\usepackage[all]{xy}
\usepackage[mathscr]{euscript}
\usepackage{dsfont}
\usepackage{hyperref}
\usepackage{graphicx}
\usepackage{amsfonts}
\usepackage{amsmath}
\usepackage{amsthm}
\usepackage{aliascnt}

\relax 
\RequirePackage{xspace}

\newcommand{\thought}[1]{}
\renewcommand{\thought}[1]{ \textbf{[#1]}}

\usepackage{enumerate}        
\newenvironment{roenumerate}{\begin{enumerate}[\upshape (i)]}{\end{enumerate}}

\newcommand{\iso}{\cong}
\newcommand{\comp}{\circ}

\newcommand\nc {\newcommand}
\newcommand\rnc{\renewcommand}

\newcount\blopone
\newcount\xone
\newcount\xtwo
\newcount\ytwo

\newtheorem{theorem}{Theorem}[subsection]
\newtheorem{thmInt}{Theorem}[section]

\newaliascnt{prop}{theorem}
\newtheorem{prop}[prop]{Proposition}
\aliascntresetthe{prop}
\newtheorem{com}[theorem]{Comment}
\newtheorem{apl}[theorem]{Application}
\newtheorem{exercise}[theorem]{Exercise}
\newtheorem{redu}[theorem]{Reduction}
\newtheorem{refinement}[theorem]{Refinement}
\newtheorem{summary}[theorem]{Summary}
\newtheorem{importnota}[theorem]{Important Notation}
\newtheorem{prblm}[theorem]{Problem}
\newaliascnt{notation}{theorem}
\newtheorem{notation}[notation]{Notation}
\aliascntresetthe{notation}
\newaliascnt{notationApp}{thmApp}

\aliascntresetthe{notationApp}
\newtheorem{explanation}[theorem]{Explanation}
\newaliascnt{defin}{theorem}
\newtheorem{defin}[defin]{Definition}
\aliascntresetthe{defin}
\newaliascnt{definApp}{thmApp}

\aliascntresetthe{definApp}
\newtheorem{caution}[theorem]{Caution}
\newaliascnt{remark}{theorem}
\newtheorem{remark}[remark]{Remark}
\aliascntresetthe{remark}
\newaliascnt{remarkApp}{thmApp}

\aliascntresetthe{remarkApp}
\newaliascnt{reminder}{theorem}
\newtheorem{reminder}[reminder]{Reminder}
\aliascntresetthe{reminder}
\newtheorem{illustration}[theorem]{Illustration}
\newtheorem{observation}[theorem]{Observation}
\newaliascnt{lemma}{theorem}
\newtheorem{lemma}[lemma]{Lemma}
\aliascntresetthe{lemma}
\newaliascnt{lemmaApp}{thmApp}

\aliascntresetthe{lemmaApp}
\newaliascnt{qn}{theorem}
\newtheorem{qn}[qn]{Question}
\aliascntresetthe{qn}
\newaliascnt{conjecture}{theorem}
\newtheorem{conjecture}[conjecture]{Conjecture}
\aliascntresetthe{conjecture}
\newtheorem{discussion}[theorem]{Discussion}
\newaliascnt{corollary}{theorem}
\newtheorem{corollary}[corollary]{Corollary}
\aliascntresetthe{corollary}
\newaliascnt{setting}{theorem}
\newtheorem{setting}[setting]{Setting}
\aliascntresetthe{setting}
\newaliascnt{construction}{theorem}
\newtheorem{construction}[construction]{Construction}
\aliascntresetthe{construction}
\newaliascnt{example}{theorem}
\newtheorem{example}[example]{Example}
\aliascntresetthe{example}
\newtheorem{conclusion}[theorem]{Conclusion}
\newtheorem{sketch}[theorem]{Sketch}
\newtheorem{triviality}[theorem]{Triviality}
\newtheorem{proto}[theorem]{Prototype Quasifibration}
\newtheorem{cauex}[theorem]{Cautionary Example}
\newaliascnt{hypo}{theorem}
\newtheorem{hypo}[hypo]{Hypothesis}
\aliascntresetthe{hypo}
\newtheorem{subth}{ }[theorem]
\newtheorem{case}{Case}[theorem]
\newtheorem{ssubth}{ }[subth]
\newtheorem{facts}[theorem]{Facts}
\newtheorem{history}[theorem]{Historical Survey}
\newtheorem{proofs}[theorem]{Discussion of the Proofs, Old and New}
\newtheorem{heuristic}[theorem]{Heuristic}

\newcommand{\st}{\,:\,} 
\nc\tri[1]{\begin{triviality}
\label{#1}}
\nc\fac[1]{\begin{facts}
\label{#1}
\begin{em}}
\nc\heu[1]{\begin{heuristic}
\label{#1}
\begin{em}}
\nc\app[1]{\begin{apl}
\label{#1}
\begin{em}}
\nc\skt[1]{\begin{sketch}
\label{#1}
\begin{em}}
\nc\hst[1]{\begin{history}
\label{#1}
\begin{em}}
\nc\pfs[1]{\begin{proofs}
\label{#1}
\begin{em}}
\nc\cas[1]{\begin{case}
\label{#1}
\begin{em}}
\nc\rfn[1]{\begin{refinement}
\label{#1}}
\nc\prt[1]{\begin{proto}
\label{#1}}
\nc\lem[1]{\begin{lemma}
\label{#1}}
\nc\pro[1]{\begin{prop}
\label{#1}}
\nc\thm[1]{\begin{theorem}
\label{#1}}
\nc\dis[1]{\begin{discussion}
\label{#1}
\begin{em}}
\nc\cor[1]{\begin{corollary}
\label{#1}}
\nc\dfn[1]{\begin{defin}
\label{#1}}
\nc\sthm[1]{\begin{subth}
\label{#1}}
\nc\exm[1]{\begin{example}
\label{#1}
\begin{em}}
\nc\obs[1]{\begin{observation}
\label{#1}
\begin{em}}
\nc\plm[1]{\begin{prblm}
\label{#1}
\begin{em}}
\nc\rmk[1]{\begin{remark}
\label{#1}
\begin{em}}
\nc\rmd[1]{\begin{reminder}
\label{#1}
\begin{em}}
\nc\ntn[1]{\begin{notation}
\label{#1}
\begin{em}}
\nc\exe[1]{\begin{exercise}
\label{#1}
\begin{em}}
\nc\xpl[1]{\begin{explanation}
\label{#1}
\begin{em}}
\nc\smr[1]{\begin{summary}
\label{#1}
\begin{em}}
\nc\cau[1]{\begin{caution}
\label{#1}
\begin{em}}
\nc\hyp[1]{\begin{hypo}
\label{#1}}
\nc\imn[1]{\begin{importnota}
\label{#1}
\begin{em}}
\nc\rdn[1]{\begin{redu}
\label{#1}
\begin{em}}
\nc\cax[1]{\begin{cauex}
\label{#1}
\begin{em}}
\nc\cmt[1]{\begin{com}
\label{#1}
\begin{em}}
\nc\con[1]{\begin{construction}
\label{#1}
\begin{em}}
\nc\cnj[1]{\begin{conjecture}
\label{#1}
\begin{em}}
\nc\ill[1]{\begin{illustration}
\label{#1}
\begin{em}}
\nc\SSSthm[1]{\begin{ssubth}
\label{#1}
\begin{em}}
\nc\cnc[1]{\begin{conclusion}
\label{#1}
\begin{em}}

\nc\elem{\end{lemma}}
\nc\erdn{\end{em}\end{redu}}
\nc\erfn{\end{refinement}}
\nc\eprt{\end{proto}}
\nc\ethm{\end{theorem}}
\nc\ecor{\end{corollary}}
\nc\edfn{\end{defin}}
\nc\esthm{\end{subth}}
\nc\epro{\end{prop}}
\nc\etri{\end{triviality}}
\nc\eexm{\end{em}
\end{example}}
\nc\eobs{\end{em}
\end{observation}}
\nc\ecmt{\end{em}
\end{com}}
\nc\efac{\end{em}
\end{facts}}
\nc\eheu{\end{em}
\end{heuristic}}
\nc\eapp{\end{em}
\end{apl}}
\nc\ermk{\end{em}
\end{remark}}
\nc\ermd{\end{em}
\end{reminder}}
\nc\eill{\end{em}
\end{illustration}}
\nc\eplm{\end{em}
\end{prblm}}
\nc\ecas{\end{em}
\end{case}}
\nc\eskt{\end{em}
\end{sketch}}
\nc\ecau{\end{em}
\end{caution}}
\nc\ecax{\end{em}
\end{cauex}}
\nc\eimn{\end{em}
\end{importnota}}
\nc\entn{\end{em}
\end{notation}}
\nc\eexe{\end{em}
\end{exercise}}
\nc\expl{\end{em}
\end{explanation}}
\nc\edis{\end{em}
\end{discussion}}
\nc\econ{\end{em}
\end{construction}}
\nc\ecnj{\end{em}
\end{conjecture}}
\nc\esmr{\end{em}
\end{summary}}
\nc\ehst{\end{em}
\end{history}}
\nc\epfs{\end{em}
\end{proofs}}
\nc\ehyp{
\end{hypo}}
\nc\ecnc{\end{em}
\end{conclusion}}
\nc\essthm{\end{em}
\end{ssubth}}

\nc\SSSt{\scriptstyle}
\newcommand{\comment}[1]{}
\newcommand{\nn}{{\mathbb N}}
\newcommand{\NN}{\mathbb{N}}
\newcommand{\ZZ}{\mathbb{Z}}

\newcommand{\K}{{\mathbf K}}

\newcommand{\AAA}{{\mathbf A}}
\newcommand{\BB}{{\mathbf B}}

\newcommand{\DD}{{\mathbf D}}
\newcommand{\TT}{{\mathbf T}}
\newcommand{\SSS}{{\mathbf S}}
\newcommand{\RR}{{\mathbf R}}
\newcommand{\UU}{{\mathbf U}}

\newcommand{\ydg}{\fY_{\mathrm{dg}}}

\newcommand{\D}{{\mathbf D}}
\newcommand{\Map}{\mathrm{Map}}
\newcommand{\Cdg}{\mathbf{C}_{\mathrm{dg}}}
\newcommand{\Aut}{\mathrm{Aut}}
\newcommand{\AutL}{\mathrm{Aut}_\infty}

\newcommand{\Iso}{\mathrm{Iso}}

\newcommand{\longhookrightarrow}{\lhook\joinrel\longrightarrow}
\nc\op{^{\hbox{\rm\tiny op}}}
\nc\mth{^{\hbox{\rm\tiny th}}}
\nc\Enh{\mathbf{Enh}}
\nc\Enhstr{\mathbf{Enh}^\text{str}}

\nc\script{\mathscr}
\nc\ca{{\script A}}
\nc\cb{{\script B}}
\nc\cc{{\script C}}
\nc\cd{{\script D}}
\nc\ce{{\script E}}
\nc\cf{{\script F}}
\nc\cg{{\script G}}
\nc\ch{{\script H}}
\nc\ci{{\script I}}
\nc\cj{{\script J}}
\nc\ck{{\script K}}
\nc\cl{{\script L}}
\nc\cm{{\script M}}
\nc\cn{{\script N}}
\nc\co{{\script O}}
\nc\cp{{\script P}}
\nc\cq{{\script Q}}
\nc\cs{{\script S}}
\nc\ct{{\script T}}
\nc\cu{{\script U}}
\nc\cv{{\script V}}
\nc\cx{{\script X}}
\nc\cy{{\script Y}}
\nc\cz{{\script Z}}
\nc\clim{{\ds\mathop{\rm lim}_{\ds\longleftarrow}}\,}
\nc\climi{\clim_{\!i}\,}
\nc\climn{\clim^{\!n}\,}
\nc\colim{{\ds\mathop{\rm colim}_{\ds\la}}}
\nc\colimj{{\ds\mathop{\rm colim}_{\ds\la}}{}_{j\,}}
\nc\loc[1]{{\text{\rm Loc(#1)}}}
\nc\coloc[1]{{\text{\rm Coloc}(#1)}}

\newcommand{\sh}[2][]{\Sigma^{#1}#2}

\nc\prf{\begin{proof}}
\nc\eprf{\end{proof}}
\nc\ds{\displaystyle}

\nc\be{\begin{roenumerate}}
\nc\ee{\end{roenumerate}}

\nc\sub{\qquad\subset\qquad}
\nc\ctr[1]{{\left.\ct\left(-,#1\right)\right|}_{\cs}}
\nc\ctrf[2]{{\left.\ct\left(#1,#2\right)\right|}_{\cs}}
\nc\Ctr[1]{{\left.\ct\left(-,#1\right)\right|}_{\ct^\alpha}}
\nc\Ctrf[2]{{\left.\ct\left(#1,#2\right)\right|}_{\ct^\alpha}}

\nc\la{\longrightarrow}
\nc\nin{\noindent}
\nc\cad[1]{\text{card}(#1)}
\nc\eq{\quad=\quad}
\nc\BA{\begin{array}{c}}
\nc\EA{\end{array}}
\nc\barr{
\[
\begin{array}{cccccccccccccccc}
}
\nc\earr{
\end{array}
\]
}
\nc\as[1]{{\langle S\rangle}^{#1}}
\nc\shi{\text{\it shift}}

\nc\Hom{{\mathop{\rm Hom}}}
\nc\HHom{{\script H}{\mathop{\rm om}}}
\nc\End{{\mathop{\rm End}}}
\nc\Ext{{\mathop{\rm Ext}}}
\nc\e{\varepsilon}
\nc\p{\varphi}

\nc\y[1]{\mathbf{y}#1}
\nc\x[1]{\mathbf{z}#1}
\rnc\mod[1]{\ensuremath{\mathop{#1\textup{--mod}}}\xspace}
\nc\mmod[1]{\text{{\rm mod}--}#1}
\nc\Mod[1]{#1\text{--{\rm Mod}}}
\nc\MMod[1]{\text{{\rm Mod}--}#1}
\nc\Proj[1]{\text{{\rm Proj}--}#1}
\nc\proj[1]{\text{{\rm proj}--}#1}
\nc\Md{\mathbf{Mod}}
\nc\ov{\overline}
\nc\wt{\widetilde}
\nc\wh{\widehat}
\nc\ph{\varphi}
\nc\tstr{{\it t}--structure}
\nc\tstrs{{\it t}--structure }
\nc\spec[1]{{\text{\rm Spec}\left(#1\right)}}
\nc\gen[2]{{\langle#1\rangle}^{}_{#2}}
\nc\genu[3]{{\langle#1\rangle}^{[#3]}_{#2}}
\nc\ogen[1]{\ov{\langle#1\rangle}}
\nc\ogenun[2]{\ov{\langle#1\rangle}_{#2}^{}}
\nc\ogenu[3]{\ov{\langle#1\rangle}^{[#3]}_{#2}}
\nc\ogenul[3]{\ov{\langle#1\rangle}^{[-\infty,#3]}_{#2}}
\nc\ogenuf[3]{\ov{\langle#1\rangle}^{[#3,\infty]}_{#2}}
\nc\genuf[3]{{\langle#1\rangle}^{[#3,\infty]}_{#2}}
\nc\genul[3]{{\langle#1\rangle}^{[-\infty,#3]}_{#2}}
\nc\dperf[1]{\D^{\mathrm{perf}}(#1)}
\nc\dbcoh{\mathbf{D}^b_{\mathrm{coh}}}
\nc\dbcohs[1]{\mathbf{D}^b_{\mathrm{coh},#1}}
\nc\dmcoh{\mathbf{D}^-_{\mathrm{coh}}}
\nc\dmcohs[1]{\mathbf{D}^-_{\mathrm{coh},#1}}
\newcommand{\Dqc}{{\mathbf D_{\text{\bf qc}}}}
\newcommand{\Dqcmi}{{\mathbf D_{\text{\bf qc}}^-}}
\newcommand{\Dqcpl}{{\mathbf D_{\text{\bf qc}}^+}}
\newcommand{\Dqcb}{{\mathbf D_{\text{\bf qc}}^b}}
\newcommand{\Dqcp}{{\mathbf D_{\text{\bf qc}}^p}}
\newcommand{\Dqcpb}{{\mathbf D_{\text{\bf qc}}^{p,b}}}
\newcommand{\Dqcmis}[1]{{\mathbf D_{\text{\bf qc},#1}^-}}
\newcommand{\Dqcpls}[1]{{\mathbf D_{\text{\bf qc},#1}^+}}
\newcommand{\Dqcbs}[1]{{\mathbf D_{\text{\bf qc},#1}^b}}
\newcommand{\Dqcps}[1]{{\mathbf D_{\text{\bf qc},#1}^p}}
\newcommand{\Dqcpbs}[1]{{\mathbf D_{\text{\bf qc},#1}^{p,b}}}
\nc\RHHom{{\script{RH}}{\mathrm{om}}}
\nc\Coprod{\mathrm{Coprod}}
\nc\COprod{\mathrm{coprod}}
\nc\add{\mathrm{add}}
\nc\Add{\mathrm{Add}}
\nc\Smr{\mathrm{smd}}
\nc\id{\mathrm{id}}
\nc\LL{\mathbf{L}}
\nc\R{\mathbf{R}}
\nc\wi{\wt{\text{\it\i}}}
\nc\exal{\ce\text{\it x}(\ct^\alpha,\ab)}
\nc\exalz{\ce\text{\it x}_{\aleph_0}^{}(\ct^\alpha,\ab)}
\nc\fc{\mathfrak{C}}
\nc\fl{\mathfrak{L}}
\nc\fs{\mathfrak{S}}
\nc\Prf{\text{\bf Perf}}
\nc\Enhq{\mathbf{Enh}_?^{}}
\nc\Enhqstr{\mathbf{Enh}_?^\text{str}}
\nc\Enhdg{\mathbf{Enh}_{\mathbf{dg}}^{}}
\nc\Enhstrdg{\mathbf{Enh}_{\mathbf{dg}}^\text{str}}
\nc\Enhin{\mathbf{Enh}_\infty^{}}
\nc\EnhGL{\mathbf{Enh}_{?}^{}}
\nc\Tri{\mathbf{Tri}}
\nc\fgt{\mathsf{Fgt}}
\newcommand{\Dqcs}[1]{{\mathbf D_{\text{\bf qc},#1}}}
\nc\coh{\mathbf{Coh}}
\nc\qc{\mathbf{Qcoh}}
\nc\vect{\mathbf{Vect}}
\nc\dperfs[2]{\D_{#1}^{\mathrm{perf}}(#2)}
\nc\SB{^{\mathrm{sb}}}
\nc\SBb{^{\mathrm{sb},b}}
\nc\tsb{\ct\SB}
\nc\tsbb{\ct\SBb}
\nc\climone{{\ds{\mathop{\rm lim}_{\ds\longleftarrow}}}^1\,}

\newcommand{\fun}[1]{\mathsf{#1}}

\newcommand{\fF}{\fun{F}}
\newcommand{\fG}{\fun{G}}

\newcommand{\fY}{\fun{Y}}
\newcommand{\kK}{\mathbb{K}}

\newcommand{\Dqca}{{\mathbf D_{\text{\bf qc}}^\star}}

\newcommand{\scat}[1]{{\mathbf{#1}}} 
\newcommand{\dgCat}{\scat{dgCat}} 
\newcommand{\Hqe}{\scat{Hqe}} 
\newcommand{\dgMod}[1]{\mathbf{dgMod}\text{--}#1}
\newcommand{\opp}{^{\circ}} 

\newcommand{\PrCat}{\mathbf{PrCat}}
\newcommand{\stable}{\mathrm{st}}
\newcommand{\Dloc}{\mathbf{D}^\mathrm{loc}}

\nc\hoco{
\begin{picture}(40,10)
\put(20,0){\makebox(0,0)[b]{\text{\rm Hocolim}}}
\put(5,-2){\vector(1,0){30}}
\end{picture}\,\,}

\renewcommand{\ge}{\geqslant}
\renewcommand{\leq}{\leqslant}
\renewcommand{\geq}{\geqslant}
\nc\tst[1]{\left({#1}^{\leq0},{#1}^{\geq1}\right)}
\nc\tstv[2]{\left({#1}_{#2}^{\leq0},{#1}_{#2}^{\geq1}\right)}
\nc\tsth[2]{{#1}_{#2}^{\heartsuit}}

\nc\holim{
\begin{picture}(40,10)
\put(20,0){\makebox(0,0)[b]{\text{\rm Holim}}}
\put(35,-2){\vector(-1,0){30}}
\end{picture}}

\newcommand{\lto}{\longrightarrow}
\newcommand{\OO}{\mathcal{O}}
\newcommand{\ho}[1]{\mathrm{Ho}(#1)} 
\newcommand{\Yon}[1][]{\fY_{#1}}
\newcommand{\YonInf}{\fY_{\infty}}
\newcommand{\YYon}[1][]{\widetilde\fY_{#1}}
\newcommand{\YYonInf}{\widetilde\fY_{\infty}}
\newcommand{\Hocolim}{\mathrm{Hocolim}}
\newcommand{\IndC}{\mathrm{Ind}_?}
\newcommand{\IndNL}{\mathrm{Ind}}
\newcommand{\IndL}{\mathrm{Ind}_\kK}
\newcommand{\ResNL}{\mathrm{Res}}
\newcommand{\ResL}{\mathrm{Res}_\kK}

\newcommand{\Res}{\fun{Res}}
\newcommand{\Comp}{\fun{Comp}}

\newcommand{\hproj}[1]{{\mathbf{h}\text{-}\mathbf{proj}}\text{--}#1} 

\newcommand{\CatI}{\mathbf{Cat}}
\newcommand{\Hc}{\mathrm{H}^0}
\newcommand{\StabInfNL}{{\CatI}_\infty^{\mathrm{st}}}
\newcommand{\StabInf}{{\CatI}_{\infty,\kK}^{\mathrm{st}}}
\newcommand{\StabInfA}{{\CatI}_{\infty,?}^{\mathrm{st}}}
\newcommand{\HoStabInfNL}{\mathrm{Ho}\left(\StabInfNL\right)}
\newcommand{\HoStabInf}{\mathrm{Ho}\left(\StabInf\right)}
\newcommand{\HoStabInfA}{\mathrm{Ho}\left(\StabInfA\right)}
\nc\EnhNL{\mathbf{Enh}}
\nc\EnhL{\mathbf{Enh}_{\kK}}
\nc\EnhA{\mathbf{Enh}_{?}}
\nc\EnhStrNL{\mathbf{Enh}^\mathrm{str}}
\nc\EnhStrL{\mathbf{Enh}_{\kK}^\mathrm{str}}
\nc\EnhStrA{\mathbf{Enh}_{?}^\mathrm{str}}

\newcommand{\Ho}[1]{\mathrm{Ho}(#1)} 
\newcommand{\Spe}{\mathbf{Sp}} 
\newcommand{\rest}[1]{|_{#1}}

\begin{document}

\author{Alberto Canonaco, Amnon Neeman, and Paolo Stellari}

\address{A.C.: Dipartimento di Matematica ``F.\ Casorati''\\
        Universit{\`a} degli Studi di Pavia\\
        Via Ferrata 5\\
        27100 Pavia\\
        ITALY}
\email{alberto.canonaco@unipv.it}

\address{A.N.: Dipartimento di Matematica ``F.\ Enriques''\\
        Universit{\`a} degli Studi di Milano\\
        Via Cesare Saldini 50\\
	20133 Milano\\
        ITALY}
\email{amnon.neeman@unimi.it}

\address{P.S.: Dipartimento di Matematica ``F.\ Enriques''\\
        Universit{\`a} degli Studi di Milano\\
        Via Cesare Saldini 50\\
	20133 Milano\\
        ITALY}
\email{paolo.stellari@unimi.it}
\urladdr{\url{https://sites.unimi.it/stellari}}

 \thanks{A.~C.~is a member of GNSAGA (INdAM) and was partially supported by the research project PRIN 2022 ``Moduli spaces and special varieties''. A.~N.~was partly supported 
   by Australian Research Council Grants DP200102537 and DP210103397,
   and by ERC Advanced Grant 101095900-TriCatApp.
    	P.~S.~was partially supported by the ERC Consolidator Grant ERC-2017-CoG-771507-StabCondEn, by the research project PRIN 2022 ``Moduli spaces and special varieties'', and by the research project FARE 2018 HighCaSt (grant number R18YA3ESPJ)}

 \title[Metrics and enhancements]
       {Metrics on triangulated categories and their enhancements}

\begin{abstract}
In this paper we investigate the uniqueness of enhancements of the natural subcategories of weakly approximable triangulated categories. The main idea is to enhance at the level of $\infty$-categories the recently developed theory of excellent metrics. The applications of our results include a vast generalization of the known results about the (strong) uniqueness of enhancements in the linear and nonlinear setting, providing positive answers to some open questions. In addition we prove that, under some natural assumptions, the equivalences between such subcategories can be lifted through their natural inclusions. This completes the picture started in our previous paper \cite{Canonaco-Neeman-Stellari24} and extends the known results about Margolis Uniqueness Conjecture for the homotopy category of spectra.
\end{abstract}



\maketitle

\setcounter{tocdepth}{1}
\tableofcontents

\setcounter{section}{0}

\section*{Introduction}
\label{S973}

It is well-known that,
in the world of triangulated categories,
the mapping cone of a morphism is not
canonically unique or functorial.
And for certain constructions this creates
headaches---let us spare the reader the
customary
polemics surrounding this, and a
recounting of problems that arise. The issue
has by now been well documented in the
literature. Suffice it to say that the experts
have come to agree that there is great value in
studying the so-called ``enhancements'' of
triangulated categories. See
\autoref{def:algtria}, and more generally
\autoref{subsec:generalmod} and
\autoref{subsec:uniqenh} for the basics about
enhancements and their possible uniqueness.

In this paper we focus on two concrete
problems, in which we feel free to use either
triangulated categories or their enhanced versions,
as the need arises.
For the moment we formulate the problems
vaguely:
\begin{enumerate}
\item[(a)] How do we canonically detect special triangulated subcategories of a given one?
\item[(b)] How do we reconstruct larger triangulated (sub)categories out of a given smaller one?
\end{enumerate}

Let us illustrate this with a well-known, classical example. Let $\D(\MMod{R})$ be the derived category of complexes of right modules over a ring $R$, and let $\dperf{R}$ be the full triangulated subcategory of $\D(\MMod{R})$ of perfect complexes. Recall that the category $\dperf{R}$ is triangle-equivalent to the homotopy category $\K^b(\proj R)$ of bounded complexes of finitely generated projective $R$-modules. More accurately, the natural functor $\K^b(\proj R)\la\D(\MMod{R})$ is fully faithful, and the essential image is $\dperf{R}$. And Question (a) has a simple answer in this case: $\dperf{R}$ coincides with the full, triangulated subcategory of compact objects in $\D(\MMod{R})$. Compact objects will be formally introduced in
\autoref{definitionofcompacts}, for now it suffices to note that an object $C\in\D(\MMod{R})$ is compact if $\Hom(C,-)$ behaves well, meaning it respects certain colimits. Thus we have moved from a purely algebraic description of $\dperf{R}\iso\K^b(\proj R)$, in terms of projective modules, to a more categorical presentation which is intrinsic in nature. More concretely, the categorical description makes it easy to see that if $S$ is another ring, then any equivalence $\D(\MMod{R})\iso\D(\MMod{S})$ restricts to an equivalence of subcategories
$\dperf{R}\iso\dperf{S}$.

Somewhat surprisingly question (b) can be recast, and the reformulated problem has a solution in this specific situation. The recasting requires us to employ enhancements. Let us view $\D(\MMod R)$ not only as a triangulated category, but as the homotopy category of a stable $\infty$-category $\TT$. In the language of
\autoref{def:algtria}, we choose an enhancement
$\TT$ of the triangulated category $\D(\MMod R)$.
Then the triangulated subcategory $\dperf{R}$, of
the triangulated category $\D(\MMod R)\iso\Hc(\TT)$,
defines for us a natural stable $\infty$-subcategory $\SSS\subset\TT$, with $\Hc(\SSS)\iso\dperf{R}$. And a classical observation shows that $\TT\iso\IndNL(\SSS)$, meaning $\TT$ is the Ind-completion
of $\SSS$. All in all, not only is there a recipe constructing from $\D(\MMod{R})$ the subcategory $\dperf{R}$ , but there is also a recipe going back, which out of $\dperf{R}$ reconstructs $\D(\MMod{R})$. The issue is that the process going back appeals to enhancements. 

So far we have treated classical, well-known results.
And now the time has come for the vast generalizations
and improvements addressed in the current paper.

Still staying with the example of the triangulated category $\D(\MMod{R})$: it carries additional structure, namely a \tstr. And this allows us to define plenty of new triangulated subcategories, for example $\D^\star(\MMod{R})$, with $\star=b,+,-$. And when $R$ is a right coherent ring we can also define $\D^b(\mmod{R})$, the bounded derived category of finitely presented $R$-modules. And when we ask (a) and (b) for these new categories, we run into much deeper and more interesting problems.

The first main point of this paper, and of the precursor \cite{Canonaco-Neeman-Stellari24}, is that these problems can be vastly generalized. They can be formulated in broad generality in terms of the theory of weakly approximable triangulated categories of \autoref{def:wa}, developed by the second author during the last few years. There will be a review of the results we need in \autoref{sec:wa}. The short summary is that a triangulated category $\ct$, which is closed under small coproducts, is \emph{weakly approximable} if it has a compact generator $G$ and a \tstrs $\tst\ct$ satisfying a list of technical hypotheses. The easy ones to state are
that
$G\in\ct^-$ is bounded above, and there are no
nonzero morphisms from $G$ to sufficiently positive shifts of $G$. The more technical condition, for which we do not yet have the notation, asserts roughly that any object in
the heart $\tsth\ct{}$ admits a surjection from
an object in the heart of the form $E^{\geq0}$, with
$E\in\ct^{\leq0}$ generated by $G$ in bounded effort.
The definition of ``bounded effort'' is the tricky part, it will be reviewed at the beginning of
\autoref{subsec:gen}.

If $\ct$ is a weakly approximable triangulated category, then it comes with a given \tstr, and in terms of this \tstrs we can consider the objects which are bounded above, bounded below or bounded on both sides. The full subcategories of these objects are, respectively, the thick triangulated subcategories $\ct^-$, $\ct^+$ and $\ct^b$. In addition we have the full triangulated subcategory $\ct^c$ of compact objects and, by taking approximation by compact objects, we can define the full subcategory $\ct^b_c$ of bounded pseudo-compact objects. The \tstrs $\tst\ct$ is also crucial in the formation
of the auxiliary triangulated subcategory $\tsb$ of \autoref{D0.7}, which will remain mysterious for now but will play an important role in the body of the paper.

These subcategories will all be properly defined in \autoref{subsec:wadef} and \autoref{subsec:waTsb}. Leaving aside the formal definition for now, we can get a grasp on them by considering the example of the geometric setting. If $X$ is a noetherian scheme and $\ct=\Dqc(X)$ is the derived category of complexes of $\OO_X$-modules with quasi-coherent cohomology, then $\ct$ is weakly approximable and the subcategories
$\ct^-$, $\ct^+$, $\ct^b$, $\ct^c$ and $\ct^b_c$ above turn out to be $\Dqcmi(X)$, $\Dqcpl(X)$, $\Dqcb(X)$, $\dperf{X}$ and $\D^b(\coh(X))$.

If we add to the list the full triangulated subcategory $\ct^-_c$ of pseudo-coherent objects, as well as $\tsbb:=\tsb\cap\ct^b$ and $\ct^{c,b}:=\ct^c\cap\ct^b$, we obtain the following diagram of natural inclusions:
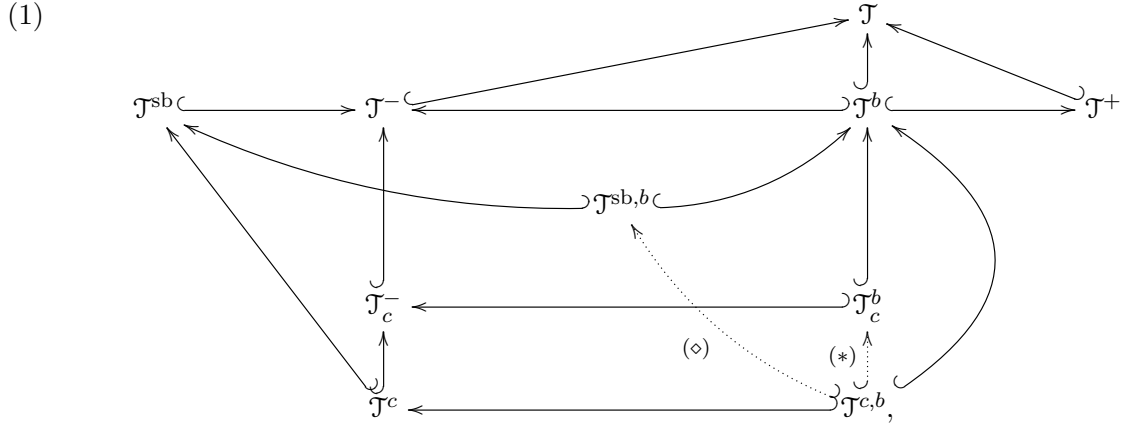
\begin{equation}\label{eq:incl}
\xymatrix@C+40pt@R-5pt{
&&&\ct&\\
\tsb\ar@{^{(}->}[r]&\ct^-\ar@{^{(}->}[urr]&&\ct^b\ar@{^{(}->}[u]\ar@{_{(}->}[ll]\ar@{^{(}->}[r]&\ct^+\ar@{_{(}->}[ul]\\
&&\tsbb\ar@/_1pc/@{^{(}->}[ur]\ar@/^1pc/@{_{(}->}[ull]&&\\
&\ct^-_c\ar@{^{(}->}[uu]&&\ct^b_c\ar@{_{(}->}[ll]\ar@{^{(}->}[uu]&\\
&\ct^c\ar@{^{(}->}[u]\ar@{_{(}->}[uuul]&&\ct^{c,b},\ar@{^{(}.>}[u]^-{(\ast)}\ar@{_{(}->}[ll]
 \ar@/_4pc/@{^{(}->}[uuu]\ar@/^1pc/@{_{(}.>}[luu]^-{(\diamond)}&
}
\end{equation}
for any weakly approximable triangulated category $\ct$.

\subsection*{The results}

Let us now illustrate the more general versions of our main results and let us postpone a little the discussion about their applications.

The first result is about question (a) above and it allows us to intrinsically recover triangulated subcategories of a given one in \eqref{eq:incl}.

\begin{thmInt}[\autoref{thm:passage}]\label{thm:main1}
If $\ct$ is a weakly approximable triangulated category, then all the inclusions in diagram \eqref{eq:incl}, represented by the solid arrows $\ca\hookrightarrow\cb$, are \emph{invariant under triangle equivalences}. By this we mean that, given a pair of weakly approximable triangulated
categories $\ct,\ct'$, as well as matching inclusions
$\ca\hookrightarrow\cb\hookrightarrow\ct$ and
$\ca'\hookrightarrow\cb'\hookrightarrow\ct'$ from diagram \eqref{eq:incl},
then any triangle equivalence $\cb\la\cb'$ must restrict to a triangle equivalence $\ca\la\ca'$.

The same is true for the inclusion $(\ast)$ in the diagram \eqref{eq:incl}, provided we further assume that one of the two conditions below holds:
\be
\item $\ct,\ct'$ are coherent or
\item $\ct^c\subset\ct^b_c$, ${\ct'}^c\subset{\ct'}^b_c$ and
$^\perp(\ct^b_c)\cap\ct^-_c=\{0\}={^\perp({\ct'}^b_c)}\cap{\ct'}^-_c$.
\ee

Finally, the inclusion $(\diamond)$ is invariant under triangle equivalences if $\ct^c\subset\ct^b_c$ and ${\ct'}^c\subset{\ct'}^b_c$.
\end{thmInt}

The technical definition of coherent weakly approximable triangulated categories is as in \cite[Definition~\ref*{D29.1}]{Neeman18A}, and for the reader's convenience it will be recalled in \autoref{def:cohtria}. For now it suffices to say that it asserts the existence, for every pseudo-compact object, of a triangle with properties somewhat less stringent than those of the truncation triangle with respect to a \tstr. It is worth pointing out that if $R$ is a right coherent ring, then $\D(\MMod{R})$ is coherent.

Most of \autoref{thm:main1} was proved in \cite{Canonaco-Neeman-Stellari24}. The present statement is an improvement in that it also includes the two new subcategories, $\tsb$ and $\tsbb$, which play an important role in this paper. Some major applications of \autoref{thm:main1}, including a vast generalization of \cite{Rickard89b} have been discussed in \cite{Canonaco-Neeman-Stellari24} and then surveyed in \cite{CNS25}. The focus of the current paper is on the relevance of \autoref{thm:main1} to providing an answer to (b).

\medskip

As we move towards answering (b), we need to make more precise what it means to reconstruct a larger triangulated category $\cb$ out of a smaller one $\ca\subset\cb$. The algebraic example we discussed, in the two paragraphs immediately following (a) and (b) above, illustrates that we understand how to perform (b) only in the presence of enhancements. It has long been conjectured that, at least in some cases, there should be an enhancement-free way to do this; we will recall one venerable, old example in \autoref{conj:margolis}. Four and a half decades after this conjecture was formulated, it remains wide open. In recent years there has been some potential progress, with the discovery of some enhancement-free algorithms, constructing new triangulated categories out of old ones. But so far none of these apply to \autoref{conj:margolis}. For the purpose of the current article we take a different tack; we prove that in many interesting cases the category $\ca$ has a unique enhancement and then, using this unique enhancement, we reconstruct $\cb$.
Assume therefore that both $\ca$ and $\cb$ admits $\infty$-categorical models, that is $\ca$ and $\cb$ both admit enhancements. And then the higher categorical version of (b) can be divided into two questions:
\begin{enumerate}
\item[(b.1)] How unique is the $\infty$-categorical model of $\cb$, given the uniqueness of the model for $\ca$?
\item[(b.2)] If in the diagram \eqref{eq:incl} we let $\ca\subset\cb$ and $\ca'\subset\cb'$ be matching pairs, of full triangulated subcategories of two weakly approximable triangulated categories $\ct$ and $\ct'$, does the existence of a triangle equivalence $\ca\iso\ca'$ imply the existence of a triangle equivalence $\cb\iso\cb'$?
\end{enumerate}

The example about rings discussed above generalizes, and the argument can easily be modified to answer (b.1) and (b.2) when $\ca=\ct^c$ and $\cb=\ct$, as long as $\cb$ has an enhancement and the enhancement of $\ca$ is unique. This is because $\cb$, in this case, has an intrinsic higher categorical model which can be described purely in terms of a given (unique) model for $\ct^c$. But what about the other natural triangulated subcategories $\ca\longhookrightarrow\cb$ in \eqref{eq:incl}? Can we perform the same intrinsic reconstruction of $\cb$ starting from $\ca$?

Here is where the theory of weakly approximable triangulated categories and the systematic use of metrics and completions kicks in. For the reader's convenience the notion of a good metric on a triangulated category is recalled in \autoref{def:metric}. This theory was developed primarily in the articles \cite{Neeman25,Neeman18A}, and back in those manuscripts the main purpose was to provide an enhancement-free algorithm that constructs new triangulated categories out of old ones. The recipe takes a triangulated category $\cs$ with a good metric $\{\cm_i\}$ to a triangulated category $\fs(\cs)$ with a good metric $\{\cn_i\}$. And the relevance, to the world of weakly approximable triangulated categories, is that in each of the pairs $(\ct^c,\ct^b_c)$ and $(\ct^b,\tsb)$ we have that one category is obtainable from the other by this algorithm, for suitable choices of metrics.

Let us recall that the triangulated category $\cs$ always has the classical Yoneda embedding $\Yon:\cs\la\MMod\cs$, which takes an object $C\in\cs$ to the functor $\Hom(-,C)$ from the opposite category $\cs\op$ to the category of abelian groups. And the construction of $\fs(\cs)$ uses nothing more then the metric on $\cs$, in combination with this classical Yoneda map. But it can be helpful in computations to have an auxiliary triangulated category $\ct$, as long as it is a good extension with respect to the metric. A good extension is a fully faithful, triangulated functor $\fF\colon\cs\la\ct$ satisfying properties we recall in \autoref{def;goodext} and
\autoref{explanationofgoodextension}. And this is where enhancements enter, at least in
the current article. If $\cs$ has an enhancement
$\SSS$, then the Ind-construction $\IndNL(\SSS)$ associated to $\SSS$ and the corresponding Yoneda functor which embeds $\SSS$ into $\IndNL(\SSS)$ \emph{always} provides a good extension, with respect
to any metric on $\cs\iso\Hc(\SSS)$. Then the issue becomes to compare good extensions. One of the key technical results in the article, \autoref{L1.5}, does exactly that. Roughly it tells us that the completion $\fl'(\cs)$ of $\cs$ in $\ct$, with respect to the metric $\{\cm_i\}$, is often independent of the choice of a good extension $\fF\colon\cs\la\ct$. One immediate application is when we have an
equality $\ct=\fl'(\cs)$; as an almost
immediate corollary of \autoref{L1.5} we obtain
\autoref{P15.5}, which asserts that if the
good extension
$\fF\colon\cs\la\ct$ is such that $\fl'(\cs)=\ct$,
and if $\ct$ has an enhancement, then this enhancement
must agree with the one given by the full subcategory of $\IndNL(\SSS)$ whose objects are those in the completion of $\cs$ in 
$\Hc(\IndNL(\SSS))$. Thus, if $\cs$ has a unique enhancement and $\ct$ has an enhancement, then the enhancement of $\ct$ is unique. 

In summary, using this combination of metrics on the triangulated categories in conjunction with enhancements, we can provide a positive answer to both (b.1) and (b.2) in gorgeous generality. Let us start with the general result answering (b.1):

\begin{thmInt}[\autoref{prop:applTwa}]\label{thm:main2}
Let $\ct$ be a weakly approximable triangulated category satisfying one of the following two assumptions:
\begin{enumerate}
\item[{\rm (a)}] either $\ct$ is coherent;
\item[{\rm (b)}] or $\ct^c\subset\ct^b_c$.
\end{enumerate}
Let $\ca\subset\cb\subset\ct$ be two of the natural subcategories in the diagram \eqref{eq:incl}, such that
$\cb\neq\ct_c^-$ and, in the case where {\rm (a)} holds but {\rm (b)} does not,
the categories $\ca,\cb\in\{\ct^{c,b},\tsbb\}$ are also excluded. Assume further that
\begin{enumerate}
\item[{\rm (i)}] $\cb$ (and thus $\ca$) has an enhancement in $\HoStabInfA$.
\item[{\rm (ii)}] $\ca$ has a unique enhancement in $\HoStabInfA$.
\item[{\rm (iii)}]
In the case where $\ca=\ct^b_c$ and
{\rm (b)} holds, we
add the assumption that 
${}^\perp(\ct^b_c)\cap\ct^-_c=\{0\}$.
\end{enumerate}
Then $\cb$ has a unique enhancement in $\HoStabInfA$.

Moreover, there is a formula
for $\cb$ and its (unique) enhancement in terms of $\ca$ and its (unique) enhancement. In particular, given a triangulated category $\ca$ with a unique enhancement, there exists at most one unique $\cb$, containing $\ca$ and possessing an enhancement, for which there might exist a weakly approximable triangulated category $\ct$, with $\ca\subset\cb\subset\ct$ being a prescribed pair of the natural subcategories in the diagram \eqref{eq:incl},
and in the case where $\ca=\ct^b_c$
the ambient $\ct$ is assumed to satisfy {\rm(iii)}.
\end{thmInt}

We will review the definition of $\HoStabInfA$ in \autoref{subsec:generalmod}. Here we just point out that $\StabInfA$ stands for either $\StabInfNL$ or $\StabInf$, the former being the category of stable $\infty$-categories and the latter the category of linear stable $\infty$-categories. Then $\HoStabInfA$ denotes a suitable localization of $\StabInfA$, where we formally invert equivalences. It is the natural place where $\infty$-categorical models of a triangulated category can be compared. For readers more familiar with the language of dg categories, it may be useful to recall that $\HoStabInf$ is equivalent to the corresponding localization of the category of pretriangulated dg categories.

In \autoref{cor:finTwa} we will see that, in the special case where
$\ct^c$ has a unique enhancement, \autoref{thm:main2} is both simple and useful: it tells us that any $\cb$ belonging to $\{\ct^b_c,\tsb,\ct^b,\ct^-,\ct^+,\ct\}$, which has an enhancement, has a unique enhancement.

Our answer to (b.2) comes from the following very general result:

\begin{thmInt}[\autoref{prop:applextTwa}]\label{thm:main3}
Let $\ct_1$ and $\ct_2$ be a pair of
weakly approximable triangulated categories, satisfying one of the following two assumptions for $i=1,2$:
\begin{enumerate}
\item[{\rm (a)}] either $\ct_i$ is coherent,
\item[{\rm (b)}] or $\ct_i^c\subset(\ct_i)^b_c$.
\end{enumerate}
For $i=1,2$, let $\ca_i\subset\cb_i\subset\ct_i$ be matching inclusions of subcategories in the diagram \eqref{eq:incl}, 
such that $\cb_i\ne(\ct_i)^-_c$, and where 
the categories $\ca_i,\cb_i\in\{\ct_i^{c,b},\tsbb_i\}$ are also excluded. Assume further that
\begin{enumerate}
\item[\rm (i)] $\cb_i$ (and thus $\ca_i$) has an enhancement in $\HoStabInfA$, for $i=1,2$.
\item[\rm (ii)] $\ca_i$ has a unique enhancement
in $\HoStabInfA$, for $i=1,2$.
\item[{\rm (iii)}]
In the case where $\ca_i=(\ct_i)^b_c$,
we assume that either  
{\rm (a)} holds for both
$\ct_1$ and $\ct_2$, or
else that {\rm (b)} holds for both.
And when both $\ct_1$ and $\ct_2$
satisfy {\rm (b)} we furthermore
suppose that
${}^\perp(\ct_i)^b_c\cap(\ct_i)^-_c=\{0\}$,
for $i=1,2$.
\end{enumerate}
If $\ca_1\iso\ca_2$, then $\cb_1\iso\cb_2$.
\end{thmInt}

Before moving on to applications, we should say a tiny bit about the ideas that enter into the proofs of our three main results. \autoref{thm:main1} is a purely triangulated statement: the underlying idea is to use the full power of the theory of weakly approximable triangulated categories and compactly generated \tstr s to characterize the natural subcategories. The new results require the development of new techniques, some of which we saw in \cite{Canonaco-Neeman-Stellari24} and more will come in the first part of the present paper. But it all happens strictly in the world of triangulated categories.

As we have already explained, \autoref{thm:main2} and \autoref{thm:main3} depend crucially on the machinery of stable $\infty$-categories. We want to embed stable $\infty$-categories into new and bigger stable $\infty$-categories, and the basic tool is the Yoneda map---somehow the key to everything is to look at suitable versions of restricted Yoneda maps. Of course we also combine these Yoneda embeddings creatively with well-chosen metrics. And if the metrics are sufficiently nice (very good or excellent, as in \autoref{def:excmetrics}) then out come stable $\infty$-categories, which we then work to identify with the enhancements we are after.

\subsection*{Applications}

The previous result are abstract, and to show that they are useful it helps to explain how they can be applied to classical examples.

\subsubsection*{(Strong) uniqueness of enhancements}

The first important applications are about uniqueness of enhancements for geometric categories. We begin with questions raised by others, for example
Question 8.16 (v) and (vi) in \cite{Antieau18}:

\medskip

\noindent{\bf Question} (Antieau) \emph{Let $X$ be a quasi-compact and quasi-separated scheme. 
Do the triangulated categories $\Dqc(X)$ and $\dperf{X}$ have unique enhancements in $\HoStabInfA$?} 

\medskip

The result is known for enhancements in $\HoStabInf$, for any commutative ring $\kK$; see Theorem B in \cite{Canonaco-Neeman-Stellari21}. In this article we show that the answer to Antieau's question is positive in full generality. At the same we time extend the scope, proving the following as an application of \autoref{thm:main2}:
\begin{itemize}
\smallskip
\item (\autoref{cor:geomappli}) Let $X$ be a quasi-compact and quasi-separated scheme. Then all the triangulated subcategories of $\Dqc(X)$
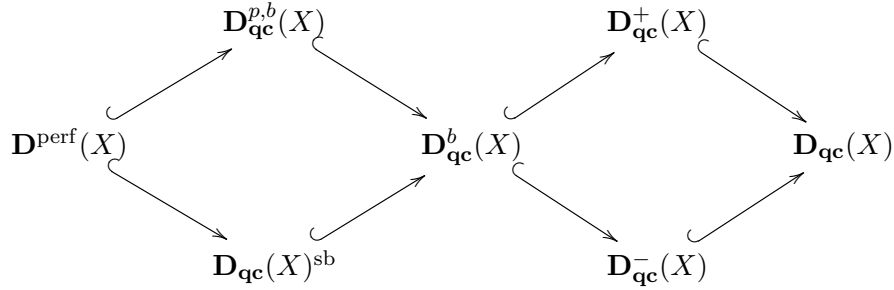
\[
\xymatrix{
&\Dqcpb(X)\ar@{^{(}->}[rd]&&\Dqcpl(X)\ar@{^{(}->}[dr]&\\
\dperf{X}\ar@{^{(}->}[rd]\ar@{^{(}->}[ru]&&\Dqcb(X)\ar@{^{(}->}[ur]\ar@{^{(}->}[dr]&&\Dqc(X)\\
&\Dqc(X)\SB\ar@{^{(}->}[ru]&&\Dqcmi(X)\ar@{^{(}->}[ur]&
}
\]
have unique enhancements in $\HoStabInfA$.
\smallskip
\end{itemize}

The category $\Dqcpb(X)$, which might look mysterious to the nonexpert, becomes familiar when we restrict $X$ to be noetherian. If $X$ is noetherian then $\D^b(\coh(X))\iso\Dqcpb(X)$; that is the category $\Dqcpb(X)$ may be viewed as the right generalization of $\D^b(\coh(X))$ when we want to allow $X$ to be non-noetherian. It should be noted that (A) not only improves all known results in the literature \cite{Antieau18,Canonaco-Neeman-Stellari21,Canonaco-Stellari18,Lunts-Orlov10}, but also deals with uniqueness of nonlinear enhancements. As we observe in \autoref{rmk:uniqueaffine}, the same result holds for the triangulated categories naturally associated to a ring $R$.

Back to the geometric applications, we can look at the full triangulated subcategory $\Dqcs Z(X)$ of $\Dqc(X)$, consisting of complexes with support on $Z$. Here $X$ is still assumed to be quasi-compact and quasi-separated, while $Z\subset X$ is a closed subset with quasi-compact complement in $X$. The uniqueness of enhancements for $\Dqcs Z(X)$ and its natural subcategories has been explored in \cite{Canonaco-Stellari14}, but only in the special case where $Z$ is projective and some further restrictions apply. Here we can prove the following general result, further extending the positive answer to Antieau's question.
\begin{itemize}
\smallskip
\item (\autoref{cor:geomappli}) Let $X$ be a noetherian scheme and let $Z\subset X$ be a closed subset. Then all the natural triangulated categories in the diagram
\[
\xymatrix{
& \Dqcs Z(X)\SB \ar@{^{(}->}[r] & \Dqcbs Z(X) \ar@{^{(}->}[r] \ar@{^{(}->}[dd] & \Dqcpls Z(X) \ar@{^{(}->}[dd] \\
\dperfs{Z}X \ar@{^{(}->}[r] \ar@{^{(}->}[ru] & \dbcohs Z(X) \ar@{^{(}->}[ru] \ar@{^{(}->}[d] & & \\
& \dmcohs Z(X) \ar@{^{(}->}[r] & \Dqcmis Z(X) \ar@{^{(}->}[r] & \Dqcs Z(X)
}
\]
have unique enhancements in $\HoStabInfA$.
\smallskip
\end{itemize}

There is a third triangulated category of interest for us: the homotopy category of spectra $\ho{\Spe}$. The uniqueness of enhancements for this category and its subcategory $\Ho\Spe^c$ of compact objects was investigated back in \cite{Schwede07}. Here we provide a much fuller treatment:
\begin{itemize}
\smallskip
\item (\autoref{cor:spectrappli})
Of the natural triangulated subcategories of $\Ho{\Spe}$ in the diagram \eqref{eq:incl},
we have that $\Ho\Spe^{c,b}$, $\Ho\Spe^c$,
$\Ho\Spe^b_c$, $\Ho\Spe\SB$, $\Ho\Spe^b$, $\Ho\Spe^+$,
$\Ho\Spe^-$ and $\Ho\Spe$ 
all have unique enhancements in $\HoStabInfNL$.
\smallskip
\end{itemize}

From our perspective this is an interesting example, since $\ho{\Spe}$ does not have stable $\infty$-categorical models which are linear over a commutative ring.

We have so far discussed uniqueness of enhancements, but there is a finer notion called strong uniqueness; see \autoref{def:struniqenh}.
Despite considerable effort over the years, our understanding of strong uniqueness remains poor. This is not for the lack of effort, the reader can consult for example
\cite{Canonaco-Stellari14,Canonaco-Stellari07, Lunts-Orlov10,OlanderThesis, Orlov97} for a sample of the literature. Here the major contribution of the current article is
\begin{itemize}
\smallskip
\item (\autoref{lem:struniqgen} and \autoref{cor:struniqexcomp})
Let $\ct$ be a weakly approximable and coherent triangulated category. Then $\ct^c$ has a strongly unique enhancement if and only if $\ct^b_c$ does.
\smallskip
\end{itemize}

In \cite[Section 5.A]{LorenzinThesis}
it is shown that,
if $X$ is a scheme proper over a field $\kK$, with depth $\geq 1$ at every closed point, then $\D^b(\coh(X))$ has a strongly unique enhancement in $\HoStabInf$. From (D) it follows that $\dperf{X}$ also has a strongly unique enhancement in $\HoStabInf$.

\subsubsection*{Morita theory for schemes}

Turning to our third main result \autoref{thm:main3}, a consequence is the following.

\begin{thmInt}[\autoref{cor:D}]\label{thm:main4}
Let $X$ and $X'$ be quasi-compact and quasi-separated schemes with closed subschemes $Z\subset X$ and $Z'\subset X'$ such that $X\setminus Z$ and $X'\setminus Z'$ are quasi-compact.
Assume further that one of the two conditions below
holds:
\be
\item
Either $Z=X$ or $Z'=X'$.
\item
One of $X$ or $X'$ is noetherian.
\ee
Then the following are equivalent:
\begin{enumerate}[{\rm (1)}]
\item There exists a triangle equivalence $\Dqcs Z(X)\iso\Dqcs{Z'}(X')$.
\item There exists a triangle equivalence $\Dqcpls Z(X)\iso\Dqcpls{Z'}(X')$.
\item There exists a triangle equivalence $\Dqcmis Z(X)\iso\Dqcmis{Z'}(X')$.
\item There exists a triangle equivalence $\Dqcbs Z(X)\iso\Dqcbs{Z'}(X')$.
\item There exists a triangle equivalence $\Dqcps Z(X)\iso\Dqcps{Z'}(X')$.
\item There exists a triangle equivalence $\Dqcpbs Z(X)\iso\Dqcpbs{Z'}(X')$.
\item There exists a triangle equivalence $\dperfs ZX\iso\dperfs{Z'}{X'}$.
\item There exists a triangle equivalence $\Dloc_Z(X)\iso\Dloc_{Z'}(X')$.
\end{enumerate}
\end{thmInt}

This provides a solution to Problem 4.11 in \cite{CNS25}, and leads to a major generalization of Theorem C in \cite{Canonaco-Neeman-Stellari24}. The reader might also wish to compare with the beautiful, classical result by Rickard \cite{Rickard89b}, which is the $\D(\MMod R)$ version of the above result.

\subsubsection*{The Uniqueness Conjecture for the homotopy category of spectra}

In the world of homotopy theory \autoref{thm:main3} has an old antecedent, the Margolis Uniqueness Conjecture. For the reader's convenience the conjecture is recalled in \autoref{conj:margolis}, but for the purpose of the introduction let us rephrase it as saying that $\ct^c=\Ho\Spe^c$ determines $\ct=\Ho\Spe$.
And now we discuss the relevance of the results
in the article.

Let $\ct_1$ and $\ct_2$ be two compactly
generated triangulated categories, and assume
that $\ct_1^c\cong\ct^c_2$. In
\autoref{prop:ax2} we will learn that
if $\ct_1$ is weakly approximable then so is
$\ct_2$, while
\autoref{excellencebycompacts}
teaches us that if $\ct_1$ is coherent then so
is $\ct_2$. 
Now let us specialize to the case
where $\ct_1=\Ho\Spe$; it is both weakly
approximable and coherent, and hence
any possible $\ct_2$ as in \autoref{conj:margolis}
must be weakly approximable and coherent.

This allows us to formulate a
generalized Margolis conjecture, see
\autoref{conj:margolis2}; the statement is
enhancement-free. What we
can prove in this direction,
as an application of
the
abstract \autoref{thm:main3},
is an assertion
involving enhancements.

\begin{thmInt}[\autoref{margolis2}]\label{thm:main5}
Let $\ct$ be a coherent, weakly approximable triangulated category. Now consider
the 
two lists of categories
\begin{align*}
\cc_\Spe:=&\left\{\Ho{\Spe},\Ho{\Spe}^b,\Ho{\Spe}^-,\Ho{\Spe}^+,\Ho{\Spe}\SB,\Ho{\Spe}^-_c,\Ho{\Spe}^b_c,\Ho{\Spe}^c\right\},\\
\cc_\ct:=&\left\{\ct,\ct^b,\ct^-,\ct^+,\tsb,\ct^-_c,\ct^b_c,\ct^c\right\}.
\end{align*}
and let $\ca_1\subset\cb_1$ in $\cc_\Spe$ and $\ca_2\subset\cb_2$ in $\cc_\ct$ be matching subcategories on the lists. If there is a triangle equivalence $\ca_1\iso\ca_2$, if $\cb_2$ has an enhancement in
$\HoStabInfNL$, and if $\cb_i\neq(\ct_i)^-_c$,
then there is a triangle equivalence $\cb_1\iso\cb_2$.
\end{thmInt}
Note that \cite{Schwede07} proved the special case which considered only the pair of subcategories $\Ho\Spe^c\subset\Ho\Spe$, as in the original Margolis Uniqueness Conjecture (see \autoref{conj:margolis}).

\subsubsection*{Comparing autoequivalence groups}

Finally we come to autoequivalence groups; there
is a rich literature out there,
studying the group of
autoequivalences of certain triangulated categories.
The contribution of the current article is
as follows.
\begin{thmInt}[\autoref{corollaryautomorphisms}]\label{thm:main6}
Suppose $\ct$ is a weakly approximable
triangulated category, such that the subcategories
$\ct^c,\tsb,\ct^b,\ct^-,\ct^+,\ct$ all
have unique enhancements. Then the automorphism
groups of these enhancements are all isomorphic.

If we furthermore assume that one of the following holds
\be
\item
either $\ct$ is coherent,
\item
or we have $\ct^c\subset\ct_c^b$ and ${}^\perp(\ct^b_c)\cap\ct^-_c=\{0\}$,
\ee
then $\ct^b_c$ also has a unique enhancement, and
the automorphism group of this unique enhancement is the same.
\end{thmInt}
In \autoref{exmallisomorphisms}
we will see that this theorem applies
to the following three cases:
\be
\item
$\ct=\Ho\Spe$, the homotopy category of
spectra. 
\item
$\ct=\Dqcs Z(X)$, where 
$X$ is a noetherian scheme and $Z$ is a closed subset.
\item
$\ct=\Dqc(X)$, where $X$ is a quasi-compact and quasi-separated scheme.
\ee
Thus with $\ct$ any of the above, for any
$\ct^\star_\Box\in
\{\ct^c,\ct^b_c,\tsb,\ct^b,\ct^-,\ct^+,\ct\}$
and with $\TT^\star_\Box$ being its unique
enhancement, the group $\AutL(\TT^\star_\Box)$ is
independent of the decoration.

There is of course the natural comparison
map
$\Comp\colon\AutL(\TT^\star_\Box)\la\Aut(\ct^\star_\Box)$,
which at some point was optimistically
conjectured to be an isomorphism. 
In relatively few cases this conjecture has
been proved; the sum of what is known
is that the map
$\Comp\colon\AutL(\TT^c)\la\Aut(\ct^c)$
is an isomorphism 
for $\ct=\Dqc(X)$, provided $X$ is a scheme proper
over a field $\kK$ and of depth $\geq 1$. See the
proof of \autoref{Literature} for a summary of
how to extract this from the literature.
In combination with
\autoref{thm:main6} this tells us that,
with $\ct=\Dqc(X)$ and $X$ still satisfying
the restrictive hypothesis above, the maps 
$\Comp\colon\AutL(\TT^\star_\Box)\la\Aut(\ct^\star_\Box)$,
are split injective for all decorations
other than $\ct^-_c$. And as the $\AutL(\TT^\star_\Box)$
are all isomorphic, so are the images of
the injective maps into $\Aut(\ct^\star_\Box)$. This answers a question by Evgeny Shinder.

\subsection*{Structure of the paper}

As we have already pointed out, the main idea in this paper is to combine techniques coming from higher categories, with those arising in the theory of weakly approximable triangulated categories and metrics on them. Thus we begin with a review of both fields, and somewhat expand on what is known.

In particular, \autoref{subsec:istr} provides a quick recap on \tstr s on triangulated categories, and \autoref{subsec:gen} reminds the reader what it means for a triangulated category to be generated, possibly by compact objects.

These are the key ingredients in the definition of weakly approximable triangulated categories which comes in \autoref{sec:wa}. \autoref{subsec:wadef} contains the precise definition and all the examples needed in this paper. With one exception the triangulated subcategories of a weakly approximale triangulated category, which are the main protagonists in this work, are recalled in \autoref{ex:waexample}. The one exception, being mostly new, is introduced and studied in \autoref{subsec:waTsb} and \autoref{subsec:functorTsb};
it happens to be technically important to our development.
And then \autoref{subsect:geomcase} works out
what this category comes down to, concretely,
in the case where $\ct=\Dqcs Z(X)$.

The discussion of metrics is postponed until
\autoref{subsec:metrics}. Even though the
subject of metrics is enhancement-free, after
\autoref{sec:wa} we switch gears and 
in \autoref{sec:algebraictria} we enter
into the realm of
higher categorical models.
In \autoref{subsec:generalmod} we briefly recap the various levels of generality which we can adopt when defining stable $\infty$-categories, and compare them to dg categories. In \autoref{subsec:uniqenh} we formalize the notion of uniqueness of enhancements and prove some general and useful results. \autoref{subsec:uniqenhex} deals with the crucial examples in this paper, while \autoref{subseq:stronguniq} discusses the strong version of uniqueness, with examples.

\autoref{S974} and \autoref{S2} are the key sections where we develop the new techniques which are used to prove some crucial parts of \autoref{thm:main2} and \autoref{thm:main3}. We investigate the interplay between enhancements and good or excellent metrics. The first important applications, to weakly approximable triangulated categories, come in \autoref{subsec:applwa}.

\autoref{S4} is the culmination of this paper. In \autoref{subsec:constr} and \autoref{subsec:backforth} we develop the necessary machinery in order to relate the enhancements of $\ct^c$ and $\ct^b_c$ and of $\tsb$ and $\ct^b$. In \autoref{subsec:waen} we prove \autoref{thm:main2} and in \autoref{subsec:lifting} we prove \autoref{thm:main3}.

The geometric and topological applications of our main results are contained in \autoref{sect:applications}. In particular, \autoref{subsect:uegeo} deals with the extensions of all known results about uniqueness of enhancements in the geometric context. \autoref{subsect:extspectra} deals with the applications of \autoref{thm:main3} and proves our extension of Rickard's result and the results about a conjecture on the homotopy category of spectra. Finally, \autoref{subsect:autoequiv} deals with our results comparing autoequivalence groups.

\subsection*{Notation}

In the remainder of the paper, we denote by $\kK$ the commutative ring over which the (triangulated, stable $\infty$, dg, etc.) categories or the schemes are linear. We write $R$ or $S$ for arbitrary and possibly noncommutative rings.

\section{Preliminaries on \tstr s and generation}\label{sec:wepreltstrgen}

This section reviews the basic definitions concerning \tstr s, as well as the notion of generation in triangulated categories. We also introduce the three basic examples which will recur throughout the paper.

\subsection{Reminder on \tstr s}\label{subsec:istr}

Recall the following standard definition.

\begin{defin}\label{def:tstr}
Let $\ct$ be a triangulated category. A \emph{\tstr}\ on $\ct$ is a pair
of strictly full subcategories $\tau=(\ct^{\leq 0},\ct^{\geq 1})$ satisfying:
\begin{itemize}
\item[(i)]
$\sh{\ct^{\leq 0}}\subset\ct^{\leq 0}$ and $\sh[-1]{\ct^{\geq 1}}\subset\ct^{\geq 1}$.
\item[(ii)]
$\Hom(\ct^{\leq 0},\ct^{\geq 1})=0$.
\item[(iii)]
For any object $B\in\ct$ there exists a distinguished triangle $A\la B\la C\la\sh{A}$, 
with $A\in\ct^{\leq 0}$ and $C\in\ct^{\geq 1}$. 
\end{itemize}
\end{defin}

In the setting above, for all $n\in\ZZ$ we
define the following
strictly full subcategories
\[
\ct^{\leq n}:=\sh[-n]{\ct^{\leq 0}}\ ,\qquad\ct^{\geq n}:=\sh[-n+1]{\ct^{\geq 1}}\ .
\]
\autoref{def:tstr}(iii) generalizes to state that, for any object $B\in\ct$ and any integer $n\in\ZZ$, there exists a distinguished triangle $A\la B\la C\la\sh{A}$
with $A\in\ct^{\leq n}$ and $C\in\ct^{\geq n+1}$. Such a triangle
is unique up to canonical isomorphism, and defines the \emph{truncation functors}
\[
\tau^{\leq n}\colon\ct\lto\ct^{\leq n}\qquad\tau^{\geq n}\colon\ct\lto\ct^{\geq n},
\]
for all $n\in\ZZ$, where $\tau^{\leq n}(B)=A$ and $\tau^{\geq n+1}(B)=C$. It is a standard result that $\tau^{\leq n}$ is the right adjoint to the inclusion $\ct^{\leq n}\hookrightarrow\ct$, while $\tau^{\geq n}$ is the left adjoint to the inclusion $\ct^{\geq n}\hookrightarrow\ct$.

\rmk{rmk:tstreq}
If $\fF\colon\ct\la\ct'$ is a triangle equivalence and $\tau=(\ct^{\leq 0},\ct^{\geq 1})$ is a \tstrs on $\ct$, then $\fF(\tau):=(\fF(\ct^{\leq 0}),\fF(\ct^{\geq 1}))$ is a \tstrs on $\ct'$.
\ermk

\exm{ex:basicex}
In this paper it will be important to keep in mind three main examples of triangulated categories with arbitrary coproducts.

(i) Let $R$ be a ring. The category $\D(\MMod{R})$ is the derived category of the Grothendieck abelian category $\MMod{R}$ of (right) modules over $R$.

(ii) Let $X$ be a quasi-compact and quasi-separated scheme, and let $Z\subset X$ be a closed subscheme such that the complement $X\setminus Z$ is quasi-compact. The triangulated category $\Dqcs{Z}(X)$ is the derived category,
where the objects are the (unbounded) complexes of $\OO_X$-modules, with quasi-coherent cohomology supported on $Z$.

(iii) Finally, we can consider the homotopy category of spectra $\ho{\Spe}$, which is the localization of the category $\Spe$ of spectra by the weak equivalences. See \cite[Section 1.2]{CNS25} for a little more detail.

Each of these triangulated categories comes with a standard \tstr, which we denote by $\tau_\text{stan}$. While the standard \tstr{s} in (i) and (ii) are classical, and the general theory developed in \cite[Chapter~1]{BeiBerDel82} can be viewed as formalizing the key properties of these old constructions, the existence of the standard \tstrs of (iii) was only discovered a decade later in \cite[Proposition~1.4]{Neeman92C}. More explicitly, the \tstrs on $\ho{\Spe}$ is
defined by taking $\ho{\Spe}^{\leq 0}$ to consist of the \emph{coconnective spectra,} and then taking its right orthogonal in order to form $\ho{\Spe}^{\geq 1}$. Once again, the reader can find a little more detail in \cite[Section 2.3]{CNS25}.
\eexm

The class of \tstr s on a triangulated category $\ct$ is endowed with an equivalence relation. Two \tstr s $\tau_1=(\ct_1^{\leq 0},\ct_1^{\geq 1})$ and $\tau_2=(\ct_2^{\leq 0},\ct_2^{\geq 1})$ on $\ct$ are declared \emph{equivalent} if there exists an integer $N\geq 0$ such that
\[
\ct_2^{\leq -N}\subset\ct_1^{\leq 0}\subset \ct_2^{\leq N}.
\]
Later, we will need the following.

\dfn{def:chartstr}
A \tstrs $\tau=(\ct^{\leq 0},\ct^{\geq 1})$ on a triangulated category $\ct$ is 
\emph{characteristic} if, for any triangulated autoequivalence
$\fF\colon\ct\la\ct$, the \tstr s $\tau$ and $\fF(\tau)$ are equivalent.
\edfn

\subsection{Generation in triangulated categories}\label{subsec:gen}

Next we recall a few basic constructions, which start with some subcategories of a triangulated category and build up others. This technique will be useful when we discuss various notions of ``generation'', and will allow us to define the ``preferred'' equivalence class of \tstr{s}.

Let $\ca$ and $\cb$ be full subcategories of a triangulated category $\ct$. We can perform the following constructions:
\begin{itemize}
\item
$\ca*\cb\subset\ct$ is the full subcategory of all objects $C\in\ct$ for which there exists a distinguished triangle $A\la C\la B$ with $A\in\ca$ and $B\in\cb$.
\item
$\add(\ca)\subset\ct$ is the full subcategory whose objects are all finite direct sums of objects of $\ca$. In case $\ct$ has (set indexed) coproducts, then $\Add(\ca)\subset\ct$ is the full subcategory whose objects are all the small coproducts of objects of $\ca$. 
\item
$\Smr(\ca)\subset\ct$ is the full
subcategory with objects all direct summands of objects of $\ca$.
\setcounter{enumiv}{\value{enumi}}
\end{itemize}
Using these operations, we can define the full subcategory $\COprod(\ca)$ to be the smallest full subcategory $\cs\subset\ct$ satisfying
\[
\ca\subset\cs,\qquad\cs*\cs\subset\cs,\qquad\add(\cs)\subset\cs.
\]
Set 
\[
\gen\ca{}:=\Smr\left(\COprod(\ca)\right).
\]

When $\ct$ has coproducts, which is the most interesting case for the purposes of this paper, we declare $\Coprod(\ca)$ to be the
smallest full subcategory $\cs\subset\ct$ satisfying
\[
\ca\subset\cs,\qquad\cs*\cs\subset\cs,\qquad\Add(\cs)\subset\cs.
\]
And by analogy with the construction above, we set
\[
\ogen\ca{}:=\Smr\left(\Coprod(\ca)\right).
\]

When $\ca$ consists only of some specified shifts of a given object $G$, we adopt the following notation, where $m\leq n$ are in $\ZZ\cup\{\pm\infty\}$:
\[
G[m,n]:=\{\sh[i]{G}\st i\in\ZZ\text{ and } m\leq -i\leq n\}\qquad \genu G{}{m,n}:=\gen{G[m,n]}.
\]
If $\ct$ has coproducts we can perform the
analoguous ``big'' constructions 
\[
\ogenu G{}{m,n}:=\ogen{G[m,n]}.
\]

Next we recall:

\dfn{definitionofcompacts}
Let $\ct$ be a triangulated category.
An object $C\in\ct$ is \emph{compact} if, given any set $\{D_\lambda\}_{\lambda\in\Lambda}$ of objects in $\ct$ such that $\coprod_{\lambda\in\Lambda}D_\lambda$ exists, the natural morphism
\begin{equation}\label{eq:cmpt}
\coprod_{\lambda\in\Lambda}\Hom(C,D_\lambda)\lto\Hom\bigg(C,\coprod_{\lambda\in\Lambda}D_\lambda\bigg)
\end{equation}
is an isomorphism.
\edfn

The full subcategory $\ct^c$ of $\ct$, consisting of all compact objects, is easily seen to be a triangulated
subcategory.
If $\ct$ has coproducts, we say that $\ct$ is \emph{compactly generated} by $G\in\ct^c$ if $\ct={\ogen G{}}^{[-\infty,+\infty]}$. This is equivalent to $\ct^c$ being \emph{classically generated} by $G$, which means that $\ct^c={\gen G{}}^{[-\infty,+\infty]}$.

\exm{ex:compgen}
Let us revisit the three examples in \autoref{ex:basicex}.

(i) It is easy to show that $R$ is a compact generator for $\D(\MMod{R})$. Moreover, $\D(\MMod{R})^c\iso\K^b(\proj{R})$, where the latter is the homotopy category of bounded complexes of finitely generated projective $R$-modules.

(ii) In \cite[Theorem 6.8]{Rouquier08} it is proved that $\Dqcs{Z}(X)$ is compactly generated by a single compact generator and that, furthermore, $\Dqcs{Z}(X)^c=\dperfs Z {X}$. The latter category is the full subcategory of $\Dqcs{Z}(X)$ consisting of perfect complexes with cohomology supported on $Z$. A complex in $\Dqc(X)$ is \emph{perfect} if it is locally quasi-isomorphic to a bounded complex of locally free sheaves with finite rank.

(iii) The triangulated category $\ho{\Spe}$ is also compactly generated by a single generator; this compact generator may be taken to be the \emph{sphere spectrum}.
\eexm

This allows us to define a special class of
\tstr s. Let $\ct$ be a triangulated category with small coproducts and consider an essentially small full subcategory $\ca\subset\ct^c$. Consider the pair of full subcategories
\[
\tau_\ca:=\big(\Coprod(\ca),\Coprod(\ca)^\perp\big).
\]
Here $\Coprod(\ca)^\perp=\{A\in\ct\st\Hom(S,A)=0\text{ for all }S\in\Coprod(\ca)\}$. It was proved in \cite[Theorem~A.1 and Proposition~A.2]{Alonso-Jeremias-Souto03} (see also \cite[Theorem 2.3.3]{Canonaco-Neeman-Stellari24} for a different proof) that, if $\sh{\ca}\subset\ca$, the pair $\tau_\ca$ is a \tstrs on $\ct$.

If $G$ is a compact generator for $\ct$ and $\ca:=\{\sh[i]{G}:i\in\ZZ_{\geq 0}\}$, then the equivalence class of the \tstrs $\tau_G:=\tau_\ca$ is independent of $G$ and is called the \emph{preferred equivalence class} of \tstr s on $\ct$.

\exm{ex:basicextstr}
Combining \autoref{ex:basicex} and \autoref{ex:compgen}, the three triangulated categories $\D(\MMod{R})$, $\Dqcs{Z}(X)$ and $\ho{\Spe}$ are compactly generated by a single compact object, and each is endowed with a standard \tstr. It follows from an easy computation for (i), from \cite[Theorem 3.2]{Neeman22A} for (ii) and from \cite[Example 5.2]{Neeman24} for (iii) that, in each of the three cases, the standard \tstrs is in the preferred equivalence class.
\eexm

\section{Weakly approximable triangulated categories and their subcategories}\label{sec:wa}

In this section we recall the crucial notion of weakly approximable triangulated categories, see \autoref{subsec:wadef}. We prove some basic, general properties of such categories, and in \autoref{subsec:waTsb} we study in some detail a special subcategory $\tsb$ of a weakly approximable triangulated category $\ct$. This subcategory was introduced relatively recently, but we will show that it shares many of the good properties of its older cousins. The proof of the first main result \autoref{thm:main1} is in \autoref{subsec:waTsb}. This section ends with an explicit description of $\tsb$ in the geometric situation, see \autoref{subsec:functorTsb} and \autoref{subsect:geomcase}.

\subsection{The definition}\label{subsec:wadef}

Weakly approximable triangulated categories were introduced in \cite{Neeman24} and, since then, they turned out to be pervasive in algebra and geometry. Let us begin with the formal definition.

\dfn{def:wa}
A triangulated category with coproducts $\ct$ is
\emph{weakly approximable} if there exist an integer $A>0$,
a compact generator $G\in\ct^c$ and a \tstrs
$\tau=\tst\ct$ on $\ct$, such that
\begin{itemize}
\item[(i)]
$G\in\ct^{\leq A}$ and $\Hom(G,\ct^{\leq-A})=0$.
\item[(ii)]
For any object $F\in\ct^{\leq0}$ there exists in $\ct$ a distinguished triangle
$E\la F\la D\la\sh{E}$ with $E\in\ogenu G{}{-A,A}$ and with
$D\in\ct^{\leq-1}$.
\end{itemize}
\edfn

A triple $(G,\tau,A)$ satisfying (i) and (ii) in \autoref{def:wa} will be referred to as \emph{weak approximation data} for $\ct$. Weakly approximable triangulated categories are already known to have in common a slew of interesting, useful  properties, and this article is about extending this further.

\rmk{rmk:wabasicprop}
(i) From \cite[Remark 5.9]{Levy-Sosnilo25} we learn that \autoref{def:wa} (ii) is not redundant. Parts of the theory can be developed without it, but not all.

(ii) If $\ct$ is a weakly approximable triangulated category with weak approximation data $(G,\tau,A)$, then the \tstrs $\tau$ is in the preferred equivalence class (see \cite[Proposition 3.4]{Neeman24}).

(iii)  Let $\ct$ be a weakly approximable triangulated category. Let $H$ be a compact generator of $\ct$ and let $\tau$ be a \tstrs in the preferred equivalence class. Then there exists an integer $B>0$ such that $(H,\tau,B)$ is a weak approximation data (see \cite[Proposition 3.6]{Neeman24}).
\ermk

It is worth noting that the weak approximability of
a compactly generated triangulated category $\ct$
can be checked inside the subcategory $\ct^c$ of compact objects in $\ct$. We formulate this as:

\pro{prop:ax2}
Let $\ct$ be a compactly generated triangulated category with coproducts. Then the following are equivalent:
\begin{itemize}
\item[(a)] The category $\ct$ is weakly approximable.
\item[(b)] There is a classical generator $G$ of $\ct^c$ and a positive integer $B$ such that:
\begin{enumerate}
\item[{\rm (i')}] $\Hom(G,\sh[n]G)=0$ for all $n\geq B$.
\item[{\rm (ii')}] For any object $F\in\genul G{}0$ there exists in $\ct^c$ a distinguished triangle
$E\la F\la D\la\sh{E}$ with $E\in\genu G{}{-B,B}$ and with
$D\in\genul G{}{-1}$.
\end{enumerate}
\end{itemize}
\epro

\prf
Suppose (a) holds
and $G$ is a classical generator
for $\ct^c$. Then $G$ is
a compact generator in $\ct$, and
$\tau_G=\tstv\ct G$ is
a \tstrs in the preferred
equivalence class.
By
\cite[Proposition 3.6]{Neeman24}.
there exists an integer $B>0$ such that
$\Hom\left(\sh[-B]G,\ct_G^{\leq0}\right)=0$, and
\cite[Lemma 3.12]{Neeman24} permits
us, possibly at the cost of
increasing the integer $B$, to also
guarantee that every object
$F\in\ct^c\cap\ct_G^{\leq0}$ admits in
$\ct^c$ a triangle
$E\la F\la D\la\sh{E}$
with $E\in\genu G{}{-B,B}$ and
with $D\in\ct^c\cap\ct_G^{\leq-1}$.
Now observe that, by definition, $\ct_G^{\leq0}=\ogenul G{}0$. Therefore we have
\[
G[-\infty,0]\subset\ct_G^{\leq0}
\qquad\text{ and }
\qquad
\ct^c\cap\ct_G^{\leq0}=\ct^c\cap\ogenul G{}0=\genul G{}0\ ,
\]
where the last equality is by
\cite[Proposition 1.9]{Neeman17}.
The vanishing of
$\Hom\left(\sh[-B]G,\ct_G^{\leq0}\right)$
implies the vanishing of
$\Hom(G,\sh[n]G)$ for all $n\geq B$,
and the conditions on
$F$ and on $D$ in the triangle
$E\la F\la D\la\sh{E}$
rewrite as in (b). We have proved that
(a) implies (b).

For the reverse implication assume (b) holds
and we want to prove (a). And (a) has two
parts, (i) and (ii)
of \autoref{def:wa}.

We have already observed that any
classical generator $G\in\ct^c$ is
a compact generator of $\ct$.
The vanishing of
$\Hom(G,\sh[n]G)=0$ for all $n\geq B$
implies the vanishing of
$\Hom\left(\sh[-B]G,\ct_G^{\leq0}\right)$
by \cite[Remark 1.24]{Neeman24}, and we
obviously have $G\in\ct^{\leq0}_G$.
Thus (i) holds, and
it remains to prove that (ii) also holds.

Let $F$ be an object in $\ct_G^{\leq 0}$. Consider the $0$-th cohomology object $\mathcal{H}^0_{\tau_G}(F)\in\tsth\ct G$, where the cohomology functor is the one with respect to the \tstrs $\tau_G$. By \cite[Lemma 3.2.1]{Canonaco-Neeman-Stellari24}, there is a set $\{C_i\}_{i\in I}\subset\ct^c\cap\ct_G^{\leq 0}$ and an epimorphism $f\colon \coprod_{i\in I}\mathcal{H}^0_{\tau_G}(C_i)\lto\mathcal{H}^0_{\tau_G}(F)$. By \cite[Lemma 3.1.2]{Canonaco-Neeman-Stellari24}, we can replace $\{C_i\}_{i\in I}$ with a collection $\{B_i\}_{i\in I}\subset\ct^c\cap\ct_G^{\leq 0}=\genul G{}0$, with isomorphisms $\ph_i\colon \mathcal{H}^0_{\tau_G}(B_i)\la \mathcal{H}^0_{\tau_G}(C_i)$, and do it in such a way that the map $f$ above may be realized as
$f\iso \mathcal{H}^0_{\tau_G}(g)$ for some
$g\colon \coprod_{i\in I}B_i\to F$. Put 
$B:=\coprod_{i\in I}B_i$.

Applying (ii') to each $B_i$ gives, for each $i\in I$, a triangle $\sh[-1]D_i\la E_i\la B_i\la D_i$
with $E_i\in\genu G{}{-B,B}$ and with
$D_i\in\genul G{}{-1}$. The coproduct of
these all is a triangle
$\sh[-1]D\la E\la B\la D$, with $D\in\ct^{\leq-1}_G$
and with $E\in\ogenu G{}{-B,B}$.
The long exact sequence in cohomology,
for the functor $\mathcal{H}^0_{\tau_G}$ applied to
this triangle, gives $\mathcal{H}^0_{\tau_G}(E)\la\mathcal{H}^0_{\tau_G}(B)\la\mathcal{H}^0_{\tau_G}(D)=0$,
that is the map $h:E\la B$ is such that $\mathcal{H}^0_{\tau_G}(h)$ is an epimorphism.
But we also know that
$\sh[-1]D,B$ both belong to $\ct^{\leq0}_G$, and
the triangle $\sh[-1]D\la E\la B$ teaches us
that
$E\in\ct^{\leq0}_G*\ct^{\leq0}_G\subset\ct^{\leq0}_G$.

Now consider the composite map $E\stackrel h\la B\stackrel g\la F$. The functor $\mathcal{H}^0_{\tau_G}$ takes
both $g$ and $h$ to epimorphisms, and hence
it takes the composite to an epimorphism.
Complete the composite to a triangle
$E\la F\la \wt D\la \sh E$.
Then $\wt D$ is an extension of
$F,\sh E$, both of which
lie in $\ct^{\leq0}_G$. Hence $\wt D\in\ct^{\leq0}_G$.
But the long exact sequence
$\mathcal{H}^0_{\tau_G}(E)\la\mathcal{H}^0_{\tau_G}(F)\la\mathcal{H}^0_{\tau_G}(\wt D)\la\mathcal{H}^1_{\tau_G}(E)=0$, plus the surjectivity of
$\mathcal{H}^0_{\tau_G}(E)\la\mathcal{H}^0_{\tau_G}(F)$,
tells us that $\mathcal{H}^0_{\tau_G}(\wt D)=0$
and therefore $\wt D\in\ct^{\leq-1}_G$. The
triangle
$E\la F\la \wt D\la \sh E$,
with $E\in\ogenu G{}{-B,B}$ and
with $\wt D\in\ct^{\leq-1}_G$,
proves that (ii) holds.
\eprf

\rmk{rmk:apprsim}
As in \cite[Definition~1.25]{Neeman24}, one can study the subclass of approximable triangulated categories. And it is true that the approximable variant of
\autoref{prop:ax2} is true. That is we assert that
the following are equivalent:
\begin{itemize}
\item[(a)] The category $\ct$ is approximable.
\item[(b)] There is a classical generator $G$ of $\ct^c$ and a positive integer $B$ such that:
\begin{enumerate}
\item[{\rm (i')}] $\Hom(G,\sh[n]G)=0$ for all $n\geq B$.
\item[{\rm (ii'')}] For any object $F\in\genul G{}0$ there exists in $\ct^c$ a distinguished triangle
$E\la F\la D\la\sh{E}$ with $E\in\genu G{B}{-B,B}$ and with
$D\in\genul G{}{-1}$.
\end{enumerate}
\end{itemize}
In the rest of the current article this
plays no role: we have not introduced the notation
$\genu G{B}{-B,B}$,  have not recalled the definition
of approximable triangulated categories (as opposed to weakly approximable ones), and approximable triangulated categories
will not appear again. Hence we leave the proof
to the interested expert, with the observation
that the argument is a minor modification of
the proof of \autoref{prop:ax2}.
\ermk

A weakly approximable triangulated category $\ct$ comes with a list of natural, full triangulated subcategories. Each of these subcategories is defined using some (any) \tstrs $\tst\ct$ in the preferred equivalence class, possibly in conjunction with the subcategory $\ct^c$ of compact objects in $\ct$. The list of subcategories studied so far goes as follows:
\begin{itemize}
\item \emph{Bounded above objects}: $\ct^-:=\bigcup_{m=1}^\infty\ct^{\leq m}$;
\item \emph{Bounded below objects}: $\ct^+:=\bigcup_{m=1}^\infty\ct^{\geq-m}$;
\item \emph{Bounded objects}: $\ct^b:=\ct^-\cap\ct^+$,
\item \emph{Compact objects}: $\ct^c$;
\item \emph{Pseudo-compact objects}: $\ct^-_c$, where an object $F\in\ct$ belongs to $\ct^-_c$ if, for any integer $m>0$, there exists in $\ct$ a distinguished triangle $E\la F\la D$ with $E\in\ct^c$ and with $D\in\ct^{\leq-m}$. Thus
\[
\ct^-_c=\bigcap_{m=1}^\infty\big(\ct^c*\ct^{\leq-m}\big);
\]
\item \emph{Bounded pseudo-compact objects}: $\ct^b_c:=\ct^-_c\cap\ct^b$.
\item \emph{Bounded compact objects}: $\ct^{c,b}:=\ct^c\cap\ct^b=\ct^c\cap\ct^b_c$.
\end{itemize}

\exm{ex:waexample}
The three triangulated categories in \autoref{ex:basicex} are all weakly approximable. Restricting attention to the first two: the weak approximability of $\D(\MMod{R})$ is an easy exercise, while the weak approximability of $\Dqcs{Z}(X)$ is \cite[Theorem 3.2 (iv)]{Neeman22A}. In these two cases, all the subcategories listed above can be described explicitly:

\medskip

\begin{center}
	\begin{tabular}{|c|c|c|}

		\hline
		& $Z\subset X$ as in \autoref{ex:basicex}  & $R$ a ring\\
		\hline
		$\ct$ & $\Dqcs Z(X)$ & $\D(\MMod{R})$\\
		$\ct^-$ & $\Dqcmis Z(X)$ & $\D^-(\MMod{R})$\\
		$\ct^+$ & $\Dqcpls Z(X)$ & $\D^+(\MMod{R})$\\
		$\ct^b$ & $\Dqcbs Z(X)$ & $\D^b(\MMod{R})$\\
		$\ct^c$ & $\dperfs ZX$ & $\dperf{R}$\\
		$\ct^-_c$ & $\Dqcps Z(X)$ & $\K^-(\proj{R})$\\
		$\ct^b_c$ & $\Dqcpbs Z(X)$ & $\K^{-,b}(\proj{R})$\\
  $\ct^{c,b}$ & $\dperfs ZX$ & $\dperf{R}$\\
		\hline
	\end{tabular}
\end{center}

\medskip

The weak approximability of $\ho{\Spe}$ is proved in \cite[Example 5.2]{Neeman24}. The peculiarity of this triangulated category is that $\ho{\Spe}^{c,b}=\{0\}$.
\eexm

The aim of \cite{Canonaco-Neeman-Stellari24} was to show that all the natural inclusions of the triangulated subcategories above are canonical, in an appropriate sense. We will review and extend this result slightly in the next section.

\subsection{The new categories and their intrinsicness as subcategories}\label{subsec:waTsb}

For the purposes of this paper, we will need to enlarge a little the list of subcategories in \autoref{subsec:wadef}. In \cite[Definition 8.5]{Neeman25} a new member was introduced. We revisit it here in more detail under the standing assumption for the entire section that $\ct$ is a weakly approximable triangulated category.
 
\dfn{D0.7}
Let $\tst\ct$ be a \tstr\ in the preferred equivalence class. Then we set
\[
\tsb:=\ct^-\cap\left(\bigcup_{n=1}^\infty{^\perp\ct^{\leq-n}}\right)\ .
\]
\edfn

\rmk{rmK:canonTsb}
Clearly $\tsb$ is well defined as a subcategory of $\ct$,
meaning that it does not depend on the choice of the \tstrs in the preferred equivalence class. But by
\cite[Proposition 6.1.3]{Canonaco-Neeman-Stellari24}
the preferred equivalence class of \tstr{s} on
$\ct^-$ is intrinsic, and hence up to equivalence the
subcategories $\ct^-\cap{^\perp\ct^{\leq-n}}\subset\ct^-$
are intrinsic. Therefore their union
$\tsb\subset\ct^-$ is intrinsic.
\ermk

The reader who compares
\autoref{D0.7} above with
\cite[Definition 8.5]{Neeman25}
will discover what looks like a
discrepancy, but the lemma below
shows that the two definitions agree.

\lem{L0.9}
Let $G\in\ct$ be a compact generator. Then
\[
\tsb=\bigcup_{n=1}^\infty\ogenu G{}{-n,n}\ .
\]
\elem

\prf
Fix a \tstrs $\tst\ct$ in the preferred equivalence class and choose an integer $A>0$ with $\Hom(G,\ct^{\leq-A})=0$. Clearly
\[
\ogenu G{}{-n,n}\subset\ct^-\cap{^\perp\ct^{\leq-n-A}}\ ,
\]
and hence
\[
\bigcup_{n=1}^\infty\ogenu G{}{-n,n}\subset
\ct^-\cap\left(\bigcup_{n=1}^\infty{^\perp\ct^{\leq-n}}\right)=\tsb\ .
\]
For the reverse inclusion we use weak approximability. Let
$F\in\tsb\subset\ct^-$ be any object. After shifting we may assume that $F\in\ct^{\leq0}$.
By \autoref{D0.7}
we may also choose an integer $n>0$ with
$F\in{^\perp\ct^{\leq-n}}$. By \cite[Corollary 3.2]{Neeman24} there exists a distinguished triangle
$E\la F\la D$ with $E\in\ogenu G{}{-n-A+1,A}$
and with $D\in\ct^{\leq-n}$, but $F\in{^\perp\ct^{\leq-n}}$
says that the map $F\la D$ must vanish.
Hence $F$ must be a direct summand of $E$, and
\[
F\in\Smr\left(\ogenu G{}{-n-A+1,A}\right)
\subset\ogenu G{}{-n-A+1,n+A-1}\ .
\]
This concludes the proof.
\eprf

\cor{C0.9specialcasering}
If $R$ is a ring and $\ct=\D(\MMod R)$,
then $\tsb$ identifies with
$\K^b(\Proj R)$, the subcategory of
bounded complexes of projective $R$-modules.
\ecor

\prf
Recalling that $R\in\D(\MMod R)$ is a compact
generator, the
Corollary follows immediately from
the formula
\[
\D(\MMod R)\SB=\bigcup_{n=1}^\infty\ogenu R{}{-n,n}
\]
of \autoref{L0.9}.
\eprf

For the case where $\ct=\Dqcs Z(X)$,
we will come back to the computation
of $\tsb$ in
\autoref{ex:Tsbrins}.

\dfn{Dsmallsubcat}
We introduce yet another subcategory to the mix, 
setting
\[
\tsbb:=\tsb\cap\ct^b\ .
\]
\edfn

\rmk{rmk:TsbTb}
If $\ct^c\subset\ct^b$ then in particular $G\in\ct^b$,
and hence for any $n>0$ we have $\ogenu G{}{-n,n}\subset\ct^b$.
Taking the union over $n>0$ and applying
\autoref{L0.9} gives that $\tsb\subset\ct^b$,
and therefore $\tsbb=\tsb\cap\ct^b=\tsb$.
\ermk

Let us give yet another description of the
subcategory $\tsb\subset\ct^-$.

\lem{lem:anotherformula}
The following formula holds:
\[
\tsb=\ct^-\cap\bigcup_{n=1}^\infty{^\perp(\ct^b\cap\ct^{\leq-n})}\ .
\]
\elem

\prf
We will prove the inclusions
\[
^\perp\ct^{\leq-n}
\subset
{^\perp(\ct^b\cap\ct^{\leq-n})}
\subset
{^\perp\ct^{\leq-n-1}}\ ,
\]
and the equality of the Lemma follows from
\autoref{D0.7} by taking the union over
$n>0$ and intersecting with $\ct^-$.
Of course, the first inclusion comes from
$\ct^b\cap\ct^{\leq-n}\subset\ct^{\leq-n}$
by taking orthogonals. It remains to
prove the second inclusion.

Suppose therefore that $C$ is an object
in ${^\perp(\ct^b\cap\ct^{\leq-n})}$ and
$D$ is an object of $\ct^{\leq-n-1}$,
and we need to prove that $\Hom(C,D)=0$.
Now by \cite[Proposition~4.2]{Neeman24}
the category $\ct$ is left-complete, meaning
that there is a (non-canonical) isomorphism
$D\iso\holim D^{\geq-\ell}$. Thus there
exists a distinguished triangle
\[\xymatrix@C+0pt{
\ds\prod_{\ell=1}^\infty
\sh[-1]{D^{\geq-\ell}}
\ar[r]
&
D
\ar[r]
&
\ds\prod_{\ell=1}^\infty
D^{\geq-\ell}
}\]
As $D^{\geq-\ell}\in\ct^b\cap\ct^{\leq-n-1}$ for all
$\ell>0$, we have that $\Hom(C,-)$ annihilates
both $D^{\geq-\ell}$ and $\sh[-1]{D^{\geq-\ell}}$.
From the triangle we learn that it annihilates $D$.
\eprf

\cor{theintrinsictsbcintb}
The inclusion $\tsbb\subset\ct^b$
is intrinsic, and more concretely we have the
formula
\[
\tsbb=\ct^b\cap\bigcup_{n=1}^\infty{^\perp(\ct^b\cap\ct^{\leq-n})}\ .
\]
\ecor

\prf
The formula for 
$\tsbb=\tsb\cap\ct^b$ comes
from intersecting with $\ct^b$ the
formula for $\tsb$ of
\autoref{lem:anotherformula}.
And the intrinsic nature of
$\ct^\text{\rm sb,b}$, as a 
subcategory of $\ct^b$,
follows from
\cite[Proposition~6.1.8]{Canonaco-Neeman-Stellari24},
which tells us that (up to equivalence) the
subcategories
$\ct^b\cap\ct^{\leq-n}\subset\ct^b$ have an
intrinsic description.
\eprf

Next we prove a couple of lemmas which will help us to show that the only subcategory in the list in \autoref{subsec:wadef} which is always naturally contained in $\tsb$ is $\ct^c$ and the inclusion is intrinsic.

\lem{L0.11}
The inclusion $\tsb\hookrightarrow\ct$ respects coproducts.
More precisely, if we are given a collection of
objects $\{D_\lambda\}_{\lambda\in\Lambda}$
and an object $C\in\tsb$ that satisfies
the universal property of the
coproduct $\coprod_{\lambda\in\Lambda}D_\lambda$ in $\tsb$, then there exists an integer $n>0$ such that
$C$ and all the $D_\lambda$'s
lie in $\ogenu G{}{-n,n}$ (where $G\in\ct$ is a compact generator),
and $C$ satisfies the universal
property of the coproduct $\coprod_{\lambda\in\Lambda}D_\lambda$ in $\ct$. 
\elem

\prf
By \autoref{L0.9}, there exists an integer
$n>0$ such
that $C\in\ogenu G{}{-n,n}$. And,
as the objects $D_\lambda$ are all
direct summands of $C$, they
must all belong to $\ogenu G{}{-n,n}$.

But now $\ogenu G{}{-n,n}$ is closed in
$\ct$ under coproducts, and hence the
coproduct $E=\coprod_{\lambda\in\Lambda}D_\lambda$ in $\ct$ must
belong to $\ogenu G{}{-n,n}\subset\tsb$.
However, since $E$ satisfies the universal property of
a coproduct in the large category $\ct$,
it also satisfies it in the subcategory $\tsb$.
Thus $C\iso E$.
\eprf

In passing we include the following consequence.

\cor{generationTsb}
Let $\ct$ be a weakly approximable triangulated category. Then any subcategory
of $\cs\subset\tsb$, which contains $\ct^c$
and is closed in $\tsb$ under those
coproducts that exist in $\tsb$,
is equal to all of $\tsb$.
\ecor

\prf
The category $\cs$ contains $G\in\ct^c$,
and hence contains $\ogenu G{}{-n,n}$ for
every integer $n>0$;
after all the extensions involved
in forming $\ogenu G{}{-n,n}$
come from
triangles in $\tsb$, and all the coproducts
happen inside $\tsb$. Hence the union
of the $\ogenu G{}{-n,n}$ is
contained in $\cs$, but is equal to $\tsb$.
\eprf

The following is the crucial technical result that we need.

\lem{L0.13}
The inclusion $\ct^c\subset\big(\tsb\big)^c$ holds.

Furthermore, if $C\in\big(\tsb\big)^c$, $G\in\ct$ is a compact generator and $n>0$ is an integer, then every map $C\la Y$ with $Y\in\ogenu G{}{-n,n}$ factors as $C\la D\la Y$ with $D\in\genu G{}{-n,n}$.
\elem

\prf
The inclusion
\[
\ct^c=\gen G{}=\bigcup_{n>0}\genu G{}{-n,n}
\subset \bigcup_{n>0}\ogenu G{}{-n,n}=\tsb
\]
tells us that all the objects
that are compact in $\ct$ belong
to $\tsb\subset\ct$.
But by \autoref{L0.11}
the inclusion $\tsb\hookrightarrow\ct$ respects
coproducts, and hence if an object
in $\tsb$ is compact in
$\ct$ it is obviously compact in $\tsb$. Thus
we have the inclusion
$\ct^c\subset\big(\tsb\big)^c$.

Now, for the second part of the statement, we can proceed as follows.
Suppose that we are given an integer $n>0$.
We define the full
subcategory $\cy\subset\ogenu G{}{-n,n}$
by the rule
\[
\text{Ob}(\cy)=\left\{
E\in\ogenu G{}{-n,n}\st
\begin{array}{c}
\text{every map $C\la E$ with $C\in\big(\tsb\big)^c$ factors}\\
\text{as $C\la D\la E$ with $D\in\genu G{}{-n,n}$}
\end{array}\right\}\ .
\]
It clearly suffices to show that $\cy=\ogenu G{}{-n,n}$.

Obviously any $\sh[i]G$, with $-n\leq i\leq n$, must
belong to $\cy$. Or in symbols $G[-n,n]\subset\cy$.
Thus it suffices to show that $\cy$ is closed
in $\ct$ under coproducts, direct summands and
extensions. And closure under direct summands is easy.

Now take any set $\{E_\lambda\}_{\lambda\in\Lambda}$
of objects in $\cy\subset\ogenu G{}{-n,n}$.
The coproduct exists in $\ct$ and belongs to
$\ogenu G{}{-n,n}\subset\tsb$, and hence satisfies
the universal property of the coproduct in $\tsb$.
But given an object $C\in\big(\tsb\big)^c$ and a
morphism $C\la\coprod_{\lambda\in\Lambda}E_\lambda$,
the compactness of $C$, as an object
in $\tsb$, guarantees that it factors
through a finite subcoproduct.
But then because the $E_\lambda$ all belong
to $\cy$ we may further factor as
\[
\xymatrix{
C\ar[r] & \ds\coprod_{\lambda\in\Lambda'}D_\lambda\ar[r]
& \ds\coprod_{\lambda\in\Lambda'}E_\lambda\ar[r]
& \ds\coprod_{\lambda\in\Lambda}E_\lambda
}
\]
In this factorization $\Lambda'\subset\Lambda$ is
a finite set, and the finitely many $D_\lambda$
belong to $\genu G{}{-n,n}$.
This proves that the subcategory $\cy$ is closed in $\ct$
under coproducts.

To complete the proof, it remains to show that the subcategory $\cy$ is closed in $\ct$ under extensions. To this purpose, we apply
\cite[Lemma~1.6 and Remark~1.7]{Neeman17} with
\[
\cs=\big(\tsb\big)^c,\qquad
\ca=\cc=\genu G{}{-n,n},\qquad\cx=\cz=\cy.
\]
Observe that $\cc=\genu G{}{-n,n}\subset\ct^c$,
and by the first assertion of the current
lemma $\ct^c\subset\big(\tsb\big)^c=\cs$.
Combining the inclusions gives
$\cc\subset\cs$, and $\cs$ is triangulated.
Hence \cite[Remark~1.7]{Neeman17} applies.
Because $\cx=\cy$ and
by the definition of $\cy$, any map $S\la B$,
with $S\in\cs$ and with $B\in\cx$, factors as
$S\la A\la B$ with $A\in\ca$.
Similarly any map $S\la D$, with $S\in\cs$ and
with $D\in\cz$, factors as $S\la C\la D$ with
$C\in\cc$. The hypotheses
of \cite[Lemma~1.6 and Remark~1.7]{Neeman17}
all hold, and any map $S\la E$, with $E\in\cy*\cy$,
must factor as $S\la F\la E$ with
$F\in\genu G{}{-n,n}*\genu G{}{-n,n}\subset\genu G{}{-n,n}$.
This completes the proof that the
subcategory $\cy$ is closed in $\ct$
under extensions.
\eprf

We are now ready to prove that the inclusion $\ct^c\hookrightarrow\tsb$ is intrinsic.

\pro{C0.15}
The following equalities hold: $\big(\tsb\big)^c=\tsb\cap\ct^-_c=\ct^c$, $\tsb\cap\ct^b_c=\ct^{c,b}$.
\epro

\prf
The inclusion $\ct^c\subset\big(\tsb\big)^c$ is
the first part of \autoref{L0.13}.
We need to prove the reverse inclusion.

Choose therefore any object $C\in\big(\tsb\big)^c$.
Because $C$ belongs to
$\big(\tsb\big)^c\subset\tsb=\bigcup_{n>0}\ogenu G{}{-n,n}$,
there must exist an integer $n$ with
$C\in\ogenu G{}{-n,n}$. Therefore
the identity map $\id:C\la C$ is
a morphism from $C\in \big(\tsb\big)^c$ to
$C\in\ogenu G{}{-n,n}$, and
by the second part of
\autoref{L0.13} it must factor through
some object $D\in\genu G{}{-n,n}$. Thus $C$ is a
direct summand of $D$, but the category
$\genu G{}{-n,n}$ is closed under direct
summands. Hence $C\in\genu G{}{-n,n}\subset\ct^c$, and therefore $\big(\tsb\big)^c=\ct^c$.

In order to prove that $\tsb\cap\ct^-_c=\ct^c$, the nontrivial inclusion to show is $\tsb\cap\ct^-_c\subset\ct^c$. Thus pick $C\in\tsb\cap\ct^-_c$. By \autoref{D0.7}, there exists a positive integer $n$ such that $C\in{^\perp\ct^{\leq-n}}$. On the other hand, since $C\in\ct^-_c$,  there is a distinguished triangle $E\la C\la D$, where $E\in\ct^c$ and $D\in\ct^{\leq-n}$. Hence
the map $C\la D$ vanishes, and $C$ is a direct summand of $E$ and is therefore compact.

The statement $\tsb\cap\ct^b_c=\ct^{c,b}$ is by intersecting with $\ct^b$ the equality $\tsb\cap\ct^-_c=\ct^c$.
\eprf

Let us now reconsider the natural inclusions in \eqref{eq:incl}.
We can extend \cite[Theorem B]{Canonaco-Neeman-Stellari24} to include $\tsb$ and its natural subcategory $\tsbb$. More precisely, the following is true:

\thm{thm:passage}
If $\ct$ is a weakly approximable triangulated category, then all the inclusions in the diagram \eqref{eq:incl}, represented by the solid arrows $\ca\hookrightarrow\cb$, are \emph{invariant under triangle equivalences}. By this we mean that given a pair of weakly approximable triangulated
categories $\ct,\ct'$, as well as matching inclusions
$\ca\hookrightarrow\cb\hookrightarrow\ct$ and
$\ca'\hookrightarrow\cb'\hookrightarrow\ct'$ from the diagram \eqref{eq:incl},
then any triangle equivalence $\cb\la\cb'$ must restrict to a triangle equivalence $\ca\la\ca'$.

The same is true for the inclusion $(\ast)$ in the diagram \eqref{eq:incl}, provided we further assume that one of the two conditions below holds:
\be
\item $\ct,\ct'$ are coherent or
\item $\ct^c\subset\ct^b_c$, ${\ct'}^c\subset{\ct'}^b_c$ and
$^\perp(\ct^b_c)\cap\ct^-_c=\{0\}={^\perp({\ct'}^b_c)}\cap{\ct'}^-_c$.
\ee

Finally, the inclusion $(\diamond)$ is invariant under triangle equivalences if $\ct^c\subset\ct^b_c$ and ${\ct'}^c\subset{\ct'}^b_c$.
\ethm

Recall the following, which is just \cite [Definition \ref*{D29.1}]{Neeman18A}.

\dfn{def:cohtria}
A weakly approximable triangulated category $\ct$ is \emph{coherent} if, for any
\tstrs $\tst\ct$ in the preferred equivalence class, there
exists an integer $N>0$ such that every
object $D\in\ct^-_c$ admits a distinguished triangle
$C\la D\la E$ with
$C\in\ct^-_c\cap\ct^{\leq N}$ and with $E\in\ct^b_c\cap\ct^{\geq0}$.
\edfn

\begin{proof}[Proof of \autoref{thm:passage}]
In view of \cite[Theorem B]{Canonaco-Neeman-Stellari24}, we only have to deal with the inclusions involving $\tsb$  and $\ct^\text{{\rm sb},b}$. Now $\tsb\hookrightarrow\ct^-$ is intrinsic by \autoref{rmK:canonTsb}. The case of $\ct^c\hookrightarrow\tsb$ follows from \autoref{C0.15} which also shows that we do not need to deal with $\tsb\cap\ct^-_c$ and $\tsb\cap\ct^b_c$. 
The inclusion $\ct^{\text{{\rm sb}},b}\subset\ct^b$ is intrinsic by \autoref{theintrinsictsbcintb}.

Under the assumption that $\ct^c\subset\ct^b_c$ we have  $\ct^{\text{{\rm sb}},b}=\tsb$, and thus $(\diamond)$ is intrinsic because it coincides with the inclusion $\ct^c\subset\tsb$.

It remains to prove that the inclusion $\ct^{\text{{\rm sb}},b}\subset\tsb$ is intrinsic. Note that, if
$G$ is a compact generator of $\ct$ and
$\tau_G=\tstv\ct G$ is the \tstrs it generates,
then
\[
\ct^+
=\bigcup_{n=1}^\infty\ct_G^{\geq-n+1}
=\bigcup_{n=1}^\infty\big(\ct_G^{\leq-n}\big)^\perp
=\bigcup_{n=1}^\infty G[-\infty,-n]^\perp
\]
where the perpendiculars are understood in the category $\ct$. But $G$ is a classical generator of the category $\ct^c$, which is intrinsic in $\tsb$. Intersecting with $\tsb$ gives
\[
\tsbb=\tsb\cap\ct^+=
\bigcup_{n=1}^\infty \left(\tsb\cap G[-\infty,-n]^\perp\right)\ ,
\]
and the expression on the right is clearly
intrinsic as a subcategory of $\tsb$.
\end{proof}

\exm{ex:coherent}
(i) It is not hard to produce examples of weakly approximable triangulated categories $\ct$ that are coherent. The easiest one is when $R$ is a coherent ring and $\ct=\D(\MMod{R})$. By \cite[Example 5.7]{Neeman18A}, the homotopy category of spectra $\Ho\Spe$ is also coherent.

(ii) By \cite[Proposition 10.2.1]{Canonaco-Neeman-Stellari24}, If $X$ is a quasi-compact and quasi-separated scheme with a closed subscheme $Z\subset X$ such that $X\setminus Z$ is quasi-compact, then $\Dqcs Z(X)$ satisfies (ii) in \autoref{thm:passage}.
\eexm

For later use, let us introduce the following definition.

\dfn{def:triachar}
A triangulated subcategory $\cs\subset\ct$ of a triangulated category $\ct$ is  
\emph{characteristic} if any triangulated autoequivalence
$\ct\la\ct$ restricts to an autoequivalence $\cs\la\cs$.
\edfn

\rmk{chartria}
Let $\ct$ be a weakly approximable triangulated category and let $\cb\hookrightarrow \ca$ be a
pair of subcategories in the diagram \eqref{eq:incl}
satisfying the restrictions (if any) imposed in
\autoref{thm:passage}. The conclusion
of \autoref{thm:passage} obviously implies that $\cb$ is a characteristic subcategory of $\ca$.
\ermk

\subsection{Functoriality of $\tsb$}\label{subsec:functorTsb}

Suppose $\fF\colon\cs\la\ct$ is an exact functor
of weakly approximable
triangulated categories.
We might reasonably ask for conditions
under which $\fF$ takes $\cs^\text{\rm sb}$ into
$\tsb$, and the useful result will be:

\lem{L0.75689}
Let $\fF\colon\cs\la\ct$ be an exact functor
of weakly approximable
triangulated categories, with a
right adjoint $\fG\colon\ct\la\cs$.
Choose \tstr{s}
$\tst\cs$ and $\tst\ct$ in the
(respective) preferred
equivalence classes.
Suppose furthermore that
there exists an integer $n>0$ with
$\fF\big(\cs^{\leq0}\big)\subset\ct^{\leq n}$.
Then $\fF$ takes
$\cs^\text{\rm sb}$ into
$\tsb$ if and only if
the integer
$n>0$ may be increased to also satisfy
$\fG\big(\ct^{\leq0}\big)\subset\cs^{\leq n}$.
\elem

\prf
Choose compact generators $H\in\cs$ and $K\in\ct$,
and choose an integer $A>0$ such that
$H\in\cs^{\leq A}$ and $\Hom\big(H,\cs^{\leq-A}\big)=0$.
And then, by applying
\cite[Lemma~3.9(iv)]{Burke-Neeman-Pauwels18},
we may increase $A>0$ to guarantee that it
also satisfies $H[0,\infty]^\perp\subset\cs^{\leq A}$.
Hence first of all we have
\[
\fF(H)\in\fF\big(\cs^{\leq A}\big)\subset\ct^{\leq A+n}\subset\ct^-\ .
\]

Next, if $\fF$ takes 
$\cs^\text{\rm sb}$ into
$\tsb$ then $\fF(H)\in\tsb$,
and there must exist an integer
$B>0$ with $\fF(H)\in{^\perp\ct^{\leq-B}}$, and
by adjunction $H\in {^\perp \fG\big(\ct^{\leq-B}\big)}$.
But then
\[
\fG\big(\ct^{\leq-B}\big)
\subset
H[0,\infty]^\perp\subset\cs^{\leq A}\ .
\]

Now suppose that the integer $n>0$ may be increased
to satisfy 
$\fG\big(\ct^{\leq0}\big)\subset\cs^{\leq n}$.
Then the inclusions
\[
\fG\big(\ct^{\leq0}\big)\subset\cs^{\leq n}\subset\sh[-n-A]{H^\perp}
\]
tell us, by adjunction, that
\[
\ct^{\leq0}\subset\sh[-n-A]{\fF(H)^\perp},
\]
which says that $\fF(H)$ belongs to
$\ct^-\cap{^\perp\ct^{\leq-n-A}}\subset\tsb$.
By \autoref{L0.9} there exists an integer
$B>0$ with $\fF(H)\in\ogenu K{}{-B,B}$.  
But the functor $\fF$ is exact and has a right
adjoint, and hence respects coproducts, and
combining with \autoref{L0.9}
this allows us to show
\begin{align*}
\fF(\cs^\text{\rm sb})&=\fF\left(\bigcup_{\ell=1}^\infty\ogenu H{}{-\ell,\ell}\right)\subset\bigcup_{\ell=1}^\infty\ogenu {\fF(H)}{}{-\ell,\ell}\\
&\subset\bigcup_{\ell=1}^\infty\ogenu {K}{}{-\ell-B,\ell+B}=\tsb\ ,
\end{align*}
concluding the proof.
\eprf

\subsection{The geometric case}\label{subsect:geomcase}

Let $f:X\la Y$ be a morphism
of quasi-compact, quasi-separated schemes.
Then the functor $\LL f^*:\Dqc(Y)\la\Dqc(X)$
is exact and has a right
adjoint $\R f_*:\Dqc(X)\la\Dqc(Y)$, 
and for the standard \tstr{s} there
exists an integer $n>0$ with
\[
\LL f^*\big(\Dqc(Y)^{\leq0}\big)
\subset\Dqc(X)^{\leq 0}
\quad\text{ and }\quad
\R f^*\big(\Dqc(X)^{\leq0}\big)
\subset\Dqc(Y)^{\leq n}\ .
\]
From \autoref{L0.75689} we deduce that
$\LL f^*$ takes $\Dqc(Y)^\text{\rm sb}$
into $\Dqc(X)^\text{\rm sb}$.

\lem{lem:localityoftsb}
Let $X$ be a quasi-compact, quasi-separated scheme,
and assume that $X=U\cup V$, with $U$ and $V$ quasi-compact open subsets. Let $u:U\la X$ and $v:V\la X$ be the open immersions. Then the combination of the
two inclusions
\[
u^*\Dqc(X)\SB\subset\Dqc(U)\SB
\quad\text{ and }\quad
v^*\Dqc(X)\SB\subset\Dqc(V)\SB\ ,
\]
which we rewrite as the single inclusion
$\Dqc(X)\SB\subset(u^*)^{-1}\Dqc(U)\SB\cap(v^*)^{-1}\Dqc(V)\SB$, is an  equality.
That is: an object
$C\in\Dqc(X)$
belongs to
$\Dqc(X)\SB$ if and only if
$u^*(C)\in\Dqc(U)\SB$
and
$v^*(C)\in\Dqc(V)\SB$.
\elem

\prf
The direction that
needs proof is that, if an object
$C\in\Dqc(X)$ is such that
$u^*(C)\in\Dqc(U)\SB$
and
$v^*(C)\in\Dqc(V)\SB$,
then
$C\in\Dqc(X)\SB$. Now note that
$u^*(C)\in\Dqc(U)^-$
and
$v^*(C)\in\Dqc(V)^-$,
and it is classical that
therefore
$C\in\Dqc(X)^-$.
Thus the only assertion that needs proof is that,
if 
$u^*(C)\in{^\perp\Dqc(U)^{\leq0}}$
and
$v^*(C)\in{^\perp\Dqc(V)^{\leq0}}$,
then there exists an integer
$n>0$ with
$C\in{^\perp\Dqc(X)^{\leq-n}}$.

Set $Z=X\setminus U$, which is a closed subset of $X$ with
quasi-compact complement, contained in
the quasi-compact open subset $V$.  Next recall that, given any object $G\in\Dqc(X)$, 
we may complete the unit of adjunction $\eta:\id\la\R u_*u^*$
to the triangle 
\[\xymatrix@C+0pt{
F
\ar[r]
&
G
\ar[r]^-{\eta}
&
\R u_*u^*(G)
\ar[r]
&
\sh F
}\]
There exists an integer $n>0$ such that, if $G\in\Dqc(X)^{\leq0}$, then all three objects above
belong to $\Dqc(X)^{\leq n}$.
But $u^*(F)=0$, making $F$ an object
in
$\Dqcs Z(X)\cap\Dqc(X)^{\leq n}=\Dqcs Z(X)^{\leq n}$.
Now a classical
result (see for example \cite[Corollary~6.4]{Neeman22A})
tells us that
$\Dqcs Z(X)^{\leq n}=\R v_*\Dqcs Z(V)^{\leq n}$,
and we deduce that any object $G\in\Dqc(X)^{\leq0}$ admits a triangle
$F\la G\la H\la \sh F$ with $F\in\R v_*\Dqc(V)^{\leq n}$ and with 
$H=\R u_*u^*(G)\in\R u_*\Dqc(U)^{\leq 0}$. In symbols this proves
\[
\Dqc(X)^{\leq0}\subset\R v_*\Dqc(V)^{\leq n}\,*\,\R u_*\Dqc(U)^{\leq0}\ .
\]
But we are given that 
$u^*(C)\in{^\perp\Dqc(U)^{\leq0}}$
and
$v^*(C)\in{^\perp\Dqc(V)^{\leq0}}$,
and by adjunction
$C^\perp$ contains both $\R u_*\Dqc(U)^{\leq0}$
and $\R v_*\Dqc(V)^{\leq0}$. Therefore $C^\perp$
contains
$\Dqc(X)^{\leq-n}$.
\eprf

It follows that
belonging to
the subcategory
$\Dqc(X)\SB\subset\Dqc(X)$
is a local property, in
the following, classical sense:

\cor{cor:specialcasedqc}
Let $X$ be a quasi-compact and quasi-separated
scheme. Then 
\be
\item
For every quasi-compact
open subset $U\subset X$,
and with $u:U\la X$ the open immersion,
we have that $u^*\Dqc(X)\SB\subset\Dqc(U)\SB$.
\item
Given a finite open cover $X=U_1\cup\dots\cup U_m$,
with $U_i$ all quasi-compact,
then an object $C\in\Dqc(X)$ is in $\Dqc(X)\SB$ if and only if its pullbacks $u^*_i(C)\in\Dqc(U_i)$ are in $\Dqc(U_i)\SB$ for every $i=1,\dots,m$.
\ee
\ecor

\prf
Part (i) is a special case of
the discussion at the beginning of this section,
while Part (ii) follows
from \autoref{lem:localityoftsb}
by induction on the number $m$ of open sets in
the cover.
\eprf

\exm{ex:Tsbrins}
Now every quasi-compact scheme admits
a finite cover by affine open subsets,
and for $U$ any affine scheme 
\autoref{C0.9specialcasering}
gives an explicit description of
$\Dqc(U)\SB$.
Together with
\autoref{cor:specialcasedqc}
this tells us
that, if $X$ is a
quasi-compact and quasi-separated scheme,
$\Dqc(X)\SB$ coincides with the full subcategory 
$\Dloc(X)\subset\Dqc(X)$ consisting of objects that are locally isomorphic to bounded complexes of locally free sheaves (not necessarily of finite rank).

Now suppose $Z\subset X$ is a closed
subset with quasi-compact complement;
we can wonder what is
$\Dqcs Z(X)\SB$.
We assert the equality
\[
\Dqcs Z(X)\SB=\Dqc(X)\SB\cap\Dqcs Z(X)\ .
\]

To see this note that, working with the standard \tstr{s}, 
for any $n\in\ZZ$ we have $\Dqcs Z(X)^{\leq n}\subset\Dqc(X)^{\leq n}$, and taking perpendiculars in the category $\Dqcs Z(X)$ gives the inclusion
\[
\Dqcs Z(X)\cap{^\perp\Dqc(X)^{\leq n}}
\subset
{^\perp\Dqcs Z(X)^{\leq n}}\ .
\]
On the other hand the inclusion $\iota\colon\Dqcs Z(X)\la\Dqc(X)$ has
a right adjoint $\iota^!\colon\Dqc(X)\la\Dqcs Z(X)$, the local cohomology
functor. And there
exists an integer $B>0$ with
$\iota^!\big(\Dqc(X)^{\leq n}\big)\subset\Dqcs Z(X)^{\leq n+B}$.
This gives that
\[
\Dqcs Z(X)\cap{^\perp\Dqc(X)^{\leq n}}=\Dqcs Z(X)\cap{^\perp\iota^!\big(\Dqc(X)^{\leq n}\big)}\supset
{^\perp\Dqcs Z(X)^{\leq n+B}}\ .
\]
Intersecting these inclusions with
$\Dqcmi (X)$ and then 
taking the union over $n\in\ZZ$ gives the desired equality. Thus
\[
\Dloc_Z(X):=\Dloc(X)\cap\Dqcs Z(X)=\Dqcs Z(X)\SB.
\]
\eexm

\section{Enhancements and their uniqueness}\label{sec:algebraictria}

This section starts with a quick reminder of the language of $\infty$-categories and the notion of stable $\infty$-categories. We also sketch the relation with the older language of dg categories. This is done in \autoref{subsec:generalmod}. In \autoref{subsec:uniqenh} we define what it means for a triangulated category to have a (unique) enhancement, and discuss some basic results which will be used in the rest of the paper. A r\'esum\'e of the known cases of uniqueness of enhancements is in \autoref{subsec:uniqenhex}. Strong uniqueness of enhancements is discussed in \autoref{subseq:stronguniq}.

\subsection{Generalities}\label{subsec:generalmod}

We want to sketch some basic features of $\infty$-categories and their linearizations. We refer to \cite{Lurie09,Lurie11,Lurieundercon} for a much more complete treatment of $\infty$-categories.

\subsubsection*{Stable $\infty$-categories}
We will not repeat here the (long) definition of an $\infty$-category, for which we refer to the books mentioned above. We only recall that an $\infty$-category $\DD$ is \emph{stable} if it carries suitable additional structure that makes it closer to a triangulated category. In particular, it must satisfy a set of axioms prescribing the existence of a zero object, the existence of fibers and cofibers of morphisms, and
such that two composable morphisms form a fiber sequence if and only if they form a cofiber sequence. We denote by $\StabInfNL$ the category of stable $\infty$-categories. Let $\HoStabInfNL$ be the localization of $\StabInfNL$ with respect to $\infty$-categorical equivalences, where \emph{localization} means that there is a functor $\StabInfNL\la\HoStabInfNL$, which is initial in the class of functors taking the specified morphisms to isomorphisms.

Later we will relate these constructions to older ones, in the world of dg categories.

It follows from the axioms that
the homotopy category of a stable $\infty$-category $\DD$ comes with a canonical triangulated structure.
We will denote this homotopy category $\Hc(\DD)$, somewhat unconventionally\footnote{If $\DD$ is a stable $\infty$-category, then the standard notation for its homotopy category is $\pi_0(\DD)$. Recall that $\pi_0(\DD)$ is the triangulated category whose objects are identical with those in $\DD$, while the morphisms are given by $\Hom_{\pi_0(\DD)}(A,B):=\pi_0(\Map_{\DD}(A,B))$. Here $\Map_{\DD}$ denotes the mapping space in $\DD$, and $\pi_0(X)$ stands for the zeroth stable homotopy group of $X$. Since in this paper we follow cohomological conventions, we will write $\Hc$ in place of $\pi_0$, and more generally we let $\mathrm{H}^i=\pi_{-i}$ for all $i\in\ZZ$.}. For later use note that, for a given $f\in\Hom_{\Hc(\DD}(A,B)$, a \emph{lift} of $f$ means a $g\in\Map_{\DD}(A,B)$ such that $f=\Hc(g)$.

\exm{ex:stablecat}
All the examples of triangulated categories of interest in the paper, namely those listed in \autoref{ex:basicex} and their natural subcategories an in the diagram \eqref{eq:incl}, can naturally be realized as homotopy categories of stable $\infty$-categories.
\eexm

We denote by $\PrCat_{\stable}$ the category of presentable stable $\infty$-categories, with morphisms the cocontinuous functors---meaning those functors which preserve small colimits. Furthermore, we let $\PrCat_{\stable}^\omega$ denote the category of compactly generated stable $\infty$-categories, with morphisms the cocontinuous functors that preserve compact objects. For us it is crucial that one can define two functors
\begin{equation*}
    \IndNL \colon \StabInfNL \lto \PrCat_{\stable}^\omega \qquad 
    (-)^c \colon \PrCat_{\stable}^\omega \lto \StabInfNL. 
\end{equation*}
given by Ind-completion and passage to compact objects. Roughly, the Ind-completion of a stable $\infty$-category provides a completion under filtered colimits. It is important to note that, if we stick to idempotent complete stable $\infty$-categories, then the two functors are mutually inverse. For every $\DD\in\StabInfNL$ the \emph{Yoneda functor}
\[
\YonInf\colon\DD\lhook\joinrel\longrightarrow\IndNL(\DD)
\]
yields a fully faithful embedding. When there is no risk of confusion between this Yoneda functor and the standard one for ordinary categories, we simply write $\Yon$ for $\YonInf$.

We should also remark that the construction of the Yoneda embedding is functorial in the following sense. If $\fF\colon\DD_1\lto\DD_2$ is a functor of stable $\infty$-categories, then there is an induced functor
\[
\IndNL(\fF)\colon\IndNL(\DD_1)\lto\IndNL(\DD_2)
\]
such that $\IndNL(\fF)\comp\YonInf=\YonInf\comp\fF$ in $\HoStabInfNL$. By \cite[Proposition 5.3.5.13]{Lurie09}, if $\fF$ is exact then $\IndNL(\fF)$ has a right adjoint $\ResNL(\fF)\colon\IndNL(\DD_2)\lto\IndNL(\DD_1)$, which is uniquely determined by the assignment $\ResNL(\fF)(D)(C):=\Map_{\DD_2}(\fF(C),D)$, for all $D\in\IndNL(\DD_2)$ and all $C\in\DD_1$. 

\rmk{rmk:Indnonstab}
Later in the paper, and in particular in \autoref{prop:uniqueadditive}, we will need to consider the analogue of the Ind-completion of $\infty$-categories which are not necessarily stable. The generality of the construction we need is the one discussed in \cite[Example 3.3(d)]{Antieau18}. In particular, if $\ca$ is a small additive category, then one can consider the stable $\infty$-category $\mathrm{Fun}^\pi(\ca\op,\Spe)$ of small product preserving functors. To simplify the notation, we will denote it by $\IndNL(\ca)$ . It should be noted that, in this case, $\Hc(\IndNL(\ca))\cong\D(\MMod{\ca})$, where $\MMod{\ca}$ denotes the abelian category of finite product preserving functors from $\ca\op$ to $\MMod{\ZZ}$. The Yoneda functor continues to provide an embedding $\DD\hookrightarrow\IndNL(\DD)$. For a thorough account of this the reader is referred to \cite[Section 7.4]{BCKW}, with more in \cite{Klemenc}. For a sketch of a different approach, based on the Dold-Kan correspondence, see  \cite[Remark 7.4.10]{BCKW}; the reader might find this illuminating.
\ermk

\subsubsection*{Linear stable $\infty$-categories}

For a fixed commutative ring $\kK$, we now want to consider (stable) $\infty$-categories with a suitable $\kK$-linearization. As we will soon see, the results we have just outlined all hold in this generality.

First note that $\D(\MMod{\kK})$ is a commutative algebra object in $\PrCat_{\stable}$ via its tensor product structure. Observe that here we are sloppily using the notation $\D(\MMod{\kK})$ for its natural $\infty$-categorical enhancement. And later $\dperf{\kK}$ will stand for the natural $\infty$-categorical enhancement of the compact objects in $\D(\MMod{\kK})$.

A \emph{presentable $\kK$-linear stable $\infty$-category} is a $\D(\MMod{\kK})$-module object of $\PrCat_{\stable}$. We set
\[
    \PrCat_\kK:=\Md_{\D(\MMod{\kK})}(\PrCat_{\stable})
\]
to be the category of such linear stable $\infty$-categories.

A \emph{compactly generated $\kK$-linear stable $\infty$-category} is a $\D(\MMod{\kK})$-module object in $\PrCat_{\stable}^{\omega}$. They form the category
\[
    \PrCat_\kK^{\omega}:=\Md_{\D(\MMod{\kK})}(\PrCat_{\stable}^{\omega}). 
\]

Of course the tensor product also makes $\dperf{\kK}$ into a commutative algebra object in $\StabInfNL$. 
A \emph{small $\kK$-linear stable $\infty$-category} is a $\dperf{\kK}$-module object of $\StabInfNL$, and together they form the category
\[
    \StabInf:=\Md_{\dperf{\kK}}(\StabInfNL).
\]

As in the nonlinear case, we have functors 
\[
\IndL \colon \StabInf \lto \PrCat_\kK^{\omega} \qquad 
(-)^c \colon \PrCat_\kK^\omega \lto \StabInf.
\]
Furthermore, for any $\DD$ in $\StabInf$ we obtain a well-defined Yoneda functor $\Yon\colon\DD\hookrightarrow\IndL(\DD)$. For nonstable categories, the construction of \autoref{rmk:Indnonstab} generalizes to the $\kK$-linear world. The Yoneda functor still has the good functorial properties that we discussed above in the non-linear case, and the functor $\IndL$ still has a right adjoint $\ResL$.

We will denote by $\HoStabInf$ the localization of $\StabInf$ with respect to $\infty$-categorical equivalences. When we do not want to specify whether we are working in the linearized setting, we will use the symbols $\StabInfA$, $\HoStabInfA$ and $\IndC$. This means that $?$ can be either $\kK$ or $\emptyset$.

\subsubsection*{Dg categories}

By \cite{Cohn13,Doni24}, if $\kK$ is a commutative ring, then the theory of stable $\kK$-linear $\infty$-categories is equivalent to that of pretriangulated dg categories. The non-experts may prefer to use this more classical language. Various introductions to it are available by now; the reader is referred to \cite{BondalKapranov91,Drinfeld04,Keller94} for some highlights of the original, foundational work on the subject, to \cite{Yekutieli20} for a detailed textbook treatment, to \cite{Keller06} for an extensive survey, and to \cite{CNS25} for a bare-bone account.

\rmk{rmk:Ainfty}
Although this language will not be used in this paper we remark that, in the linear case, one can equivalently adopt the language of (pretriangulated) $A_\infty$ categories (with various notions of unit). See \cite{Keller01} for a survey of $A_\infty$ world, and \cite{COS1,COS2} for recent work on the relation with the dg one.
\ermk

Recall that a \emph{dg category} is a $\kK$-linear category $\DD$ whose morphism spaces $\Hom\left(A,B\right)$ are complexes of $\kK$-modules and, for all triples of objects $A,B,C$ in $\DD$, the composition maps $\Hom(B,C)\otimes_{\kK}\Hom(A,B)\to\Hom(A,C)$ are morphisms of complexes. There is also a natural notion of a dg functor between dg categories. If $\dgCat$ is the category of dg categories and dg functors, we can form its localization $\Hqe$ with respect to quasi-equivalences. Recall that a dg functor $\fF\colon\DD_1\to\DD_2$ is a \emph{quasi-equivalence} if it is essentially surjective and, for every pair of objects $A,B\in\DD_1$, the natural map $\Hom(A,B)\to\Hom(\fF(A),\fF(B))$ is a quasi-isomorphism of complexes. 

\rmk{rmk:HqeVSStab}
The language of (pretriangulated) dg categories may appear more manageable than that of stable $\infty$-categories, but passing to the localization $\Hqe$ is technically more involved than $\HoStabInf$. The calculus of fraction turns out to be simpler in
the passage from $\StabInf$ to the localized category $\HoStabInf$. 
\ermk

Let $\Cdg(\kK)$ stand for the dg category of complexes of $\kK$-modules.
For every dg category $\DD$ we can construct a much larger one
\[
\dgMod{\DD}:=\Hom(\DD\opp,\Cdg(\kK))
\]
whose objects are called \emph{(right) dg $\DD$-modules}. The homotopy category $\Hc(\dgMod{\DD})$ has a natural triangulated structure. Given a dg category $\DD$, there is a fully faithful dg Yoneda functor $\ydg\colon\DD\hookrightarrow\dgMod{\DD}$. A dg category $\DD$ is \emph{pretriangulated} if the essential image of $\Hc(\ydg)$ is a full triangulated subcategory of $\Hc(\dgMod{\DD})$. The property of being pretriangulated is analoguous to the one of being stable for $\infty$-categories. More precisely, the full subcategory of $\Hqe$ consisting of pretriangulated dg categories is known to be equivalent to $\HoStabInf$. This result was proved in \cite[Corollary 5.5]{Cohn13}; see also \cite[Theorem 0.1]{Doni24} and the paragraph before Meta Theorem 14 in \cite{Antieau18}.

The image of $\ydg$ is contained in the full dg subcategory $\hproj{\DD}$ of $\dgMod{\DD}$ whose objects are \emph{h-projective} dg $\DD$-modules. Recall that an object $M\in\dgMod{\DD}$ is \emph{h-projective} if $\Hom_{\Hc(\dgMod{\DD})}(M,N)=0$ for every $N\in\dgMod{\DD}$ which is acyclic (meaning that $N(A)$ is an acyclic complex for every $A\in\DD$). It is important to keep in mind that the dg category $\hproj{\DD}$ is the dg analogue of the $\infty$-category $\IndL(\DD)$ in the $\infty$-categorical setting above. And $\ydg$ is the dg replacement of the Yoneda functor into the Ind-completion.

\subsection{Uniqueness of enhancements: general results}\label{subsec:uniqenh}

In the remainder of the paper we will mostly be interested in triangulated categories admitting enhancements. We recall that there are at least two, genuinely distinct notions of enhancements.

\dfn{def:algtria}
If $?=\kK,\emptyset$ 
then an \emph{enhancement in $\StabInfA$}, of a triangulated category $\ct$, is a pair $(\TT,\fF)$, where $\TT\in\StabInfA$ and $\fF\colon\Hc(\TT)\la\ct$ is a triangle equivalence.
\edfn

By the previous discussion, the existence of an enhancement in $\StabInf$ is equivalent to the existence of a dg enhancement. Classically, triangulated categories admitting a dg enhancement are called \emph{algebraic}, while those admitting an enhancement in $\StabInfNL$ are called \emph{topological}. 

\exm{ex:dgnonalg}
(i) It is clear that if $R$ is a ring, then $\D(\MMod{R})$ has an enhancement in $\StabInfA$. The same is true for $\Dqcs Z(X)$, where $X$ is a quasi-compact and quasi-separated scheme with a closed subscheme $Z\subset X$ such that $X\setminus Z$ is quasi-compact. It is also clear that if $X$ is a noetherian scheme, then $\Dqcs Z(X)$ is coherent.

(ii) On the other hand, it is well-known that $\ho{\Spe}$ does not have a dg  enhancement, and so also no linear stable $\infty$-categorical enhancement---see, for example, \cite[Propositions 1 and 4]{Schwede10}. However, it does have natural nonlinear $\infty$-categorical enhancements.

(iii) Finally, we should not expect that all triangulated category have an enhancement. Examples are provided in \cite{MSS,Rizzardo-VanDenBergh1} (see also \cite[Section 3.2]{Canonaco-Stellari17} for an extensive discussion).
\eexm

For $?=\kK,\emptyset$ and  $\ct$ a triangulated category, we denote by $\EnhA(\ct)$ the set of equivalence classes of enhancements of $\ct$ in $\StabInfA$. Here two enhancements $(\TT_1,\fF_1)$ and $(\TT_2,\fF_2)$ of $\ct$ in $\StabInfA$ are defined to be equivalent if $\TT_1$ and $\TT_2$ are isomorphic in $\HoStabInfA$. 
And when we write $\EnhA(\Box)$ we mean that we do not, in advance,
specify the triangulated category $\ct$. Thus an element of $\EnhA(\Box)$ can be identified with an isomorphism class of objects in $\HoStabInfA$.

\dfn{def:Enh}
For $?=\kK,\emptyset$, a triangulated category $\ct$ has a \emph{unique enhancement in $\HoStabInfA$} if the set $\EnhA(\ct)$ contains exactly one element.
\edfn

We will often say simply that the given triangulated category has a unique enhancement, with the context clarifying whether we mean uniqueness in $\HoStabInf$ or in $\HoStabInfNL$.

We begin with the following easy but useful observation.

\lem{lem:enhwa}
Let $\ct$ be a weakly approximable triangulated category and take $[\SSS]\in\EnhA(\ct^c)$. Then $\Hc(\IndC(\SSS))$ is weakly approximable.
\elem

\prf
The triangulated category $\Hc(\IndC(\SSS))$ is closed under coproducts by construction. Since $\ct^c\iso\Hc(\SSS)$ and $\Hc(\SSS)\iso\Hc(\IndC(\SSS))^c$, the Lemma follows from \autoref{prop:ax2}.
\eprf

The main point of this paper is to show that, under reasonable assumptions, many of the naturally defined full triangulated subcategories of a weakly approximable triangulated category listed in \autoref{subsec:wadef} have unique enhancements. The key idea is to produce new enhancements from old ones and compare them. First of all we introduce the obvious construction.

\con{restrictenhancements}
Let $\ct$ be a triangulated category and let $\cs\subset\ct$ be
a full triangulated subcategory. We define a restriction map
$\Phi^?_{\ct,\cs}$,
taking an enhancement $(\TT,\fF)$ of $\ct$ to an enhancement
$(\SSS,\fG)=\Phi^?_{\ct,\cs}(\TT,\fF)$ of
$\cs$, where $\SSS$ is defined to be
$\SSS=\fF^{-1}(\cs)$, and $\fG$ is the restriction of $\fF$. When $\fF$ is not relevant, we will also write $\SSS=\Phi^?_{\ct,\cs}(\TT)$. Here $?=\emptyset,\kK$, depending on whether we are considering nonlinear or linear enhancements of $\ct$ (and thus of $\cs$).
\econ

We then have the following.

\lem{R1.25}
If $\cs\subset\ct$ is a characteristic subcategory, then
the assignment $\Phi^?_{\ct,\cs}$ induces a map
\[
\Phi^?_{\ct,\cs}:\EnhA(\ct)\la\EnhA(\cs), 
\]
for $?=\emptyset,\kK$.
\elem

\prf
Suppose that $(\TT_1,\fF_1)$ and $(\TT_2,\fF_2)$ are enhancements of $\ct$ such that there is an isomorphism $f\colon\TT_1\la\TT_2$ in $\HoStabInfA$,
and let
$(\SSS_i,\fG_i):=\Phi^?_{\ct,\cs}(\TT_i,\fF_i)$, for $i=1,2$.
Consider the triangulated autoequivalence $\fF:=\fF_2\comp\Hc(f)\comp\fF_1^{-1}$ of $\ct$. Since $\cs$ is characteristic, $\fF$ restricts to an autoequivalence of $\cs$, and
it follows that the essential image of $\Hc(f\rest{\SSS_1})$ coincides with $\Hc(\SSS_2)$. This proves that $\SSS_1\iso\SSS_2$ in $\HoStabInfA$.
\eprf

Sometimes the restriction map of
\autoref{restrictenhancements} might be reversible.
In \autoref{lem:uniqenhTsb}
we illustrate this phenomenon with a simple example,
but first some preparation in order to properly state the result.

\con{constrtenhTsb}
Let $\cs$ be a triangulated category
that has a classical generator. Suppose that $\SSS\in\StabInfA$ is an enhancement of $\cs$ and choose a classical generator $G$ of $\Hc(\SSS)\iso\cs$.
We define $\Delta(\SSS)\in\StabInfA$ to be a
full subcategory of $\IndC(\SSS)$.
Being a full subcategory, it suffices to specify which
objects in $\IndC(\SSS)$ belong to $\Delta(\SSS)$, but the objects of $\IndC(\SSS)$ are the same as the objects
of $\Hc(\IndC(\SSS))$. The formula is
\[
\text{Ob}(\Delta(\SSS))
\eq
\bigcup_{n=1}^\infty\ogenu G{}{-n,n}
\]
where the union is formed in the triangulated
category $\Hc(\IndC(\SSS))$.
\econ

\rmk{remarkpropcon}
The category $\Delta(\SSS)$ is clearly independent
of the choice of the classical generator $G\in\Hc(\SSS)$.
It delivers a well-defined map
\[
\Delta\colon\Enhq(\cs)\la\Enhq(\Box)\ ;
\]
all this means is that if $\SSS\iso\SSS'$ in $\HoStabInfA$, then
$\Delta(\SSS)\iso\Delta(\SSS')$ in $\HoStabInfA$.
But we do not know in advance the triangulated
category $\Box=\Hc(\Delta(\SSS))$.
\ermk

\pro{lem:uniqenhTsb}
Let $\ct$ be a weakly approximable triangulated category. Then the map
\[
\Phi^?_{\tsb,\ct^c}\colon\Enhq(\tsb)\la\Enhq(\ct^c)
\]
of \autoref{restrictenhancements} is a monomorphism.

More precisely, with
$\Delta$ the map of
\autoref{constrtenhTsb},
the composite
\[\xymatrix@C+0pt{
\Enhq(\tsb)
\ar[rr]^-{\Phi^?_{\tsb,\ct^c}}
&&
\Enhq(\ct^c)
\ar[rr]^-\Delta
&&
\Enhq(\Box)
}\]
takes to itself any isomorphism class $[\TT]$, of
objects
$\TT\in\HoStabInfA$ satisfying $\Hc(\TT)\iso\tsb$.
\epro

\begin{proof}
First of all, 
by \autoref{C0.15}
the subcategory $\ct^c\subset\tsb$ is
characteristic, and then
\autoref{R1.25} says that
that the restriction map
$\Phi^?_{\tsb,\ct^c}\colon\Enhq(\tsb)\la\Enhq(\ct^c)$
is well-defined.
Because $\ct$ is weakly approximable the
category $\ct^c$ has a classical generator,
and 
\autoref{remarkpropcon}
allows us to construct the map
$\Delta\colon\Enhq(\ct^c)\la\Enhq(\Box)$.
Choose an enhancement $\TT$ of $\tsb$, and
let $\SSS=\Phi^?_{\tsb,\ct^c}(\TT)$.
It  suffices to prove the ``more precisely''
assertion, we will produce a natural equivalence
$\TT\la\Delta(\SSS)$.

Consider the restricted Yoneda functor
$\YYon\colon\TT\la\IndC(\SSS)$, where $\YYon(T):=\Map_\TT(-,T)|_{\SSS}$ for any $T\in\TT$. It
clearly suffices
to prove
\be
\item
The map 
$\YYon\colon\TT\la\IndC(\SSS)$
factors through the full
subcategory
$\Delta(\SSS)\subset\IndC(\SSS)$.
\item
The induced map
$\TT\la\Delta(\SSS)$ is an
isomorphism in $\HoStabInfA$.
\ee
And both of these are
assertions that can be checked at
the triangulated level.
Now the objects in $\ct^c=\Hc(\SSS)$
are compact in
$\tsb=\Hc(\TT)$
by \autoref{L0.13}, and compact in
$\Hc(\IndC(\SSS))$ by construction. Hence
the functor
$\Hc(\YYon)\colon\Hc(\TT)\la\Hc(\IndC(\SSS))$
respects those coproducts which exist in
$\tsb=\Hc(\TT)$. If $G\in\ct^c$ is
a classical generator, then the inverse
image under $\Hc(\YYon)$
of the subcategory $\ogenu {\YYon(G)}{}{-n,n}\subset\Hc(\IndC(\SSS))$
contain $\ogenu G{}{-n,n}$.
Thus the inverse image of the union
contains the union, and from
\autoref{L0.9} 
we learn that
$\YYon\colon\TT\la\IndC(\SSS)$
takes $\TT$ into $\Delta(\SSS)$.
This proves (i).

To prove (ii) we need to show that
the functor $\Hc(\YYon):\Hc(\TT)\la\Hc(\Delta(\SSS))$ is
fully faithful and essentially surjective.

For the full faithfulness: 
we let
$\cl\subset\tsb$
be the full subcategory defined by
\[
\text{Ob}(\cl)\eq
\left\{
X\in\tsb=\Hc(\TT)
\left|
\begin{array}{c}
\text{For all objects }Y\in\tsb\text{ the map }\\
\Hom_{\tsb}(X,Y)\la\Hom_{\Hc(\Delta(\SSS))}(\YYon(X),\YYon(Y))\\
\text{ is an isomorphism}
\end{array}
\right.
\right\}\ .
\]
By the definition of $\YYon$ the category $\cl$ contains 
$\Hc(\SSS)=\ct^c$. It is obviously a
triangulated subcategory of $\tsb$,
and as the functor $\Hc(\YYon)$ respects
coproducts $\cl$ is also closed in $\tsb$
under those coproducts that exist in
$\tsb$. 
Appealing to \autoref{generationTsb}
we have that $\cl=\tsb$,
and hence $\Hc(\YYon):\Hc(\TT)\la\Hc(\Delta(\SSS))$
is fully faithful.

The essential surjectivity is
now obvious: by the full faithfulness of
$\YYon$ we have that 
$\ogenu G{}{-n,n}$ maps
isomorphically to $\ogenu {\YYon(G)}{}{-n,n}$.
Hence the union maps onto the union, that
is $\tsb$ maps onto $\Hc(\Delta(\SSS))$.
\end{proof}

\rmk{actualmap}
The proof shows a little more than
the statement of 
\autoref{lem:uniqenhTsb}
asserts: given any $\TT$ with $\Hc(\TT)\iso\tsb$,
the proof produces a morphism
$\TT\la\Delta\Phi^?_{\tsb,\ct^c}(\TT)$
in the category $\StabInfA$, and proves
that it is an isomorphism in $\HoStabInfA$.
What is more: this map is natural in $\TT$,
it is nothing more than a restriction of
the restricted Yoneda funtor
$\YYon\colon\TT\la\IndC(\Phi^?_{\tsb,\ct^c}(\TT))$.
This naturality will play a role in \autoref{subsect:autoequiv}.
\ermk

We recall the following well-known result---see, for example, \cite[Proposition 1.17]{Lunts-Orlov10} in the dg setting. The (classical) proof is
a simpler version of what we saw in
\autoref{lem:uniqenhTsb} above.

\lem{lem:restTTc}
Let $\ct$ be a compactly  triangulated category,
and let $\ct^c$ be the subcategory of compact
objects.
Then the map
\[\xymatrix@C-18pt{
\Phi^?_{\ct,\ct^c}
\ar@{}[r]|-{\ds:}
&
\Enhq(\ct)
\ar[rrrrr]
&&&&&
\Enhq(\ct^c)
}\]
of \autoref{restrictenhancements} is a monomorphism.

More precisely, the composite
\[\xymatrix@C-18pt{
\Enhq(\ct)
\ar[rrrrr]^-{\Phi^?_{\ct,\ct^c}}
&&&&&
\Enhq(\ct^c)
\ar[rrrrr]^-{\IndC}
&&&&&
\Enhq(\Box)
}\]
takes to itself any isomorphism class $[\TT]$, of
objects
$\TT\in\HoStabInfA$ satisfying $\Hc(\TT)\iso\ct$.
\elem

\rmk{rmk:liftTsbTc}
It should be noted that a nice application of \autoref{lem:uniqenhTsb} and \autoref{lem:restTTc}, is the following. If $\ct_1$ and $\ct_2$ are weakly approximable triangulated categories, such that $\ct_1^c\iso\ct_2^c$ and this category has a unique enhancement in $\HoStabInfA$, then $\tsb_1\iso\tsb_2$ (respectively $\ct_1\iso\ct_2$) as long as both of the  $\tsb_i$ (respectively both of the $\ct_i$) have enhancements in $\StabInfA$.
\ermk

\subsection{Uniqueness of enhancements: examples}\label{subsec:uniqenhex}

In this section we collect interesting situations, where the existing literature already proves that some of the triangulated categories described in \autoref{ex:basicex}, and/or their natural subcategories as
in diagram \eqref{eq:incl}, have unique enhancements
in the sense of \autoref{def:Enh}.

Recall that if $\ca$ is a small $\kK$-linear category then, in view of the discussion in \autoref{subsec:generalmod}, we can interpret $\ca$ as a dg category and thus as a stable $\infty$-category. Therefore it makes sense to consider the following:

\pro{prop:uniqueadditive}
If  $\ca$ is a small $\kK$-linear category, then $\Hc(\IndNL(\ca))$ and $\Hc(\IndNL(\ca))^c$ have unique enhancements in $\HoStabInfA$.
\epro

\prf
The uniqueness of enhancement in $\HoStabInf$ follows from \cite[Proposition 2.6]{Lunts-Orlov10}, it amounts to reformulating in the language of linear stable $\infty$-categories the dg-category result cited above. See the  discussion of \autoref{subsec:generalmod}.

But the proof works also in the nonlinear case. If $\DD\in\StabInfNL$ is such that there is a triangle equivalence $\fF\colon\Hc(\IndNL(\ca))\lto\Hc(\DD)$, then we can consider the full subcategory $\BB\subset\DD$ with objects the isomorphs of the $\fF(A)\in\DD$, with $A\in\ca$. Now the Hom-spaces in $\DD$ (and thus in $\BB$) are objects in the category of spectra, which has a natural \tstrs as in \autoref{ex:compgen} (iii). Hence we can form the $\infty$-category $\tau^{\leq 0}\BB$ with the same objects as $\BB$ and with truncated Hom-spaces. Then, as in the proof of \cite[Proposition 2.6]{Lunts-Orlov10}, the proof here amounts to using the quasi-equivalence $\ca\to\Hc(\BB)\leftarrow\tau^{\leq 0}\BB\to\BB$. The first functor is the one induced by $\fF$.
\eprf

\autoref{prop:uniqueadditive} implies the uniqueness of enhancements in the affine setting. After all, if $R$ is a (not necessarily commutative) $\kK$-algebra, then we let $\ca$ be the $\kK$-linear category with only a single object, whose endomorphism ring is $R$. In this case $\Hc(\IndNL(\ca))\iso\D(\MMod{R})$ and hence $\Hc(\IndNL(\ca))^c\iso\dperf{R}$. Thus \autoref{prop:uniqueadditive}
specializes to the following:

\cor{thm:uniqenhring}
If $R$ is a $\kK$-algebra, then the triangulated categories $\D(\MMod{R})$ and $\dperf{R}=\D(\MMod{R})^c$ have unique enhancements in $\HoStabInfA$.
\ecor

The following is a deeper result, that came after the partial progress in \cite{Antieau18,Canonaco-Stellari18,Lunts-Orlov10}.

\thm{thm:uniqenhangeom}
{\rm (i)} If $\ca$ is an abelian category, then the triangulated category $\D^\star(\ca)$ has a unique enhancement in $\HoStabInfA$, for $\star=b,+,-,\emptyset$.

{\rm (ii)} Let $X$ be a quasi-compact and quasi-separated scheme. Then the triangulated categories $\Dqca(X)$ and $\dperf{X}$ have unique enhancements in $\HoStabInfA$, for $\star=b,+,-,\emptyset$.
\ethm

\prf
The uniqueness of enhancements in $\HoStabInf$ may be found in \cite[Theorems A and B]{Canonaco-Neeman-Stellari21}. The uniqueness result in $\HoStabInfNL$ in (i) follows from the linear case and \cite[Meta Theorem 14]{Antieau18}.

It remains to deal with the uniqueness of enhancements in $\HoStabInfNL$ of the triangulated categories in (ii). For this we use the argument of Section 8.1 in \cite{Canonaco-Neeman-Stellari21}. The idea is simple enough: we reduce to the affine case of
\autoref{thm:uniqenhring} by using a finite open cover of $X$.
In \cite{Canonaco-Neeman-Stellari21},
the gluing of the local enhancements is performed by means of \cite[Theorem 7.4]{Canonaco-Neeman-Stellari21}, and the proof there is given in the setting of dg categories. But the argument generalizes to the $\infty$-category context: what is proved is that an enhancement of either $\Dqca(X)$ or $\dperf{X}$ is the homotopy pullback of the enhancements of the restrictions to open sets in the cover. The point being that all these restrictions are localization maps. This proof works verbatim in the nonlinear case---see \cite[Section 1.4.4]{Lurie11} for the theory of localizations of stable $\infty$-categories.
\eprf

We remark that \autoref{thm:uniqenhangeom} answers in the positive Question 8.16 (v) and (vi) in \cite{Antieau18}. Another easy consequence of
\autoref{thm:uniqenhangeom} is:

\cor{cor:uniqensupp}
If $X$ is a noetherian scheme and $Z\subset X$ is a closed subset, then $\Dqcpbs Z(X)$ has a unique enhancement in $\HoStabInfA$.
\ecor

\prf
Since $X$ is noetherian we have $\Dqcpbs Z(X)\iso\D^b(\coh_Z(X))$. The result now follows from \autoref{thm:uniqenhangeom} (i).
\eprf

The categories listed, in the uniqueness-of-enhancement results so far, all come with natural linearizations. There is in the literature one important exception,
the homotopy category of spectra. In this case, the uniqueness of enhancement refers to the $\infty$-categorical models.

\begin{theorem}[{\cite[Rigidity Theorem]{Schwede07}}]\label{thm:uniqenhspectra}
The homotopy category of spectra $\Ho{\Spe}$ and the full subcategory $\Ho{\Spe}^c$ of finite spectra have unique enhancements in $\HoStabInfNL$.
\end{theorem}

We will revisit and improve these results in \autoref{sect:applications}. For the moment, note that uniqueness of enhancements is not true in general for triangulated categories. Examples are contained in \cite{Rizzardo-Symons-VanDenBergh,Rizzardo-VanDenBergh2,Schlichting02}.

\subsection{Strong uniqueness of enhancements}\label{subseq:stronguniq}

In this section, we consider a stronger notion of uniqueness of enhancement. 

First, we fix a variant of the notation $\EnhA(\ct)$
introduced right before \autoref{def:Enh}. If $\ct$ is a triangulated category and $?=\kK,\emptyset$, we denote by $\EnhStrA(\ct)$ the set of equivalence classes of enhancements of $\ct$ in $\StabInfA$. What changes, when we refer to $\EnhStrA(\ct)$ rather than $\EnhA(\ct)$, is the equivalence relation: in the case of $\EnhStrA(\ct)$, two enhancements $(\TT_1,\fF_1)$ and $(\TT_2,\fF_2)$ of $\ct$ in $\StabInfA$ are defined to be equivalent if there is an isomorphism $f\colon\TT_1\to\TT_2$ in $\HoStabInfA$ such that $\fF_1\iso\fF_2\comp\Hc(f)$. Recall that, in forming $\EnhA(\ct)$, two enhancements are declared equivalent if there exists an isomorphism  $f\colon\TT_1\to\TT_2$ in $\HoStabInfA$, with no compatibility condition with $\fF_1$ and $\fF_2$.

\dfn{def:struniqenh}
A triangulated category $\ct$ has \emph{a strongly unique enhancement in $\HoStabInfA$}, for $?=\kK,\emptyset$, if the set $\EnhStrA(\ct)$ contains exactly one element.
\edfn

A key example, of strong uniqueness of enhancements in a geometric context, is the following result. We give here the best known version due essentially to Olander \cite{OlanderThesis}, building on partial results by Lunts and Orlov who deal with projective schemes (see \cite[Theorem 9.9]{Lunts-Orlov10}) and some key ideas from \cite{Canonaco-Stellari14} which we will recall in the course of the proof of \autoref{prop:struniqex} and in \autoref{thm:suppCS}. In the statement of the result, we say that a scheme $X$ has depth $\ge1$ if the stalk $\OO_{X,x}$ has depth $\geq 1$, for any closed point $x\in X$. This is equivalent to $T_0(\OO_X)=0$, where $T_0(\OO_X)$ denotes the maximal $0$-dimensional torsion subsheaf of $\OO_X$.

\begin{prop}\label{prop:struniqex}
If $X$ is a scheme proper over a field $\kK$, such that $X$ has depth $\geq 1$, then $\D^b(\coh(X))$ has a strongly unique linear enhancement.
\end{prop}

\prf
As explained in \cite[Section 5.A]{LorenzinThesis}, the result about the strong uniqueness of linear enhancements follows directly from \cite[Proposition 3.7]{Canonaco-Stellari14} together with \cite[Lemma 3.3.2]{OlanderThesis}.
\eprf

Recall that, under the assumptions in \autoref{prop:struniqex}, the inclusion $\dperf{X}\subset\D^b(\coh(X))$ need not be an equality. Thus \autoref{prop:struniqex}
proves only the strong uniqueness of $\D^b(\coh(X))$, and not of $\dperf{X}$. However, if $X$ is assumed projective (and not only proper), then the older result of \cite[Theorem 9.9]{Lunts-Orlov10} shows that both $\dperf{X}$ and $\D^b(\coh(X))$ have a strongly unique enhancement, under the same depth hypothesis.

\rmk{rmk:otherstunique}
It is a deep and interesting challenge to work out
the extent to which strong uniqueness of enhancements holds, in the geometric or algebraic settings. In view of \cite{Jasso-Keller-Muro}, we know not to expect to always have strong uniqueness. The reader can also have a look at \cite{Rizzardo-Symons-VanDenBergh} for recent examples in this direction. In general we expect the existence of rings $R$ for which either $\D(\MMod R)$, and/or some of its natural subcategories, do not have a strongly unique enhancement. Nonetheless, computations of special examples seem to point in the opposite direction, see \cite{CNSStrUniq25,CY,Lo}.
\ermk

In the case of the derived categories of schemes with support in a closed subscheme, we have the following additional result.

\begin{theorem}[\cite{Canonaco-Stellari14}, Theorem 1.2]\label{thm:suppCS}
Let $X$ be a quasi-projective scheme over a field $\kK$ containing a projective subscheme $Z\subset X$ such
that $Z$ has depth $\ge1$ and $\OO_{iZ}\in\dperf{X}$,
for all $i> 0$.
Then $\dperfs Z{X}$ has a strongly unique
enhancement in $\HoStabInf$.
\end{theorem}

The sheaf $\OO_{iZ}$ in the statement above denotes the quotient $\OO_X/\mathcal{I}_Z^i$, where $\mathcal{I}_Z^i$ is the $i$-th power of the ideal $\mathcal{I}_Z$ of the closed subscheme $Z\hookrightarrow X$. 

We expect that none of the results discussed in this section is optimal and that they will be subject to further improvements. As we will see in \autoref{subsec:waen}, any such new result can be fed into our general machine, transporting the strong uniqueness of enhancements back and forth between the category of perfect complexes and the bounded derived category of coherent sheaves.

\section{Metrics, enhancements and completions}
\label{S974}

We begin with \autoref{subsec:metrics}, recalling the notion of metrics on triangulated categories. Completion with respect to a metric has $\infty$-category versions which we introduce in \autoref{theconstructionL}. After a quick discussion of the basic properties, we prove a result showing that, subject to technical hypotheses that we specify, the completion
of an enhancement as in \autoref{theconstructionL}
achieves the reverse of the restriction map on
enhancements of \autoref{restrictenhancements}.
See \autoref{P15.5} for the precise statement.

\subsection{Metrics}\label{subsec:metrics}

We recall some fundamental definitions and results from \cite{Neeman18A}.

\dfn{def:metric}
A \emph{good metric} on a triangulated category $\ct$ is a sequence $\{\cm_i\}_{i\in\NN}$ of additive full subcategories of $\ct$ containing $0$, and such that for every $i\in\NN$ the following properties hold:
\begin{enumerate}
\item[{\rm (i)}] $\sh[-1]{\cm_{i+1}}\cup\cm_{i+1}\cup\sh{\cm_{i+1}}\subset\cm_i$,
\item[{\rm (ii)}] $\cm_i\ast\cm_i=\cm_i$.
\end{enumerate}
\edfn

In the rest of the paper we will often simplify the notation, and for a good metric we will write $\{\cm_i\}$ in place of $\{\cm_i\}_{i\in\NN}$

\exm{ex:metricTsb}
Let $\ct$ be a weakly approximable triangulated category, and fix a \tstr\ $\tst\ct$ in the preferred equivalence class.

(i) The triangulated category $\ct^c$ carries a good metric $\{\cm_i\}$, where $\cm_i:=\ct^c\cap\ct^{\leq -i}$ for $i\in\NN$. Also the category $(\ct^b_c)\op$ is endowed with a good metric $\{\cn_i\op\}$, where $\cn_i:=\ct^b_c\cap\ct^{\leq -i}$ for all $i\in\NN$ (see \cite[Example 1.5]{Neeman18A}).

(ii) In the same setting, following \cite[Example 8.4 and Definition 8.5]{Neeman25}, for $i\in\NN$ consider the full subcategory $\cm_i:=\tsb\cap\ct^{\leq-i}$ of $\tsb$. The collection $\{\cm_i\}$ provides a good metric on $\tsb$. Moreover, $(\ct^b)\op$ can be endowed with a good metric $\{\cn_i\op\}$, where $\cn_i:=\ct^b\cap\ct^{\leq -i}$ for all $i\in\NN$.
\eexm

For later use we keep in mind the following definition.

\dfn{def:ordermetric}
Let $\{\cm_i\}$ and $\{\cn_i\}$ be good metrics on $\ct$.

{\rm (i)} We say that $\{\cm_i\}$ is \emph{finer than} $\{\cn_i\}$ and write $\{\cm_i\}\preceq\{\cn_i\}$ if, for every integer $i>0$, there exists an integer $j>0$ such that $\cm_j\subset\cn_i$.

{\rm (ii)} The good metrics $\{\cm_i\}$  and $\{\cn_i\}$ are \emph{equivalent} if $\{\cm_i\}\preceq\{\cn_i\}\preceq\{\cm_i\}$.
\edfn

Given a triangulated category $\ct$ with a good metric $\{\cm_i\}$, following \cite{Neeman18A}, one can form three additional categories, actually three subcategories of the category $\MMod{\ct}$ of additive functors $\ct\op\to\MMod{\ZZ}$. Let us review here their constructions.

First of all, recall the following:

\dfn{def:Cauchy}
A sequence $E_\bullet=\{E_1\to E_2\to E_3\to\dots\}$ of objects and maps in $\ct$ is a \emph{Cauchy sequence} if, for every integer $i>0$, there exists an integer $M>0$ such that, in any distinguished triangle $E_m\to E_{m'}\to D_{m,m'}$ with $M\leq m<m'$, the object $D_{m,m'}$ is in $\cm_i$.
\edfn

Next denote by $\Yon[\ct]\colon\ct\la\MMod{\ct}$ the (fully faithful) \emph{Yoneda functor}, defined by $\Yon[\ct](E):=\Hom_\ct(-,E)$, for $E\in\ct$. When there is no possible confusion, we will often write $\Yon$ for $\Yon[\ct]$.

The full subcategory $\fl(\ct)\subset\MMod{\ct}$ has for objects all isomorphs of colimits in $\MMod\ct$ of Cauchy sequences in $\ct$. That is, the objects of $\fl(\cs)$ can all be written as
$\colim\Yon[\ct](E_i)$, where $E_\bullet=\{E_1\to E_2\to E_3\to\dots\}$ is a Cauchy sequence in $\ct$. 

Further one considers the following full subcategory of $\MMod\ct$:
\[
\fc(\ct):=\left\{F\in\MMod{\ct}:\text{there is an integer }i>0\text{ with }\Hom_{\MMod{\ct}}(\Yon[\ct](\cm_i),F)=0\right\}.
\]

Finally, we form the third full subcategory of $\MMod\ct$
\begin{equation}\label{eq:sigma}
\fs(\ct):=\fl(\ct)\cap\fc(\ct),
\end{equation}
which will be fundamental in the rest of this paper.

\rmk{rmk:Stria}
A beautiful property of the category $\fs(\ct)$, which was proved in \cite[Theorem 2.4]{Neeman18A}, is that it is triangulated (with an explicit set of distinguished triangles). Furthermore, $\fs(\ct)$ carries a good metric $\{\cn_i\}$ which, as in \cite[Lemma 1.6]{Neeman25}, is defined as follows. First, for every integer $i\in\nn$ consider the full subcategory $\cl_i\subset\fl(\ct)$ defined by
\begin{equation}\label{eq:metLT}
\cl_i:=\left\{F\in\fl(\ct)\left|\begin{array}{c}\text{There is a Cauchy sequence $E_\bullet$ in $\ct$ such}\\ \text{that $F\iso\colim\Yon[\ct](E_j)$ and $E_j\in\cm_i$ for all $j$}\end{array}\right.\right\}.
\end{equation}
Next one sets
\begin{equation}\label{eq:metrN}
\cn_i:=\cl_i\cap\fs(\ct).
\end{equation}
\ermk

\exm{ex:concrsigma}
Let $\ct$ be a weakly approximable triangulated category.

(i) It is proved in \cite[Proposition 0.16]{Neeman18A} that, for the full subcategories $\ct^c$ and $\ct^b_c$ with the good metrics in \autoref{ex:metricTsb} (i), we have $\fs\big(\ct^c\big)\iso\ct^b_c$ while, if $\ct$ is coherent, then $\fs((\ct^b_c)\op))\iso(\ct^c)\op$.

In particular, if $R$ is a right coherent ring, $\fs(\dperf{R})\iso\D^b(\mmod{R})$ and $\fs\big(\D^b(\mmod{R})\op\big)\iso\dperf{R}\op$. Similarly, in the geometric setting of \autoref{ex:basicex} (ii), if we assume further that $X$ is noetherian then we have  $\fs(\dperfs Z {X})\iso\dbcohs Z(X)$ and $\fs(\dbcohs Z(X)\op)\iso\dperfs Z {X}\op$.

(ii) \cite[Example 8.4]{Neeman25} shows that if $\ct^b$ is endowed with the good metric of \autoref{ex:metricTsb} (ii), then $\fs\big(\big(\ct^b\big)\op\big)\iso\big(\tsb\big)\op$.
\eexm

As in \cite{Neeman18A}, we recall also the following important notion.

\dfn{def;goodext}
Let $\ct_1$ and $\ct_2$ be triangulated categories, and assume that $\ct_1$ has a good metric $\{\cm_i\}$. A fully faithful triangulated functor $\fF\colon\ct_1\to\ct_2$ is a \emph{good extension} with respect to the metric  $\{\cm_i\}$ if, for any Cauchy sequence $E_\bullet$ in $\ct_1$, the natural morphism
\[
\colim\Yon[\ct_1](E_i)\la\YYon[\fF](\Hocolim(\fF(E_i)))
\]
is an isomorphism in $\MMod{\ct_1}$.
\edfn

\rmk{explanationofgoodextension}
Here $\YYon[\fF]\colon\ct_2\to\MMod{\ct_1}$ is the \emph{restricted Yoneda functor} defined as $\YYon[\fF](E):=\Hom_{\ct_2}(\fF(-),E)$, for all $E\in\ct_2$. Again, when there is no risk of confusion, we will write $\YYon$.

Note that, while we are not assuming that the category
$\ct_2$ has coproducts, for the definition to make
sense means in particular that for any Cauchy sequence
$E_\bullet$ in $\ct_1$, we are assuming the existence in
$\ct_2$ of the countable coproduct $\coprod_{i=1}^\infty\fF(E_i)$.
And (as usual) the homotopy colimit is
defined, in $\ct_2$, by completing to a distinguished triangle the
map $\id-\shi$, as in
\[\xymatrix@C+10pt{
\ds\coprod_{i=1}^\infty\fF(E_i)
\ar[rr]^-{\id-\shi}
& &
\ds\coprod_{i=1}^\infty\fF(E_i)
\ar[r]
&
\hoco\fF(E_i)
\ar[r]
&
\ds\coprod_{i=1}^\infty\sh{\fF(E_i)}
}\]
Now, take any object $A\in\ct_1$, and apply to this triangle
the functor $\Hom(\fF(A),-)$ to obtain the bottom
row in the commutative diagram
\[\xymatrix@C-4pt{
\ds\coprod_{i=1}^\infty\Hom\big(\fF(A)\,\,,\,\,\fF(E_i)\big)
\ar[rr]^-{\id-\shi}
\ar[d]^\alpha
& &
\ds\coprod_{i=1}^\infty\Hom\big(\fF(A)\,\,,\,\,\fF(E_i)\big)
\ar[r]^-\gamma
\ar[d]^\alpha
&
\colim\,\Hom\big(\fF(A)\,\,,\,\,\fF(E_i)\big)
\ar@{.>}[d]^{\exists!\beta}
\\
\ds\Hom\left(\fF(A)\,\,,\,\,\coprod_{i=1}^\infty\fF(E_i)\right)
\ar[rr]^-{\id-\shi}
& &
\ds\Hom\left(\fF(A)\,\,,\,\,\coprod_{i=1}^\infty\fF(E_i)\right)
\ar[r]^-\delta
&
\Hom\left(\fF(A)\,\,,\,\,\hoco\fF(E_i)\right)
}\]
The top row is the short exact sequence
for $\colim\,\Hom\big(\fF(A),\fF(E_i)\big)$
in the category of abelian groups.
The vertical map $\alpha$ is the canonical morphism, and
the reader should be alerted that we are not
  assuming the map $\alpha$ to be an isomorphism. There
is no compactness hypothesis on the objects
in the image of the functor
$\fF:\ct_1\la\ct_2$. The solid square on
the left obviously commutes,
and the two equal composites, of
the solid arrows from top left to
bottom right, both vanish.
This permits us to define the
vertical map $\beta$ as the unique
group homomorphism making the diagram commute. 
And the hypothesis of
\autoref{def;goodext} amounts to the assertion
that the map $\beta$ is an isomorphism for all $A\in\ct_1$.

Because the map $\gamma$ is always an epimorphism
so is $\beta\gamma=\delta\alpha$, and hence
$\delta$ must be an epimorphism. The
exactness of
the (extended) bottom row, for all $A\in\ct_1$, tells
us that the bottom row must be a short exact sequence
of abelian groups.
\ermk

As in \cite[Definition~\ref*{D21.7}]{Neeman18A}, if $\fF\colon\ct_1\to\ct_2$ is a good extension with respect to a good metric $\{\cm_i\}$ on $\ct$, we will denote by $\fl'(\ct_1)$ the following full subcategory of $\ct_2$
\begin{equation}\label{eq:L'}
\fl'(\ct_1):=\left\{F\in\ct_2:\text{there is a Cauchy sequence $E_\bullet$ in $\ct_1$ such that }F\iso\Hocolim(\fF(E_i))\right\}.
\end{equation}

We also need to recall something about morphisms of Cauchy sequences.

Let $\ct$ be a triangulated category with a good metric $\{\cm_i\}$. Given two Cauchy sequences $E_\bullet$ and $F_\bullet$ in $\ct$, one defines a sequence $f_\bullet=\{f_i\colon E_i\to F_i\}$ of morphisms in $\ct$ to be a \emph{morphism of Cauchy sequences} if $f_i$ is compatible with the maps in $E_\bullet$ and $F_\bullet$ in the obvious sense, for all $i\in\NN$. A morphism of Cauchy sequences $f_\bullet\colon E_\bullet\to F_\bullet$ is of \emph{type-$n$} if, in any completion to a Cauchy sequence of distinguished triangles
\[
E_\bullet\la F_\bullet\la G_\bullet\ ,
\]
we have $G_i\in\cm_n$ for all $i\gg 0$. Note that the definition requires that the sequence $G_\bullet$ obtained by completing the morphisms $f_j$ to distinguished triangles in $\ct$ can be made to form a Cauchy sequence. This is achieved in \cite[Lemma 2.8 (iii)]{Neeman18A}. 
Consider a morphism $f\colon E\to F$ in $\fl(\ct)$. Because $E$ and $F$ are objects in $\fl(\ct)$ there exist Cauchy sequences $E_\bullet$ and $F_\bullet$ such that $E\iso\colim\Yon[\ct](E_i)$ and $F\iso\colim\Yon[\ct](F_i)$. We say that $f$ is of \emph{type-$n$ with respect to $(E_\bullet,F_\bullet)$} if there exist Cauchy subsequences  $E'_\bullet$ and $F'_\bullet$ of $E_\bullet$ and $F_\bullet$ and a morphism $f'_\bullet\colon E'_\bullet\to F'_\bullet$ of type-$n$ such that $f=\colim f'_\bullet$.

Next we specialize to smaller classes of metrics, where we can prove more. 

\dfn{def:excmetrics}
A good metric $\{\cm_i\}$ on a triangulated category $\ct$ is \emph{very good} if the following holds:
\begin{enumerate}
\item[{\rm (i)}] $\ct=\bigcup_{i\in\NN}{}^\perp\cm_i$;
\item[{\rm (ii)}] For every integer $m\in\NN$ there exists an integer $n>m$ such that any object $F\in\ct$ sits in a distinguished triangle $E\to F\to D$ with $E\in{}^\perp\cm_n$ and $D\in\cm_m$;
\end{enumerate}
A very good metric $\{\cm_i\}$ is \emph{excellent} if the following holds:
\begin{enumerate}
\setcounter{enumiv}{2}
\item[{\rm (iii)}]For every integer $m\in\NN$ there exists an integer $n>m$ such that any object $F\in\ct$ admits, in $\fl(\ct)$, a type-$m$ morphism $\Yon[\ct](F)\to D$ with respect to $(F,D_\bullet)$, where in the notation $(F,D_\bullet)$ the Cauchy sequence F is taken to be the constant sequence $F\stackrel{\id}{\la} F\stackrel{\id}{\la} F\stackrel{\id}{\la}\dots$, and $D= \colim
\Yon[\ct](D_i)\in\fs(\ct)\cap\cl_n^\perp$, for $\{\cl_i\}$ as in \eqref{eq:metLT}.
\end{enumerate}
\edfn

\exm{ex:excelmetriccompact}
Let $\ct$ be a weakly approximable triangulated category. Then the good metric defined in \autoref{ex:metricTsb} on $\tsb$ is excellent by \cite[Example 8.4]{Neeman25}. If $\ct$ is coherent, then \cite[Example 2.3]{Neeman25} shows that the metric on $\ct^c$ in \autoref{ex:metricTsb} is also excellent.
\eexm

Let $\ct$ be a weakly approximable triangulated
category, and let $\ct^c$ be the subcategory
of compact objects.
In \autoref{ex:metricTsb}(i)
we defined a metric on $\ct^c$, which seems to
depend on the embedding into $\ct$.
But \cite[Remark \ref*{R22.105.5} or Proposition \ref*{P22.109}]{Neeman18A}
proves that the equivalence class of this metric has intrinsic descriptions,
which use only the category $\ct^c$ and its
triangulated structure. In view of this fact,
the following lemma tells us that 
$\ct$ being coherent can be checked entirely inside $\ct^c$.

\lem{excellencebycompacts}
Let $\ct$ be a weakly approximable triangulated
category. The triangulated category $\ct$ is
coherent if and only if the metric on
$\ct^c$ of \autoref{ex:metricTsb}(i) is excellent.
\elem

\prf
As mentioned in \autoref{ex:excelmetriccompact},
if the category $\ct$ is coherent then the
metric on $\ct^c$ is known to be excellent. We need to
prove the converse.

Assume therefore that the metric on
$\ct^c$ is excellent. Then of course it is
very good, and by its
definition in \autoref{ex:metricTsb}(i)
we have that $\cm_i=\ct^c\cap\ct^{\leq-i}$
is closed in $\ct^c$
under direct summands. There is
a choice of metric in the equivalence class
closed under direct summands.
We also know that the embedding
$\ct^c\la\ct$ is a good extension with respect
to the metric, and
\cite[Lemma~9.5]{Neeman24} tells us
that $\fl'(\ct^c)=\ct^-_c$.
But the metrics on $\ct^c$ is
excellent, and
\cite[Proposition \ref*{P94.3}]{Neeman25}
applies.
For the integer $m=1$, there must
exist an integer $N>0$
such that any object $F'\in\ct^c$
admits a distinguished triangle
$E'\la F'\la D$
with $E'\in\ct^-_c\cap\ct^{\leq-1}$ and
with $D\in\ct^b_c\cap\ct^{\geq-N}$.
Now any object $F\in\ct^-_c$, by the
definition of $\ct^-_c$, admits a distinguished triangle
$F'\la F\la D'$, with $F'\in\ct^c$
and with $D'\in\ct^-_c\cap\ct^{\leq-N-2}$, with $N$ the
integer above.
Starting with the object $F\in\ct^-_c$,
we first form the triangle
$\sh[-1]D'\la F'\la F\la D'$,
and then, observing that
$F'\in\ct^c$, we form the triangle
$E'\la F'\la D$. The composite
$\sh[-1]D'\la F'\la D$ is a morphism from
$\sh[-1]D'\in\ct^{\leq-N-1}$ to $D\in\ct^{\geq-N}$
and must vanish, and the triangle 
$\sh[-1]D'\la F'\la F$  tells us that
the map $F'\la D$ must factor as the composite $F'\la F\la D$.

Completing the composable morphisms
$F'\la F\la D$ to an octahedron
in the triangulated category $\ct^-_c$ 
gives the diagram
\[\xymatrix{
\sh[-1]D'
\ar@{=}[d]
\ar[r]
&
E'
\ar[r]\ar[d]
&
E
\ar[d]\ar[r]
&
D'
\\
\sh[-1]D'
\ar[r]
&
F'
\ar[r]
\ar[d]
&
F
\ar[d]
\\
&
D
\ar@{=}[r]
&
D
}\]
By construction we know that $D\in\ct^{\geq-N}$,
and the distinguished triangle $E'\la E\la D'$,
coupled with the fact that 
$E'\in\ct^{\leq-1}$ and $D'\in\ct^{\leq-N-2}$, gives that
$E\in\ct^{\leq-1}$. The distinguished triangle
$E\la F\la D$, with $E\in\ct_c^-\cap\ct^{\leq-1}$ and
with $D\in\ct^b_c\cap\ct^{\geq-N}$,
therefore proves the coherence of
the triangulated category $\ct$.
\eprf

\dfn{D1.252525}
A good metric on a triangulated category $\ct$ is \emph{characteristic}
if any triangulated autoequivalence of $\ct$
takes the metric to an equivalent one in the sense of \autoref{def:ordermetric}.
\edfn

\exm{ex:charmetricTc}
If $\ct$ is a coherent weakly approximable triangulated category then the excellent metrics on $\ct^c$ and $\ct_c^b$, of \autoref{ex:metricTsb}, are characteristic by \cite[Remark \ref*{R22.105.5} or Proposition \ref*{P22.109}]{Neeman18A} and \cite[Proposition \ref*{P1.105}]{Neeman18A}, respectively.
\eexm

\lem{lem:charmetric}
Let $\ct$ be a weakly approximable
triangulated category. Then the metric $\{\cm_i\}$ on $\tsb$ in \autoref{ex:metricTsb} is characteristic.
\elem

\prf
If $G$ is any compact generator for $\ct$, there
exists an integer $A>0$ and inclusions of subcategories
of $\ct$
\[
\ct^{\leq-A}\subset G[0,\infty]^\perp\subset\ct^{\leq A}\ ;
\]
The first inclusion is part of the weak approximability
assumptions, and the second inclusion is by
\cite[Lemma~3.9(iv)]{Burke-Neeman-Pauwels18}.
Intersecting with $\tsb$ gives inclusions
\[
\cm_{i+A}\subset \tsb\cap G[-i,\infty]^\perp\subset\cm_{i-A}\ .
\]
Thus it suffices to check that any triangulated autoequivalence $\fF$
of $\tsb$ takes the metric
$\cn_i=\tsb\cap G[-i,\infty]^\perp$ to an
equivalent one.

But now $\fF$ must respect the subcategory $\big(\tsb\big)^c$,
and Corollary~\ref*{C0.15} tells
us that $\big(\tsb\big)^c=\ct^c$. Therefore
the classical generator $G\in\ct^c$ must go
to some other classical generator
$H\in\ct^c$, and hence $\fF$ takes
\[
\tsb\cap G[-i,\infty]^\perp\qquad\text{ to }
\qquad \tsb\cap H[-i,\infty]^\perp\ ,
\]
and these are clearly equivalent metrics.
\eprf

\subsection{Metric completion of enhancements}\label{subsec:restextenh}

Suppose we are given a pair of
triangulated categories $\cs\subset\ct$.
Back in \autoref{restrictenhancements}
we defined a restriction map
$\Phi^?_{\ct,\cs}\colon\Enhq(\ct)\la\Enhq(\cs)$,
and mentioned that we will be interested in
constructions that reverse this process---we
even gave an example in
\autoref{constrtenhTsb}.
In this section we will exhibit how
completion with respect to a metric can
also be used.

With this in mind we revisit, in the setting of stable $\infty$-categories, the constructions of \cite{Neeman25} that were briefly summarized in \autoref{subsec:metrics}.

\con{theconstructionL}
Let $\SSS\in\StabInfA$ and set $\cs:=\Hc(\SSS)$. Assume that we are given a very good metric
$\{\cm_i\}_{i\in\nn}$ on $\cs$ as in \autoref{def:excmetrics}.
Assume further that each of the subcategories $\cm_i\subset\cs$ is closed in $\cs$ under direct summands.

Consider the Yoneda embedding $\Yon\colon\SSS\hookrightarrow\IndC(\SSS)$. By \cite[Example~\ref*{E21.3}]{Neeman18A} the fully faithful functor
$\Hc(\Yon)\colon\Hc(\SSS)\la\Hc(\IndC(\SSS))$
is a good extension  with respect to the metric, in the sense of \autoref{def;goodext}. As in \eqref{eq:L'}, we can consider the full subcategory $\fl'(\cs)\subset\Hc(\IndC(\SSS))$, which is triangulated by \cite[Proposition~\ref*{P93.9}]{Neeman25}.
This allows us to form the stable $\infty$-category
\[
\LL(\SSS):=\Phi^?_{\Hc(\IndC(\SSS)),\fl'(\cs)}(\IndC(\SSS))\ ,
\]
with $\Phi$ as in \autoref{restrictenhancements}.
Recall that this simply means that
$\LL(\SSS)\subset\IndC(\SSS)$
is the full $\infty$-subcategory having for objects
$\fl'(\cs)\subset\Hc(\IndC(\SSS))$.
\econ

\rmk{L1.4}
The fact that every object of $\fl'(\cs)$ is isomorphic
to a homotopy colimit in $\Hc(\IndC(\SSS))$, of a Cauchy sequence
in $\cs$, can be reformulated in the stable $\infty$-category $\IndC(\SSS)$ as follows.
Given a Cauchy sequence
in $\cs=H^0(\SSS)$
we first lift it to $\SSS$, meaning
we choose in $\SSS$ a sequence
of  1-morphisms 
$E_1\stackrel{\varphi_1}{\la} E_2\stackrel{\varphi_2}{\la}\cdots E_n\stackrel{\varphi_n}{\la} E_{n+1}\la\cdots$,
that we denote by $E_\bullet$. This
sequence is such that
\begin{itemize}
\item[{\rm (a)}] The image in $\cs\iso\Hc(\SSS)$ is a Cauchy sequence.
\end{itemize}
We then choose a lifting to
$\IndC(\SSS)$ of the
morphism 
\[
\xymatrix@C+30pt{
\ds\coprod_{i=1}^\infty\Yon(E_i)\ar[r]^-{\id-\shi} &
\ds\coprod_{i=1}^\infty\Yon(E_i)\ .
}
\]
In other words, the morphism above
is well-defined in the additive category
$\Hc(\IndC(\SSS))$, and we choose a 1-morphism
in $\IndC(\SSS)$ lifting it.
And then
\begin{itemize}
\item[{\rm (b)}]
The category $\LL(\SSS)$ has for objects
all homotopy cofibers of
1-morphisms $(\id-\shi)$ as above.
\end{itemize}
\ermk

\rmk{rmk:invL}
Suppose we are given a pair of objects
$\SSS,\SSS'\in\HoStabInfA$, as well as a
very good metric $\{\cm_i\}$
on $\Hc(\SSS)$ and a very good metric $\{\cm'_i\}$ on $\Hc(\SSS')$, with $\cm_i$ and
$\cm'_i$ all closed under direct summands.
It is clear from the definition that any isomorphism
$\ph:\SSS\iso\SSS'$ in $\HoStabInfA$, respecting the metrics up to equivalence, induces an isomorphism $\LL(\ph):\LL(\SSS)\iso\LL(\SSS')$ in $\HoStabInfA$.
The hypothesis that $\ph$ \emph{respects the metric up to equivalence} means
that $\Hc(\ph):\Hc(\SSS)\iso\Hc(\SSS')$ takes
the metric $\{\cm_i\}$ on $\Hc(\SSS)$ to one equivalent to the metric $\{\cm'_i\}$ on
$\Hc(\SSS')$.
\ermk

\rmk{rmk:thedefinitionisok}
Let $\cs$ be a triangulated category,
and let $\{\cm_i\}$ be a characteristic
metric on $\cs$ which is very good,
and where the $\cm_i$ are
closed in $\cs$ under direct summands.
By \autoref{rmk:invL} we obtain a map
\[
\LL:\Enhq(\cs)\la\Enhq(\Box),
\]
taking an isomorphism class $[\SSS]$, of enhancements of the
metric triangulated category
$\cs$, to the isomorphism class $[\LL(\SSS)]$. To repeat in slightly different words: given $\SSS$ and a characteristic, very good metric on $\cs=\Hc(\SSS)$, the construction above provides an enhancement for the triangulated category $\Box=\Hc\big(\LL(\SSS)\big)=\fl'(\cs)$. But there is no reason to expect
the triangulated category
$\Hc\big(\LL(\SSS)\big)$ to be independent of the enhancement $\SSS$ of
the metric triangulated category $\cs$.
Hence there is no reason to expect that there is only one triangulated category
$\Box=\Hc\big(\LL(\SSS)\big)$ in the output of the map $\LL$ above, as we vary the isomorphism class of enhancements for $\cs$.
\ermk

Still assuming that $\cs$ is a triangulated category with a very good metric $\{\cm_i\}_{i\in\nn}$ such that each of the subcategories $\cm_i\subset\cs$ is closed in $\cs$
under direct summands, let $\fG\colon\cs\la\ct$ be a good extension
with respect to the metric. Assume further that the subcategory $\cs\hookrightarrow\ct$ is characteristic and the metric is characteristic on $\cs$.
Under these assumptions there are well-defined maps $\Phi^?_{\ct,\cs}\colon\EnhA(\ct)\la\EnhA(\cs)$,
and $\LL\colon\EnhA(\cs)\la\EnhA(\Box)$.
It is natural to wonder what, if anything, we can say about the composite. This discussion begins in the next section.

\subsection{Comparing completions}\label{subsec:keyres}

We are now ready to state and prove the main technical result of this section; it  will be applied twice, first in the proof of \autoref{P15.5} and again in the proof of \autoref{P4.5}. We begin by setting up the notation.

\begin{setting}\label{sett:1}
Let $\cs$ be a triangulated category with a very good metric
$\{\cm_i\}_{i\in\NN}$ such that each $\cm_i\subset\cs$ is closed in $\cs$
under direct summands. Let $\fG\colon\cs\la\ct$ be a good extension with respect to
the metric.
\end{setting}

Suppose that $\TT$ is an enhancement of $\ct$ in $\StabInfA$ and let $\SSS:=\Phi^?_{\ct,\cs}(\TT)$. This yields two good extensions of $\cs$ with
respect to the metric $\{\cm_i\}$, and a map between them. First there
is the given good extension $\fG\colon\cs\la\ct\iso\Hc(\TT)$
of \autoref{sett:1},
but there also is the good extension
$\cs\iso\Hc(\SSS)\la\Hc(\IndC(\SSS))$ that comes
from the fact that the essential image of the Yoneda embedding
$\Hc(\Yon)\colon\Hc(\SSS)\la\Hc(\IndC(\SSS))$
is contained in $\Hc(\IndC(\SSS))^c$, see
\cite[Example~\ref*{E21.3}]{Neeman18A}.

Let us denote the two ``completions'' of $\cs$
with respect to the metric $\{\cm_i\}$, defined in \eqref{eq:L'},
by $\fl'_1(\cs)\subset\ct$ and
$\fl'_2(\cs)\subset\Hc(\IndC(\SSS))$. Both are triangulated subcategories.


\pro{L1.5}
In \autoref{sett:1},
the restricted Yoneda functor $\YYon\colon\TT\la\IndC(\SSS)$
has the property that $\Hc(\YYon)\colon\Hc(\TT)\la\Hc(\IndC(\SSS))$
restricts to a triangle equivalence
$\fl'_1(\cs)\la\fl'_2(\cs)$, which yields an isomorphism
\[
\YYon'\colon\Phi^?_{\ct,\fl'_1(\cs)}(\TT)\la\Phi^?_{\Hc(\IndC(\SSS)),\fl'_2(\cs)}\big(\IndC(\SSS)\big)=\LL(\SSS)
\]
in $\HoStabInfA$.
\epro

\prf
Take any object $T\in\fl'_1(\cs)$, and express it
in $\ct$ as $\hoco\fG(S_\bullet)$ for a Cauchy sequence
$S_\bullet$ in $\cs$. Because $\cs=\Hc(\SSS)$, we
can choose in $\SSS$ a sequence $E_\bullet$ of 1-morphisms lifting $S_\bullet$ as in \autoref{L1.4}.

On the other hand the coproduct
$Y:=\coprod_{i\in\nn}\fG(S_i)$ exists in $\ct$. Set $X_i:=\fG(S_i)$. Interpreting $X_i$ and $Y$ as objects in $\TT$
we may choose, for each $i\in\nn$, a 
1-morphism $g_i\colon X_i\la Y$ in $\TT$ lifting the inclusion
of $X_i$ into $Y$ in $\ct$. Also the map $(\id-\shi)\colon Y\la Y$ in $\ct=\Hc(\TT)$ must lift
to a 1-morphism $f\colon Y\la Y$ in $\TT$. This gives 
us in $\IndC(\SSS)$ a 2-commutative square
\[
\xymatrix@C+30pt{
\ds\coprod_{i=1}^\infty \YYon(X_i)\ar[r]^-{h}\ar[d]_-\alpha &
\ds\coprod_{i=1}^\infty \YYon(X_i) \ar[d]^-\alpha\\
\YYon(Y) \ar[r]^-{\YYon(f)} &
\YYon(Y),
}
\]
where $\alpha$ is the natural map induced by the $g_i$ above and $h$ stands for a 1-morphism in $\IndC(\SSS)$ lifting the 
morphism $\id-\shi$ in $\Hc(\IndC(\SSS))$.
Let $M\in\TT$ be the mapping cone of $f\colon Y\la Y$, and let
$P\in\LL(\SSS)$ be the mapping
cone of $\id-\shi$. The data
permits us to construct
a morphism of
cofiber sequences in $\IndC(\SSS)$
\begin{equation}\label{eq:dagger}
\xymatrix@C+30pt{
\ds\coprod_{i=1}^\infty \YYon(X_i)\ar[r]^-{h}\ar[d]_-\alpha &
\ds\coprod_{i=1}^\infty \YYon(X_i) \ar[d]^-\alpha\ar[r] &
P\ar[d]^\beta&\\
\YYon(Y) \ar[r]^-{\YYon(f)} &
\YYon(Y)\ar[r] & \YYon(M).&
}
\end{equation}
Now let $A\in\cs$ be any object
and apply to this diagram the functor
$\Hom_{\Hc(\IndC(\SSS))}(\YYon(\fG(A)),-)$. The fact that $\Hc(\Yon)(\cs)\subset\Hc(\IndC(\SSS))^c$
tells us that this agrees with the commutative diagram
of \autoref{explanationofgoodextension},
and in particular $\Hom_{\Hc(\IndC(\SSS))}(\YYon(\fG(A)),-)$ takes the
map $\beta$ to an isomorphism. Warning: because
we are not assuming any compactness of the objects
of
$\cs$ in the category $\ct$, the
functor $\YYon$ need not
take $\alpha$ to an isomorphism. Anyway, because
$\Hom_{\Hc(\IndC(\SSS))}(\YYon(\fG(A)),-)$ takes the
map $\beta$ to an isomorphism
for all $A\in\cs$, we have that 
the map $\beta\colon P\la\YYon(M)$
is such that $\Hc(\beta)$ is an isomorphism,
meaning $\beta$ is a quasi-isomorphism.
From this it easily
follows that the
functor $\Hc(\YYon)\colon\Hc(\TT)\la\Hc(\IndC(\SSS))$
takes the subcategory $\fl'_1(\cs)\subset\Hc(\TT)$
into the subcategory $\fl'_2(\cs)\subset\Hc(\IndC(\SSS))$,
and the induced functor
$\fl'_1(\cs)\la\fl'_2(\cs)$ is essentially
surjective.

It remains to establish that, for any two objects $M,M'\in\TT$
both
lying in $\fl'_1(\cs)$, 
the natural map
$\Hom_\TT^{}(M,M')\la \Hom_{\IndC(\SSS)}^{}\big(\YYon(M),\YYon(M')\big)$
is a quasi-isomorphism.
And for this the restriction on $M'$ is
unnecessary: we will show that
this map is a quasi-isomorphism for every
$M'\in\TT$.

Now the object $M\in\LL(\SSS)\subset\TT$ comes
with the cofiber sequence $Y\stackrel f\la Y\la M$ in
$\TT$, and the functor $\YYon$ yields a morphism of cofibrations
\[
\scalebox{0.97}{
\xymatrix@C+20pt{
\Hom_\TT^{}(M,M')\ar[r]\ar[d] &
\Hom_\TT^{}(Y,M')\ar[r]^-\ph\ar[d] &
\Hom_\TT^{}(Y,M')\ar[d]
\\
\Hom\big(\YYon(M),\YYon(M')\big) \ar[r] &
\Hom\big(\YYon(Y),\YYon(M')\big) \ar[r]^-{\YYon(f)} &
\Hom\big(\YYon(Y),\YYon(M')\big)
}}
\]
If we apply the functor $\Hom\big(-,\YYon(M')\big)$,
to the map of cofiber sequences in $\eqref{eq:dagger}$, this extends
to a composable
pair of morphisms of cofibrations
\[
\scalebox{0.9}{
\xymatrix@C+20pt{
\Hom_\TT^{}(M,M')\ar[r]\ar[d] &
\Hom_\TT^{}(Y,M')\ar[r]^-{\varphi}\ar[d] &
\Hom_\TT^{}(Y,M')\ar[d]
\\
\Hom\big(\YYon(M),\YYon(M')\big) \ar[r]\ar[d]^{\Hom(\beta,\YYon(M'))} &
\Hom\big(\YYon(Y),\YYon(M')\big) \ar[r]^-{\YYon(f)}\ar[d]^{\Hom(\alpha,\YYon(M'))} &
\Hom\big(\YYon(Y),\YYon(M')\big)\ar[d]^{\Hom(\alpha,\YYon(M'))}
\\
\Hom\big(P,\YYon(M')\big) \ar[r] &
\ds\prod_{i=1}^\infty\Hom\big(\YYon(X_i),\YYon(M')\big) \ar[r]^-{\id-\shi} &
\ds\prod_{i=1}^\infty\Hom\big(\YYon(X_i),\YYon(M')\big)
}}
\]
and the composite is
\[
\scalebox{0.96}{
\xymatrix@C+20pt{
\Hom_\TT^{}(M,M')\ar[r]\ar[d]^{\beta'} &
\Hom_\TT^{}(Y,M')\ar[r]^-\ph\ar[d]^-{\alpha'} &
\Hom_\TT^{}(Y,M')\ar[d]^-{\alpha'}
\\
\Hom\big(P,\YYon(M')\big) \ar[r] &
\ds\prod_{i=1}^\infty\Hom\big(\YYon(X_i),\YYon(M')\big) \ar[r]^-{\id-\shi} &
\ds\prod_{i=1}^\infty\Hom\big(\YYon(X_i),\YYon(M')\big).
}}
\]
As $Y$ is the coproduct in $\ct$ of the objects
$X_i=\fG(S_i)$ with $S_i\in\cs\subset\ct$,
the map $\alpha'$ must be a quasi-isomorphism,  and hence
so is $\beta'$. But $\beta'$ is the composite 
\[
\scalebox{0.97}{
\xymatrix{
\Hom_\TT^{}(M,M')\ar[r]\ar@/^3pc/[rrrr]^-\sim
& \Hom_{\IndC(\SSS)}^{}\big(\YYon(M),\YYon(M')\big)
\ar[rrr]^-{\Hom(\beta,\YYon(M'))} &&& \Hom_{\IndC(\SSS)}^{}\big(P,\YYon(M')\big).
}}
\]
On the other hand $\beta\colon P\la\YYon(M)$ is a
quasi-isomorphism in $\IndC(\SSS)$ by the previous discussion. The
full faithfulness of $\fl'_1(\cs)\la\fl'_2(\cs)$
immediately follows.
\eprf

\subsection{Good metrics and enhancements}
\label{S15}

The first application is to specialize \autoref{L1.5}
to the situation where $\fl'(\cs)=\ct$. We need now to update \autoref{sett:1} as follows:

\begin{setting}\label{sett:2}
Assume we are in \autoref{sett:1}, and suppose further that:
\be
\item
The inclusion $\fl'(\cs)\hookrightarrow\ct$ is an equality;
\item
The subcategory $\cs\subset\ct$ is
characteristic;
\item
The metric $\{\cm_i\}_{i\in\nn}$,
on the category $\cs$, is
characteristic.
\ee
\end{setting}

The result is then the following.

\pro{P15.5}
with the assumptions as in \autoref{sett:2},
the map of \autoref{restrictenhancements}
\[
\Phi^?_{\ct,\cs}\colon\Enhq(\ct)\la\Enhq(\cs)
\]
is a monomorphism. And more precisely:
with $\LL$ the map of \autoref{theconstructionL}
we have that the composite
\[\xymatrix@C+20pt{
\Enhq(\ct)
\ar[r]^-{\Phi^?_{\ct,\cs}}
&
\Enhq(\cs)
\ar[r]^-{\LL}
&
\Enhq(\Box)
}\]
takes to itself any isomorphism class $[\TT]$, of
objects
$\TT\in\HoStabInfA$ satisfying $\Hc(\TT)\iso\ct$.
\epro

\prf
First of all, since $\cs\subset\ct$ is a characteristic subcategory,
\autoref{R1.25} guarantees that $\Phi^?_{\ct,\cs}$ gives
a well-defined map $\Phi^?_{\ct,\cs}\colon\Enhq(\ct)\la \Enhq(\cs)$.
And because the metric
$\{\cm_i\}$
on the category $\cs$ is
assumed characteristic,
\autoref{rmk:thedefinitionisok}
tells us that the $\LL$
of \autoref{theconstructionL}
yields a well-defined
map $\LL\colon\Enhq(\cs)\la\Enhq(\Box)$.
It follows that the composite
\[\xymatrix@C+20pt{
\Enhq(\ct)\ar[r]^-{\Phi^?_{\ct,\cs}} & \Enhq(\cs)\ar[r]^-\LL & \Enhq(\Box)
}\]
is well-defined. But by \autoref{sett:2} (i)
we are assuming 
$\fl'_1(\cs)=\ct$, and \autoref{L1.5} produces
a (explicit) isomorphism
$\TT\la\LL(\Phi^?_{\ct,\cs}(\TT))$ in $\HoStabInfA$. This proves the
second assertion of the statement, and the first is
an immediate consequence of the second.
\eprf

\rmk{actualmapagain}
As in \autoref{actualmap} we
note that
the proof shows a little more than
the statement of 
\autoref{lem:uniqenhTsb}
asserts: given any $\TT$ with $\Hc(\TT)\iso\ct$,
the proof produces a morphism
$\TT\la\LL(\Phi^?_{\ct,\cs}(\TT))$
in the category $\StabInfA$, and proves
that it is an isomorphism in $\HoStabInfA$.
What is more: this map is natural in $\TT$,
it is nothing more than a restriction of
the Yoneda map
$\YYon\colon\TT\la\IndC(\Phi^?_{\ct,\cs}(\TT))$.

And once again this naturality will play a role in
\autoref{subsect:autoequiv}.
\ermk

We draw attention to the following easy consequence of the result above.

\cor{C1.9}
In \autoref{sett:2},
if $\cs$ has a unique (linear) enhancement, then $\ct$ has at most one (linear) enhancement.
\ecor

\rmk{rmk:back}
Looking more carefully at \autoref{C1.9}, we note that \autoref{P15.5} even
permits us to go backwards and recover,
up to isomorphism in $\HoStabInfA$, both the triangulated category
$\ct$ and its unique enhancement from the triangulated category
$\cs$ and its metric.
Precisely, if $\cs$ has a unique (linear) enhancement $\SSS$ and $\ct=\fl'(\cs)$ has some (linear) enhancement, then $\ct$ and its enhancement are
uniquely determined by the formula
$\ct=\Hc(\LL(\SSS))$ having $\LL(\SSS)$ as its unique enhancement.
\ermk

\section{The restriction functor for characteristic subcategories}
\label{S2}

This short section is about applications of \autoref{P15.5}. Suppose $\ct$ is a weakly approximable triangulated category  and $\ca\subset\cb$ is a pair of the natural  subcategories of $\ct$ in the diagram \eqref{eq:incl}. It turns out that, with suitable choices of metrics on $\ca$,  there are five pairs $(\ca,\cb)$ to which \autoref{P15.5} applies. In  four of the cases the proof is essentially identical, and we will give a unified argument in \autoref{subsec:genappl1}. And then \autoref{subsec:applwa} explains in detail how the general argument specializes to the four pairs $(\ca,\cb)$ we mentioned, and then show how to apply \autoref{P15.5} to the fifth pair.

\subsection{General results}\label{subsec:genappl1}

We begin by fixing the following:

\begin{setting}\label{sett:3}
Let $\ct$ be a triangulated category with countable
coproducts, and let $\tst\ct$ be a
nondegenerate \tstr\ on $\ct$.
Assume further that:
\be
\item There exists an
integer $r>0$ such that, for any
countable set of
objects $\{X_i,\,i\in\nn\}$ contained
in $\ct^{\geq0}$, the coproduct
$\coprod_{i\in\nn}X_i$ belongs to $\ct^{\geq-r}$;

\item The \tstr\
$\tst\ct$ is characteristic, as is its
restriction to the subcategories $\ct^\star$, where $\star$ can be any of the symbols $-,+,b$.
\ee
\end{setting}

\exm{ex:appl1}
For any weakly approximable triangulated category $\ct$ we may choose a
compactly generated \tstrs $\tst\ct$ in the preferred equivalence class, and
all assumptions in \autoref{sett:3} are satisfied both by $\ct$ with $\tst\ct$ and by $\ct\op$ with the induced \tstr\ $\tst{\big(\ct\op\big)}$. After all, the category $\ct$, being compactly generated, contains all small coproducts and products of its objects.
Because the \tstr\ is compactly generated, \cite[Proposition~A.2]{Alonso-Jeremias-Souto03} tells us that
the subcategory $\ct^{\geq0}$ is closed in $\ct$ under coproducts. And by
\cite[Lemma~3.2]{Canonaco-Neeman-Stellari24}
there exists an integer $r>0$ such that, for any
collection $\{D_\lambda\}_{\lambda\in\Lambda}$
of objects in $\ct^{\leq 0}$, the product
$\prod_{\lambda\in\Lambda}D_\lambda$ belongs to $\ct^{\leq r}$.
Furthermore, in \cite[Propositions~5.3, 5.7 and 5.8]{Canonaco-Neeman-Stellari24}
the reader can find the facts that the
(respective) restrictions,
to $\ct^-$,
$\ct^+$ and $\ct^b$, of a \tstr\ in
the preferred equivalence class, are all
characteristic.
\eexm

Next we consider the following easy consequence of the results in the previous section.

\lem{E2.1}
Let $\ct$ be a triangulated category as in \autoref{sett:3}.
Then
\be
\item
The restriction map
$\Phi^?_{\ct,{\ct^-}}\colon\Enhq(\ct)\la\Enhq(\ct^-)$ is injective.
Furthermore, if $\ct^-$ has a unique enhancement
$\SSS$ then, assuming
$\ct$ has an enhancement,
the category $\ct$ and its unique enhancement
are given by
the formula $\ct=\Hc\big(\LL(\SSS)\big)$.
\item
The restriction map
$\Phi^?_{\ct^+,\ct^b}\colon\Enhq(\ct^+)\la\Enhq(\ct^b)$ is injective.
Furthermore, if $\ct^b$ has a unique enhancement
$\SSS$ then, assuming
$\ct^+$ has an enhancement,
the category $\ct^+$ and its unique enhancement
are given by
the formula $\ct^+=\Hc\big(\LL(\SSS)\big)$.
\ee
\setcounter{enumiv}{\value{enumi}}
\elem

\prf
Since the \tstr{s} are assumed to be
characteristic on all the subcategories,
and since the subcategories are defined in terms
of the \tstr{s}, in the commutative square
\[\xymatrix{
\ct^b\ar@{^{(}->}[r] \ar@{^{(}->}[d] &
\ct^- \ar@{^{(}->}[d] \\
\ct^+\ar@{^{(}->}[r] &\ct
}\]
the inclusions are all characteristic.
Next we apply \autoref{P15.5}. More precisely:
\be
\setcounter{enumi}{\value{enumiv}}
\item
To prove (i) we
let $\cs=\ct^-$, let
the metric be given by
$\cm_i=\ct^-\cap\ct^{\geq i}$,
and let the good extension be the
inclusion $\cs=\ct^-\la\ct$.
From \cite[Example~7.9(i)]{Neeman25}
we learn that this is a good extension with $\fl'(\cs)=\ct$.
\item
To prove (ii) we
let $\cs=\ct^b$, let
the metric be given by
$\cm_i=\ct^b\cap\ct^{\geq i}$,
and let the good extension be the
inclusion $\cs=\ct^b\la\ct^+$.
From \cite[Example~7.9(ii)]{Neeman25}
we learn that this is a good extension with $\fl'(\cs)=\ct^+$.
\ee
It is clear 
that the hypotheses of \autoref{P15.5}
all hold, and the Lemma follows.
\eprf

\subsection{The weakly approximable case}\label{subsec:applwa}

Let us specialize considerably: assume now that $\ct$ is a
weakly approximable triangulated category. By \autoref{ex:appl1} we may apply \autoref{E2.1}  both to $\ct$ with the \tstr\ $\tst\ct$, and to $\ct\op$
with the induced \tstr\ $\tst{\big(\ct\op\big)}$, and deduce the following:

\pro{E2.5}
Let $\ct$ be a weakly approximable triangulated category.
\be
\item
The restriction map
$\Phi^?_{\ct,\ct^-}\colon\Enhq(\ct)\la\Enhq(\ct^-)$ is injective. Furthermore, if $\ct^-$ has a unique enhancement
$\SSS$ then, assuming
$\ct$ has an enhancement,
the category $\ct$ and its unique enhancement
are given by
the formula  $\ct=\Hc\big(\LL(\SSS)\big)$.
\item
The restriction map
$\Phi^?_{\ct^+,\ct^b}\colon\Enhq(\ct^+)\la\Enhq(\ct^b)$ is injective.
Furthermore, if $\ct^b$ has a unique enhancement
$\SSS$ then, assuming
$\ct^+$ has an enhancement,
the category $\ct^+$ and its unique enhancement
are given by
the formula  $\ct^+=\Hc\big(\LL(\SSS)\big)$.
\item
The restriction map
$\Phi^?_{\ct,\ct^+}\colon\Enhq(\ct)\la\Enhq(\ct^+)$ is injective.
Furthermore, if $\ct^+$ has a unique enhancement
$\SSS$ then, assuming
$\ct$ has an enhancement,
the category $\ct$ and its unique enhancement
are given by
the formula  $\ct\op=\Hc\big(\LL(\SSS\op)\big)$.
\item
The restriction map
$\Phi^?_{\ct^-,\ct^b}\colon\Enhq(\ct^-)\la\Enhq(\ct^b)$ is injective.
Furthermore, if $\ct^b$ has a unique enhancement
$\SSS$ then, assuming
$\ct^-$ has an enhancement,
the category $\ct^-$ and its unique enhancement
are given by
the formula
$\big(\ct^-\big)\op=\Hc\big(\LL(\SSS\op)\big)$.
\ee
\epro

We can also prove the following.

\pro{cor:applwaTsb}
If $\ct$ is a weakly approximable triangulated category, then the restriction map
$\Phi^?_{\ct^-,\tsb}\colon\Enhq(\ct^-)\la\Enhq(\tsb)$ is injective. Assume that $\tsb$ has a unique enhancement $\RR$ and that $\ct^-$ has an enhancement. Then the latter is unique,
and both $\ct^-$ and
its unique enhancement  
can be recovered by the formula
$\ct^-=\Hc\big(\LL(\RR)\big)$.
\epro

\prf
From \autoref{ex:excelmetriccompact},
we know that the subcategory $\tsb\subset\ct^-$,
together with the metric $\{\cn\op_i\}_{i\in\nn}$ given by
the formula
\[
\cn\op_i=\tsb\cap\ct^{\leq-i}\ ,
\]
is a triangulated category with an excellent metric.
The formula of \autoref{D0.7} makes it clear
that $\tsb\subset\ct^-$ is characteristic,
and in \autoref{lem:charmetric} we saw
that the metric on $\tsb$ is characteristic.
Furthermore, \cite[Examples~\ref*{E96.7}]{Neeman25}
proves that the embedding $\tsb\subset\ct^-$ is
a good extension, with $\fl'\big(\tsb\big)=\ct^-$.
And the fact that $\cn\op_i=\tsb\cap\ct^{\leq-i}$ is
closed in $\tsb$ under direct summands is
immediate
from the formula. Thus,
\autoref{C1.9} and \autoref{rmk:back} imply the result.
\eprf

\section{Triangulated categories with excellent metrics, and their enhancements}
\label{S4}

The last important technical tool of
this atricle is the enhanced version of
passing, from a triangulated category $\cs$
with a good metric $\{\cm_i\}$, to the triangulated
category $\fs(\cs)$ with the induced good metric
$\{\cn_i\}$. This procedure takes an object
in $\SSS\in\StabInfA$ with $\Hc(\SSS)\iso\cs$ to
an object $\Psi(\SSS)\in\StabInfA$
with
$\Hc(\Psi(\SSS))\iso\fs(\cs)$.
We will meet
it in \autoref{theconstructionpsi}
and study its basic properties in
\autoref{subsec:constr} and
\autoref{subsec:backforth}.

This completes our toolkit. We remind the reader
that we already have
\be
\item
The restriction map $\Phi^?_{\ct,\cs}$ of
\autoref{restrictenhancements},
which takes a pair of triangulated categories
$\cs\subset\ct$ and an object
$\TT\in\StabInfA$ with $\Hc(\TT)\iso\ct$,
to the object
$\Phi^?_{\ct,\cs}(\TT)\in\StabInfA$ with $\Hc(\Phi^?_{\ct,\cs}(\TT))\iso\cs$.
\item
The map $\Delta$
of \autoref{constrtenhTsb},
which takes a triangulated
category $\cs$ with a classical generator,
and an object
$\SSS\in\StabInfA$ with $\Hc(\SSS)\iso\cs$,
to the object $\Delta(\SSS)\in\StabInfA$.
We remind the reader that if $\cs=\ct^c$
then, under suitable hypotheses,
\autoref{lem:uniqenhTsb} proved that
$\Hc(\Delta(\SSS))\iso\tsb$.
\item
If
$\cs$
is a triangulated category
with a very good metric $\{\cm_i\}$,
and if each $\cm_i$ is closed in
$\cs$ under direct summands, then
there is also the
$\LL$ of
\autoref{theconstructionL}.
It takes an
object
$\SSS\in\StabInfA$ with $\Hc(\SSS)\iso\cs$
to the object $\LL(\SSS)\in\StabInfA$.
And $\Hc(\LL(\SSS))$ is some sort of
``best possible'' $\fl'(\cs)$,
in the sense of \autoref{L1.5}.
By this we mean that given any good extension
$\cs\subset\ct$ with respect to the
metric $\{\cm_i\}$, we can form in $\ct$
the subcategory $\fl'(\cs)$; under the
hypotheses on the metric the subcategory
$\fl'(\cs)$ must be triangulated. The construction
seems to depend on the choice of good embedding
$\cs\subset\ct$. But 
\autoref{theconstructionL}
tells us that, as long as $\ct$ has an enhancement,
then its triangulated subcategory $\fl'(\cs)\subset\ct$
must be equivalent to $\Hc(\LL(\SSS))$.
\ee
After setting up this extensive toolkit, in
\autoref{subsec:waen}
and \autoref{subsec:lifting}
we bring it all together and prove the main theorems
of the introduction.

\subsection{Passing from enhancements of $\cs$ to enhancements of $\fs(\cs)$}\label{subsec:constr}

As we explained in \autoref{subsec:metrics}, given a triangulated category $\cs$ with a good metric $\{\cm_i\}_{i\in\nn}$,
we may form the category $\fs(\cs)$ as defined in \eqref{eq:sigma},
and it is triangulated by \autoref{rmk:Stria}. Assume now that $\SSS$
is an enhancement of $\cs$ in $\HoStabInfA$. We can then take the good extension
$\Hc(\Yon)\colon\Hc(\SSS)\la\Hc(\IndC(\SSS))$ with respect to the given metric. Next, take in $\Hc(\IndC(\SSS))$ the subcategory
$\fl'(\cs)$ as in \eqref{eq:L'},
and then we obtain a triangle equivalence \cite[Corollary~\ref*{C21.15}]{Neeman18A}
\[
\fs(\cs)\iso\fl'(\cs)\cap\YYon^{-1}\big(\fc(\cs)\big),
\]
where $\YYon\colon\Hc(\IndC(\SSS))\la\MMod\cs$ is the restricted Yoneda functor.
Note that we appeal to \cite[Theorem~\ref*{T20.17}]{Neeman18A}
to show that the equivalence
of \cite[Corollary~\ref*{C21.15}]{Neeman18A}
respects
the triangulated structure. Following the abuse of notation
established in \cite[Remark~\ref*{R93.-11}]{Neeman25},
we will write $\fs(\cs)$ both for the
subcategory of the abelian category $\MMod\cs$ and for $\fl'(\cs)\cap\YYon^{-1}\big(\fc(\cs)\big)$ in
$\Hc(\IndC(\SSS))$.

\con{theconstructionpsi}
With $\SSS$ an object of
$\HoStabInfA$ and given a good metric
$\{\cm_i\}$ on $\cs=\Hc(\SSS)$, we define
\[
\Psi(\SSS):=\Phi^?_{\Hc(\IndC(\SSS)),\fs(\cs)}(\IndC(\SSS))
\]
with the embedding $\fs(\cs)\subset\Hc(\IndC(\SSS))$
as in the discussion above.
\econ

\rmk{psiokforcharacteristic}
It is not hard to see that if the metric $\{\cm_i,\,i\in\nn\}$ is characteristic,
then the $\Psi$ of
\autoref{theconstructionpsi}
respects quasi-equivalences. Hence it
induces a map
\[
\Psi\colon\Enhq(\cs)\la\Enhq\big(\fs(\cs)\big)
\]
which takes the equivalence class $[\SSS]$, of
enhancements of $\cs$, to
the equivalence class $[\Psi(\SSS)]$ of
enhancements of $\fs(\cs)$. Note that this time the output triangulated
category is unambiguous;
$\Hc(\Psi(\SSS))$ is always triangle equivalent to
$\fs(\cs)$, irrespective of the enhancement $\SSS$.
\ermk

\rmk{Rreminerofinclusion}
Let $\cs$ be a triangulated category with a good
metric $\{\cm_i\}_{i\in\nn}$.
By \cite[Remark~\ref*{R200978}]{Neeman18A}
we know that the subcategory $\cs\cap\fs(\cs)$,
where the intersection is formed
inside $\MMod\cs$, is always a triangulated subcategory
inside each of $\cs$ and $\fs(\cs)$. Furthermore
the two triangulated structures
on $\cs\cap\fs(\cs)$, the one coming
from the embedding into $\cs$ and the one
coming from the embedding into $\fs(\cs)$,
always agree with each other.
It makes sense to ask that
$\cs\cap\fs(\cs)=\cs$, which we will write
as $\cs\subset\fs(\cs)$. And
\cite[Remark~\ref*{R29876}]{Neeman18A}
assures us that this condition can be tested
in
any good extension with respect to the metric.
\ermk

In the special case where $\cs\subset\fs(\cs)$
the map $\Psi$ above satisfies
the following, easy analog of
\autoref{L1.5}.

\lem{Leasyanalog}
Let $\cs$ be a triangulated category with a good
metric $\{\cm_i\}_{i\in\nn}$, and suppose further that
$\cs\subset\fs(\cs)$,
as in \autoref{Rreminerofinclusion}.
Let $\TT$ be an enhancement of $\fs(\cs)$
and let $\SSS=\Phi^?_{\fs(\cs),\cs}(\TT)$.
Then
the restricted Yoneda functor $\YYonInf\colon\TT\la\IndC(\SSS)$
factors through $\Psi(\SSS)\subset\IndC(\SSS)$,
and the map $\TT\la\Psi(\SSS)$ is an isomorphism
in $\HoStabInfA$.
\elem

\prf
It suffices to prove
the triangulated category
statements, namely that the
triangulated functor
$\Hc(\YYonInf)\colon\Hc(\TT)\la\Hc(\IndC(\SSS))$
factors through $\Hc(\Psi(\SSS))$,
and the induced map
$\Hc(\TT)\la\Hc(\Psi(\SSS))$ is an equivalence.

We first prove the factorization assertion.
As $\cs=\Hc(\SSS)$ is a subcategory
of $\Hc(\IndC(\SSS))$, it gives
an ordinary, unenriched restricted Yoneda
functor
$\YYon\colon\Hc(\IndC(\SSS))\la\MMod\cs$.
Now take an object $A\in\TT$.
By assumption $\TT$ is an enhancement for
$\fs(\cs)$, which means that
the natural Yoneda functor
$\Hc(\TT)\la\MMod\cs$ has image
$\fs(\cs)\subset\MMod\cs$.
Thus $\YYon\comp\Hc(\YYonInf)(A)$
is an object of $\fs(\cs)\subset\MMod\cs$.
Therefore there exists an object
$B\in\Psi(\SSS)$ with $\YYon(B)\iso \YYon\comp\Hc(\YYonInf)(A)$.
The fact that $B\in\cl'(\cs)$,
together with
\cite[Lemma~\ref*{L21.11}]{Neeman18A},
allows  us to lift this to a morphism
$\ph\colon B\la\Hc(\YYonInf)(A)$ in
$\Hc(\IndC(\SSS))$. Now $\YYon$ takes
$\ph$ to an isomorphism in $\MMod\cs$,
meaning that $\Hom(C,\ph)$ is an isomorphism
for all objects $C\in\cs$.
But $\cs$ generates the triangulated
category $\Hc(\IndC(\SSS))$;  we deduce that
the map $\ph$
must be an isomorphism in $\Hc(\IndC(\SSS))$. 
Hence
the Yoneda functor
$\YYonInf\colon\TT\la\IndC(\SSS)$
factors through $\Psi(\SSS)\subset\IndC(\SSS)$.

But
$\Psi(\SSS)$ and therefore its subcategory
 $\YYonInf(\TT)$ both lie
in the subcategory
$\fl'(\cs)\cap\YYonInf^{-1}\big(\fc(\cs)\big)$
of $\Hc(\IndC(\SSS))$, and on
this subcategory
the map $\YYon\colon\Hc(\IndC(\SSS))\la\MMod\cs$
restricts to a
triangle equivalence with $\fs(\cs)$;
see \cite[Corollary~\ref*{C21.15} and Theorem~\ref*{T20.17}]{Neeman18A}.
In other words we have the following commutative
triangle
\[\xymatrix{
\Hc(\TT)
\ar[rr]^-{\Hc(\YYonInf)}
\ar[dr]_-{\raisebox{+4pt}[0pt][0pt]{\rotatebox{-45}{$\sim$}}}
& &
\Hc(\Psi(\SSS))
\ar[dl]^-{\raisebox{-4pt}[0pt][0pt]{\rotatebox{+45}{$\sim$}}}
\\
&
\fs(\cs)
}\]
where the slanted arrows are equivalences, the
one on the left by the hypotheses of the Lemma and
the one on the right by
construction. Hence
the horizontal map is an equivalence.
\eprf

\cor{Corofeasyanalog}
Let $\cs$ be a triangulated category with a good
metric $\{\cm_i\}_{i\in\nn}$,
and also assume
\be
\item
The metric is characteristic.
\item
We have $\cs\subset\fs(\cs)$
as in \autoref{Rreminerofinclusion}.
\ee
Then the natural map
$\Psi\colon\Enhq(\cs)\la\Enhq(\fs(\cs))$
is surjective. If we further assume that
$\cs\subset\fs(\cs)$ is a
characteristic subcategory then the
map $\Psi$ is bijective, with the inverse
given by
$\Phi^?_{\fs(\cs),\cs}\colon\Enhq(\fs(\cs))\la \Enhq(\cs)$.
\ecor

\prf
We are assuming that the metric is characteristic,
and hence the map 
$\Psi\colon\Enhq(\cs)\la\Enhq(\fs(\cs))$
is well-defined (see \autoref{psiokforcharacteristic}). From \autoref{Leasyanalog}
we know that any enhancement $\TT$ of
$\fs(\cs)$ is equivalent to $\Psi(\SSS)$ for
some $\SSS\in\Enhq(\cs)$, proving the
surjectivity.

If $\cs\subset\fs(\cs)$ is characteristic then the map
$\Phi^?_{\fs(\cs),\cs}\colon\Enhq(\fs(\cs))\la \Enhq(\cs)$
is also well-defined. The composite
$\Phi^?_{\fs(\cs),\cs}\comp\Psi$ takes an
enhancement $\SSS$ of $\cs$ to
$\Phi^?_{\Hc(\IndC(\SSS)),\cs}\IndC(\SSS)=\SSS$, and hence
is the identity. This implies that $\Psi$ is also injective. 
\eprf

\rmk{actualmapagainagain}
As in \autoref{actualmap}
and in \autoref{actualmapagain}, we
note that
the proof shows a little more than
the statement of 
\autoref{Corofeasyanalog}
asserts: given any $\TT$ with $\Hc(\TT)\iso\fs(\ct)$,
the proof produces a morphism
$\TT\la\Psi\comp\Phi^?_{\fs(\cs),\cs}(\TT)$
in the category $\StabInfA$, and proves
that it is an isomorphism in $\HoStabInfA$.
What is more: this map is natural in $\TT$,
it is nothing more than a restriction of
the Yoneda map
$\YYon\colon\TT\la\IndC(\Phi^?_{\fs(\cs),\cs}(\TT))$.

And once again this naturality will play a role in
\autoref{subsect:autoequiv}.
\ermk

Next we drop the assumption that
$\cs\subset\fs(\cs)$, but restrict the
class of allowed metrics. The following lemma will
play a key role in understanding this situation.

\lem{L4.3}
Let $\cs$ be a triangulated category with an excellent
metric $\{\cm_i\}_{i\in\nn}$, and suppose that we are given
an enhancement $\SSS$ of $\cs$ in $\HoStabInfA$.
Then the following is true:
\be
\item
The embedding 
$\fs(\cs)\op=\Hc\big(\Psi(\SSS)\big)\op\la\Hc(\IndC(\SSS))\op$
is a good extension with respect to the
metric $\{\cn\op_i\}_{i\in\nn}$ on $\fs(\cs)\op$, where $\cn_i$ is as in \eqref{eq:metrN}.
\item
As full subcategories
of $\Hc(\IndC(\SSS))\op$,
we have an equality
\[
\fl'(\cs)\op=\fl'\big(\fs(\cs)\op\big).
\]
\setcounter{enumiv}{\value{enumi}}
\ee
\elem

\prf
The compactly generated triangulated category
$\Hc(\IndC(\SSS))$ has products, and hence
the opposite category
$\Hc(\IndC(\SSS))\op$ has coproducts. Thus
we can form in $\Hc(\IndC(\SSS))\op$
all homotopy colimits and, in particular,
the homotopy colimits of Cauchy sequences
in $\fs(\cs)\op$.

To prove (i), it remains to show that the requirement in \autoref{def;goodext} is satisfied, as in \autoref{explanationofgoodextension}. And the proof of this fact,
as well as the proof of (ii) above, both require us to understand
the homotopy colimits,
in $\Hc(\IndC(\SSS))\op$,
of Cauchy sequences in $\fs(\cs)\op$.

To this end recall that, given any Cauchy sequence
$B_\bullet$ in $\fs(\cs)\op$, the homotopy colimit
is obtained by forming
in $\Hc(\IndC(\SSS))$ the distinguished triangle
\[
\xymatrix{
\holim B_\bullet\ar[r]^-\ph &
\ds\prod_{i=1}^\infty B_i\ar[rr]^-{\id-\shi} &&
\ds\prod_{i=1}^\infty B_i\ar[r] &
\sh{\holim B_\bullet}.
}\]
Note that we are using here that the stable presentable $\infty$-category $\IndC(\SSS)$ is closed under small limits (and thus small products). For this see \cite[Proposition 5.5.2.4]{Lurie09}.

The next step is to apply the restricted Yoneda functor $\YYon\colon\Hc(\IndC(\SSS))\la\MMod\cs$,
which amounts to studying
what happens when we apply 
$\Hom_{\Hc(\IndC(\SSS))}(A,-)$ to the triangle, with $A\in\cs$.
By \cite[Lemma~12.10(i)]{Neeman18A}, for $A\in\cs$ and $B_\bullet$
a Cauchy sequence in $\fs(\cs)\op$, we have 
$\climone \Hom_{\Hc(\IndC(\SSS))}(A,B_\bullet)=0$. Hence the map
\[\xymatrix{
\ds\Hom\left(A,\prod_{i=1}^\infty B_i\right)\ar[rr]^-{\id-\shi} &&
\ds\Hom\left(A,\prod_{i=1}^\infty B_i\right)
}\]
is an epimorphism. The long exact sequence, obtained
by applying the functor
$\Hom(A,-)$ to the triangle,
breaks down into short exact sequences, and thus we obtain
an isomorphism
\[
\Hom\big(A,\holim B_\bullet\big)\la\clim\Hom(A,B_\bullet)\ .
\]
As this is true for every object
$A\in\cs$, we have shown that $\YYon$ satisfies
\[
\YYon\big(\holim B_\bullet\big)=\clim\YYon(B_\bullet),
\]
where $\YYon(B_\bullet)$ stands for the Cauchy sequence whose terms are the objects $\YYon(B_i)$.

Now, combining \cite[Lemmas~\ref*{L3.92929} and \ref*{L3.90909}]{Neeman25},
we can find an object $\wt B\in\cl(\cs)\subset\MMod\cs$
and an isomorphism $\wt B\la\clim B_\bullet$, which can be
rewritten as an isomorphism
$\wt B\la\YYon\big(\holim B_\bullet\big)$.
By \cite[Observation~\ref*{O21.9}]{Neeman18A} there
exists an object $B\in\fl'(\cs)\subset\Hc(\IndC(\SSS))$
with $\YYon(B)\iso\wt B$, and hence
we have, in $\MMod\cs$, an isomorphism
$\eta\colon\YYon(B)\la\YYon\big(\holim B_\bullet\big)$ with $B\in\fl'(\cs)$.
Now \cite[Lemma~\ref*{L21.11}]{Neeman18A} tells
us that there is a morphism
$\ph\colon B\la\holim B_\bullet$, in
$\Hc(\IndC(\SSS))$, lifting $\eta$.
By construction, for every object $S\in\cs$, the functor
$\Hom(S,-)$ takes $\ph$ to an isomorphism and since $\cs=\Hc(\SSS)$ compactly generates
the triangulated category $\Hc(\IndC(\SSS))$,
we deduce that $\ph$ is an isomorphism
in $\Hc(\IndC(\SSS))$. That is,
any object $\holim B_\bullet$ in $\fl'\big(\fs(\cs)\op\big)$ is isomorphic to an object
$B\in\fl'(\cs)\op$, giving the inclusion
$\fl'(\cs)\op\supset\fl'\big(\fs(\cs)\op\big)$.

To conclude the proof of (ii) we establish the reverse inclusion. We already know that, for any Cauchy sequence
$B_\bullet$ in $\fs(\cs)\op$, there exists an object
$B\in\fl'(\cs)$ and an isomorphism $B\iso\holim B_\bullet$, which induces an isomorphism
$\YYon(B)\iso\clim\YYon(B_\bullet)$.

Choose any object $A\in\fl'(\cs)$. By 
\cite[Lemma~\ref*{L3.7} (iii)]{Neeman25}, applied
to the object $\YYon(A)\in\fl(\cs)$, there
exists in $\fs(\cs)\op$ a Cauchy sequence
$B_\bullet$ with a map $\psi\colon\YYon(A)\la\clim\YYon(B_\bullet)$,
and by \cite[Lemma~\ref*{L3.90909}]{Neeman25}
the map $\psi$ is an isomorphism.
By the previous discussion, we have that $B=\holim B_\bullet$ is an object of $\fl'(\fs(\cs)\op)\subset\fl'(\cs)\op$ with 
$\YYon(B)\iso\clim\YYon(B_\bullet)$.
Therefore $\YYon(A)\iso\YYon(B)$.
By \cite[Corollary~\ref*{C21.902}]{Neeman18A},
we can lift this isomorphism in
$\fl(\cs)$ to a morphism
$\ph\colon A\la B$ in $\fl'(\cs)$,
and by \cite[Corollary~\ref*{C21.908})]{Neeman18A}
the map $\ph$ must be an isomorphism.
Thus any $A\in\fl'(\cs)$ is isomorphic
to $\holim B_\bullet$ for some Cauchy sequence
$B_\bullet$ contained in $\fs(\cs)\op$, proving
that $\fl'(\cs)\op\subset\fl'\big(\fs(\cs)\op\big)$.

And finally we turn
to the proof of (i). Let $A\in\fl'(\cs)$ be any object
and choose a Cauchy sequence
$A_\bullet$ in $\cs$ with $A=\hoco A_\bullet$. From
\cite[Lemma~\ref*{L94.657} (ii)]{Neeman25}
we know that, for any object $B\in\fs(\cs)$,
we have $\climone\Hom(A_\bullet,B)=0$.
But then \cite[Lemma~\ref*{L21.13}]{Neeman18A}
gives that
\[
\Hom_{\Hc(\IndC(\SSS))}(A,B)\la\Hom_{\MMod\cs}\big(\YYon(A),\YYon(B)\big)
\]
is an isomorphism.
Now, let $B_\bullet$ be a Cauchy sequence in $\fs(\cs)\op$,
and choose an object $B\in\fl'(\cs)$,
with $B\iso\holim B_\bullet$, and with
$\YYon(B)\iso\clim \YYon(B_\bullet)$.
Then, for any object $C\in\fs(\cs)$, we
have isomorphisms
\begin{gather*}
\Hom\big(\holim B_\bullet,C\big)\iso
\Hom(B,C)\iso\Hom\big(\YYon(B),\YYon(C)\big)\\\iso\Hom(\clim \YYon(B_\bullet),\YYon(C))\iso\colim\,\Hom(B_i,C)
\end{gather*}
where the last isomorphism is by
\cite[Lemma~\ref*{L3.90909}]{Neeman25}.
And the existence of the isomorphism from the last to the first term proves that
$\fs(\cs)\op\la\Hc(\IndC(\SSS))\op$ is a good extension.
\eprf

\subsection{Going back and forth}\label{subsec:backforth}

The discussion in the previous section leads to the following result, which is key to many of our applications.

\pro{P4.5}
Let $\cs$ be a triangulated category with an excellent
metric $\{\cm_i\}_{i\in\nn}$ such that
$\cs$ and its subcategories $\cm_i$ are closed
in $\fl(\cs)$ under direct summands. Assume that $\SSS$ is an enhancement of $\cs$ in $\HoStabInfA$. Then there is an (explicit) isomorphism
\[\xymatrix@C+30pt{
\SSS\la\Psi\big(\Psi(\SSS)\op\big)\op\
}\]
in $\HoStabInfA$.
\epro

\prf
By \autoref{L4.3} (i), the
inclusion $\fs(\cs)\op\la\Hc(\IndC(\SSS))\op$ is
a good extension with respect to the
metric $\{\cn\op_i,\,i\in\nn\}$, where $\cn_i$ is as in \eqref{eq:metrN}.

Set $\UU:=\Psi(\SSS)$. We have an isomorphism in $\HoStabInfA$
\begin{eqnarray*}
\left(\Phi^?_{\Hc(\IndC(\SSS)),\fl'(\cs)}\big(\IndC(\SSS))\right)\op&=&
  \Phi^?_{\Hc(\IndC(\SSS))\op,\fl'_1\big(\fs(\cs)\op\big)}\big(\IndC(\SSS)\op\big)\\
&\iso & \Phi^?_{\Hc\big(\IndC(\UU\op)\big),\fl'_2\big(\fs(\cs)\op\big)}\big(\IndC(\UU\op)\big),
\end{eqnarray*}
the equality just amounts to
distributing the $\{-\}\op$ over the three terms
inside the brackets, while the last isomorphism is the restriction to the subcategories of
the natural restricted
Yoneda functor for the inclusion
$\UU\op\subset\big(\IndC(\SSS)\big)\op$
(a non-isomorphism!) $\YYonInf\colon\big(\IndC(\SSS)\big)\op\la\IndC(\UU\op)$. 
The fact that the restriction to the subcategories
is an isomorphism was proved in
\autoref{L1.5}. Note that we also appeal to
\autoref{L4.3} (ii), which tells
us that, as subcategories of
$\Hc(\IndC(\SSS))\op$, we have $\fl'(\cs)\op=\fl_1'\big(\fs(\cs)\op\big)$.

And now we apply the restricted Yoneda functors
\[
\YYon'\colon\Hc(\IndC(\SSS))\la\MMod\cs\qquad\YYon''\colon\Hc(\IndC(\UU\op))\la\MMod\fs(\cs)\op,
\]
as in the previous proof,
and we find the commutative diagram
\[\xymatrix{
&\fl'(\cs)\op\ar@{=}[r]\ar[d]_{(\YYon')\op} & \fl'_1\big(\fs(\cs)\op\big)\ar[r]^-\sim &
\fl'_2\big(\fs(\cs)\op\big)\ar[d]^{\YYon''}&\\
\cs\op\ar@{^{(}->}[r] &
\fl(\cs)\op \ar[rr]^-{\wh\Yon} & & \fl\big(\fs(\cs)\op\big)&
\fs\big(\fs(\cs)\op\big)\ar@{_{(}->}[l]
}\]
where the equality is by \autoref{L4.3} (ii), the
triangle equivalence
in the top row is the triangle equivalence in \autoref{L1.5},
and the equivalence $\wh\Yon\colon\fl(\cs)\op\la \fl\big(\fs(\cs)\op\big)$
is by \cite[Proposition~\ref*{P3.94949}]{Neeman25}. 
But over $\cs\op\subset\fl(\cs)\op$ (respectively, over the subcategory
$\fs\big(\fs(\cs)\op\big)\subset\fl\big(\fs(\cs)\op\big)$) the functor $(\YYon')\op$ (respectively, $\YYon''$) restricts
to a triangle equivalence
by \cite[Theorem~\ref*{T20.17}]{Neeman18A}. This means that,
if we restrict further from $\fl'(\cs)\op\iso\fl'_2\big(\fs(\cs)\op\big)$
to the subcategories $\cs\op\iso\fs\big(\fs(\cs)\op\big)$, we obtain
an isomorphism
\[
\SSS\op=\left(\Phi^?_{\Hc(\IndC(\SSS)),\cs}\big(\IndC(\SSS)\big)\right)\op\iso\Phi^?_{\Hc\big(\IndC(\UU\op)\big),\fs\big(\fs(\cs)\op\big)}\big(\IndC(\UU\op)\big)=\Psi\big(\Psi(\SSS)\op\big)
\]
in $\HoStabInfA$.
\eprf

\rmk{actualmapagainagainagain}
As in \autoref{actualmap}, in \autoref{actualmapagain}
and in \autoref{actualmapagainagain}, we note that
the proof shows a little more than the statement of 
\autoref{P4.5} asserts. The assertion is
the existence of an isomorphism
\[
\SSS\la\Psi\big(\Psi(\SSS)\op\big)\op
\]
and in \autoref{subsect:autoequiv}
it will be important to note that this
isomorphism is natural in $\SSS$.
Let us therefore go carefully through
the way this isomorphism was constructed.

We began with the (unrestricted) Yoneda functor
$\YonInf\colon\SSS\la\IndC(\SSS)$, which is
unmistakeably natural. Then we used the
(characteristic) metric to construct
the subcategory $\fl'(\cs)\subset\Hc(\IndC(\SSS))$,
and inside it
$\fs(\cs)\subset\fl'(\cs)\subset\Hc(\IndC(\SSS))$.
The full $\infty$-subcategory $\Psi(\SSS)\subset\IndC(\SSS)$
has the same objects $\fs(\cs)\subset\Hc(\IndC(\SSS))$. Because the metric is characteristic, the subcategory
$\Psi(\SSS)\subset\IndC(\SSS)$ is stable
under quasi-equivalences.

And now comes the hairy part: we have
the full subcategory
$\Psi(\SSS)\op\subset\IndC(\SSS)\op$, and hence an
induced (restricted) Yoneda functor $\YYonInf\colon\IndC(\SSS)\op\la\IndC(\Psi(\SSS)\op)$.
This allows us to form the composite
\[\xymatrix@C+0pt{
\SSS\op
\ar[r]^-{\YYon\op}
&
\IndC(\SSS)\op
\ar[r]^-{\YYon'}
&
\IndC(\Psi(\SSS)\op)\ ,
}\]
which is clearly natural with respect to
quasi-equivalences.
And what is actually proved in
\autoref{P4.5}
is first that this composite is
fully faithful, and then that the
essential image is precisely the
subcategory
$\Psi(\Psi(\SSS)\op)\subset\IndC(\Psi(\SSS)\op)$.
\ermk

\cor{C4.7}
Let $\cs$ be a triangulated category with an excellent
metric $\{\cm_i\}_{i\in\nn}$ such that
$\cs$ and its subcategories $\cm_i$ are closed
in $\fl(\cs)$ under direct summands. Assuming further that the metrics $\{\cm_i\}_{i\in\nn}$
on $\cs$ and
$\{\cn_i\}_{i\in\nn}$ of the form \eqref{eq:metrN} on $\fs(\cs)$ are both characteristic,
then the map $\Psi$ gives a bijection
\[
\Enhq(\cs)\longleftrightarrow\Enhq\big(\fs(\cs)\big)
\]
\ecor

\prf
Because the metric $\{\cm_i\}$
on $\cs$
is characteristic, \autoref{R1.25}
says that  
$\Psi$ respects isomorphisms in $\HoStabInfA$ and
induces a map
\[
\Psi\colon\Enhq(\cs)\la\Enhq\big(\fs(\cs)\big)\ .
\]
Since the metric
$\{\cn_i\}$ on $\fs(\cs)$
is charecteristic,
so is the metric
$\{\cn\op_i\}$ on $\fs(\cs)\op$,
and \autoref{R1.25} applies again
to give a well-defined map
\[
\Psi\colon\Enhq\big(\fs(\cs)\op\big)\la \Enhq(\cs\op).
\]
Then \autoref{P4.5} shows that $\Psi\comp\Psi=\id$ and
therefore $\Psi$ is a bijection, with its inverse
(up to passing to opposite categories)
being $\Psi$.
\eprf

\subsection{The weakly approximable case}\label{subsec:waen}

We want to specialize to the case where $\ct$ is a weakly approximable triangulated category. We start with the following easy observation.

\lem{lem:restTbcTc}
Let $\ct$ be a weakly approximable triangulated category,
and assume further that $\ct^c\subset\ct^b_c$.
Endow $\ct^c$ with the
good metric of \autoref{ex:metricTsb}.
Then $\fs(\ct^c)=\ct^b_c$, and the map
$\Psi:\Enhq(\ct^c)\la\Enhq(\ct^b_c)$
of \autoref{subsec:constr} is surjective.

If we further assume that 
${}^\perp(\ct^b_c)\cap\ct^-_c=\{0\}$ then
this map is a bijection. More precisely, we have
the pair of 
inverse bijections
\[\xymatrix@C-15pt{
\Phi^?_{\ct^b_c,\ct^c}\colon
\Enhq(\ct^b_c)
\ar@<0.3ex>[rr]
&&
\Enhq(\ct^c)\colon \Psi
\ar@<0.3ex>[ll]
}\]
\elem

\prf
The good metric of
\autoref{ex:metricTsb}
is characteristic by
\cite[Remark~\ref*{R22.105.5} and
  Proposition~\ref*{P22.109}]{Neeman18A},
and by \cite[Example~\ref*{E22.3}]{Neeman18A}
we have that $\fs(\ct^c)=\ct^b_c$.
The inclusion $\ct^c\subset\fs(\ct^c)$
holds in the good extension $\ct$, and hence
it also holds in $\MMod\cs$,
see \cite[Remark~\ref*{R29876}]{Neeman18A}.
Moreover, if we assume the extra hypothesis
that
${}^\perp(\ct^b_c)\cap\ct^-_c=\{0\}$
then  
the inclusion $\ct^c\subset\ct^b_c$ is characteristic by \autoref{thm:passage}. The result
now follows from \autoref{Corofeasyanalog}.
\eprf

The following is another easy application of the general results in the previous section.

\lem{lem:unienTbcTc}
Let $\ct$ be a weakly approximable triangulated category. 
Then there is a bijection between the sets $\Enhq(\ct^b)$ and $\Enhq(\tsb)$. Furthermore, if $\ct$ is coherent, then we have bijections of sets
\[
\Enhq(\ct^c)\longleftrightarrow\Enhq(\ct^b_c)\qquad\EnhStrA(\ct^c)\longleftrightarrow\EnhStrA(\ct^b_c).
\]
\elem

\prf
Under the hypotheses of
the Lemma, the 
categories $\ct^c$ and $\tsb$
have natural excellent metrics,
see \autoref{ex:excelmetriccompact}.
With respect to these metrics we have
$\fs(\ct^c)=\ct^b_c$ and 
$\fs(\tsb)=\ct^b$ by \autoref{ex:concrsigma}.
And the metrics in question are all
characteristic: 
the metric on
$\tsb$ by \autoref{lem:charmetric}, 
the induced metric on $\fs(\tsb)=\ct^b$
by
\cite[Proposition~6.1.8]{Canonaco-Neeman-Stellari24},
and the metrics
on $\ct^c$ and on $\fs(\ct^c)=\ct^b_c$ by
\autoref{ex:charmetricTc}.
The assertions about $\Enhq$ now follow as
an application of \autoref{C4.7}.

In order to deal with the case of strong uniqueness of enhancements, it is enough to observe that, with our assumptions, the passage from $\ct^c$ to $\fs(\ct^c)\iso\ct^b_c$ (and the analogous passage for the corresponding enhancements) is intrinsic in the sense that it induces an isomorphism of the autoequivalence groups of $\ct^c$ and $\ct^b_c$ (see Remark 4.7, Proposition 4.8 and Proposition 6.5 in \cite{Neeman18A}).
\eprf

By putting together the discussion in the previous section, we can prove the following.

\thm{prop:applTwa}
Let $\ct$ be a weakly approximable triangulated category satisfying one of the following two assumptions:
\begin{enumerate}
\item[{\rm (a)}] either $\ct$ is coherent;
\item[{\rm (b)}] or $\ct^c\subset\ct^b_c$.
\end{enumerate}
Let $\ca\subset\cb\subset\ct$ be two of the natural subcategories in the diagram \eqref{eq:incl}, such that
$\cb\neq\ct_c^-$ and, in the case where {\rm (a)} holds but {\rm (b)} does not,
the categories $\ca,\cb\in\{\ct^{c,b},\tsbb\}$ are also excluded. Assume further that
\begin{enumerate}
\item[{\rm (i)}] $\cb$ (and thus $\ca$) has an enhancement in $\HoStabInfA$.
\item[{\rm (ii)}] $\ca$ has a unique enhancement in $\HoStabInfA$.
\item[{\rm (iii)}]
In the case where $\ca=\ct^b_c$ and
{\rm (b)} holds, we
add the assumption that 
${}^\perp(\ct^b_c)\cap\ct^-_c=\{0\}$.
\end{enumerate}
Then $\cb$ has a unique enhancement in $\HoStabInfA$.

Moreover, there is a formula
for $\cb$ and its (unique) enhancement in terms of $\ca$ and its (unique) enhancement. In particular, given a triangulated category $\ca$ with a unique enhancement, there exists at most one unique $\cb$, containing $\ca$ and possessing an enhancement, for which there might exist a weakly approximable triangulated category $\ct$, with $\ca\subset\cb\subset\ct$ being a prescribed pair of the natural subcategories in the diagram \eqref{eq:incl},
and in the case where $\ca=\ct^b_c$
the ambient $\ct$ is assumed to satisfy {\rm(iii)}.
\ethm

\prf
First of all, if in the diagram \eqref{eq:incl}
we have a triple of subcategories $\ca\subset\cb\subset\cc$,
none of which is in the list of exceptions, then
it is enough to prove the assertions  for the
pairs $(\ca,\cb)$ and $(\cb,\cc)$, with the pair
$(\ca,\cc)$ following by transitivity.

The cases where $(\ca,\cb)\in\{(\ct^-,\ct),(\ct^b,\ct^+),(\ct^+,\ct),(\ct^b,\ct^-)\}$ follow directly from \autoref{E2.5}. 
The case $(\ca,\cb)=(\tsb,\ct^-)$ 
is obtained by \autoref{cor:applwaTsb}.

The case $(\ca,\cb)=(\ct^c,\tsb)$ is handled in \autoref{lem:uniqenhTsb}.

Under assumption (b),
the case $(\ca,\cb)=(\tsb,\ct^b)$
is an example of an inclusion $\tsb\subset\fs(\tsb)=\ct^b$, and we can invoke \autoref{Corofeasyanalog}.

Again under the assumption (b), the case
$(\ca,\cb)=(\ct^c,\ct^b_c)$ follows from \autoref{lem:restTbcTc}.

To treat the pair $(\ca,\cb)=(\ct^b_c,\ct^b)$ we appeal to the sequence of inclusions and bijections
\[
\EnhA(\ct^b)\longleftrightarrow\EnhA(\tsb)\longhookrightarrow\EnhA(\ct^c)\longleftrightarrow\EnhA(\ct^b_c),
\]
The middle injection comes from \autoref{lem:uniqenhTsb}.
Under assumption
(a) the first and the last bijections come from
\autoref{lem:unienTbcTc}. Under assumption (b),
only the first bijection follows
from \autoref{lem:unienTbcTc}, but because
we are assuming (iii), the second bijection
is by \autoref{lem:restTbcTc}.

It remains to consider the case $(\ca,\cb)=(\ct^-_c,\ct^-)$,
but first let us discuss the ``Moreover''assertion,
for all the other cases.

Here the point is simple: in each of the
pairs $(\ca,\cb)$ treated above, there
is an explicit  recipe constructing,
up to equivalence, the (unique)
enhancement $\BB$ of $\cb$ out of
the unique enhancement $\AAA$ of
$\ca$. If the relation between $\ca$ and
$\cb$ is that $\cb=\fs(\ca)$ then
$\BB=\Psi(\AAA)$, with $\Psi$
as in \autoref{theconstructionpsi}. For the
pairs $(\ca,\cb)\in\{(\ct^-,\ct),(\ct^b,\ct^+),(\ct^+,\ct),(\ct^b,\ct^-),(\tsb,\ct^b)\}$,
then what works is either
$\BB=\LL(\AAA)$
or $\BB\op=\LL(\AAA\op)$,
with $\LL$ as in 
\autoref{theconstructionL}. In the case of
the pair $(\ca,\cb)=(\ct^c,\tsb)$ 
of \autoref{lem:uniqenhTsb}, the formula
is that
$\BB=\Delta(\AAA)$,
with $\Delta$ as in
\autoref{constrtenhTsb}.

The only case not covered by the above is
the pair $(\ca,\cb)=(\ct^b_c,\ct^b)$. Under
assumption (a), the recipe
sketched in the pre-``moreover'' part of the proof
takes the
enhancement $\AAA$ of $\ct^b_c$
to the object $\Psi\comp\Delta\comp(\Psi(\AAA\op)\op)$ in the category $\HoStabInfA$.
To spell this out: for the
excellent metric on $(\ct^b_c)\op$ we have that
$\Psi(\AAA\op)$ is the unique enhancement
of $(\ct^c)\op=\fs((\ct^b_c)\op)$,
then we apply $\Delta$ to
the unique enhancement 
$(\Psi(\AAA\op)\op)$ of $\ct^c$
to obtain the unique enhancement
of $\tsb$, and finally $\Psi$ to obtain the
unique enhancement for $\ct^b$.

Still with $(\ca,\cb)=(\ct^b_c,\ct^b)$
but under the assumption (b), we must be in
the situation
of (iii). Making explicit the recipe we sketched
in the pre-``moreover'' part of the proof,
the unique enhancement $\AAA$ of $\ct^b_c$ goes 
to the unique enhancement $\Psi\comp\Delta\comp\Phi^?_{\ct^b_c,\ct^c}(\AAA)$ of the category $\ct^b$.

It remains to treat the case $(\ca,\cb)=(\ct^-_c,\ct^-)$.
Suppose $\AAA$ is the unique
enhancement of $\ct^-_c$ and $\BB$ is some
enhancement of $\ct^-$. By \autoref{thm:passage}
$\ct^-_c$ is a characteristic subcategory of $\ct^-$,
and the map $\Phi^?_{\ct^-,\ct^-_c}\colon\Enhq(\ct^-)\la\Enhq(\ct^-_c)$ is
therefore well-defined.
Hence $\Phi^?_{\ct^-,\ct^-_c}(\BB)$ is an enhancement of
$\ct^-_c$, and by the uniqueness of the enhancement of $\ct^-_c$ 
we must have an isomorphism $\AAA\iso\Phi^?_{\ct^-,\ct^-_c}(\BB)$.
But \autoref{thm:passage} also tells us that
$\ct^c\subset\ct^-_c$ is a characteristic
subcategory and therefore
$\Phi^?_{\ct^-_c,\ct^c}\colon\Enhq(\ct^-_c)\la\Enhq(\ct^c)$
is well-defined. Applying this map to the
isomorphism above, we obtain the
isomorphism in the sequence
\[
\Phi^?_{\ct^-_c,\ct^c}(\AAA)\iso
\Phi^?_{\ct^-_c,\ct^c}\left(\Phi^?_{\ct^-,\ct^-_c}(\BB)\right)
=
\Phi^?_{\ct^-,\ct^c}(\BB)\ ,
\]
with the equality being obvious.
But the restriction maps
\[
\xymatrix@C+15pt{\Enhq(\ct^-)
\ar[r]^-{\Phi^?_{\ct^-,\tsb}}
&
\Enhq(\tsb)
\ar[r]^-{\Phi^?_{\tsb,\ct^c}}
&
\Enhq(\ct^c)
}
\]
are both monomorphisms: the first by
\autoref{cor:applwaTsb} and the second by
\autoref{lem:uniqenhTsb}.
The map $\Phi^?_{\cb,\ct^c}$, being the composite of these monomorphisms, is also a monomorphism.
Hence there exists at most
one $\BB$ with
$\Phi^?_{\ct^-,\ct^c}(\BB)\iso\Phi^?_{\ct^-_c,\ct^c}(\AAA)$, proving the uniqueness of $\BB\in\Enhq(\cb)$.

For the ``Moreover'' assertion
we remember that, for both the 
restriction maps $\Phi^?_{\ct^-,\tsb}$ and $\Phi^?_{\tsb,\ct^c}$,
there is an explicit reconstruction procedure.
One procedure to obtain the unique possible $\BB$
goes as follows. We start with the
enhancement $\Phi^?_{\ct^-_c,\ct^c}(\AAA)$ of
$\ct^c$, apply $\Delta$ to obtain an
enhancement $\SSS$ of $\tsb$, and then apply $\LL$ to
form an enhancement of $\TT$ of $\ct^-$.
\eprf

The following is now straightforward.

\cor{cor:finTwa}
Let $\ct$ be a weakly approximable triangulated category having an enhancement and satisfying one of the following two assumptions:
\begin{enumerate}
\item[{\rm (a)}] either $\ct$ is coherent;
\item[{\rm (b)}] or $\ct^c\subset\ct^b_c$.
\end{enumerate}
If $\ct^c$ has a unique enhancement in $\HoStabInfA$, then any of the natural subcategories $\ca$ in \eqref{eq:incl}, such that $\ca\not\in\{\ct^-_c,\ct^{c,b},\tsbb\}$ when {\rm (a)} holds and $\ca\ne\ct^-_c$ when {\rm (b)} holds, has a unique enhancement in $\HoStabInfA$.
\ecor

The above results suggest the following natural question.

\begin{qn}\label{qn:1}
Let $\ct$ be a weakly approximable triangulated category. Assume that $\ct$ has an enhancement in $\StabInfA$, and that $\ct^c$ has a unique enhancement. Does $\ct^-_c$ have a unique enhancement in $\HoStabInfA$?
\end{qn}

We conclude this section with a result that is also
relevant to strong uniqueness of enhancements.

\lem{lem:struniqgen}
Let $\ct$ be a weakly approximable and coherent triangulated category. 
Then $\ct^c$ has a (strongly) unique enhancement if and only if $\ct^b_c$ does.
\elem

\prf
This follows from \autoref{lem:unienTbcTc}.
\eprf

\subsection{Extending equivalences}\label{subsec:lifting}

An important consequence of \autoref{thm:passage} is that if $\ct$ and $\ct'$ are two weakly approximable triangulated categories and we have a pair of matching subcategories $\cb\subset\ct$ and $\cb'\subset\ct'$ in the diagram \eqref{eq:incl}, any equivalence $\cb\iso\cb'$ descends to all the matching subcategories $\ca\subset\cb$ and $\ca'\subset\cb'$ in the diagram \eqref{eq:incl}. The aim of this section is to go upward and extend the existence of equivalences.

\thm{prop:applextTwa}
Let $\ct_1$ and $\ct_2$ be a pair of
weakly approximable triangulated categories, satisfying one of the following two assumptions for $i=1,2$:
\begin{enumerate}
\item[{\rm (a)}] either $\ct_i$ is coherent,
\item[{\rm (b)}] or $\ct_i^c\subset(\ct_i)^b_c$.
\end{enumerate}
For $i=1,2$, let $\ca_i\subset\cb_i\subset\ct_i$ be matching inclusions of subcategories in the diagram \eqref{eq:incl}, 
such that $\cb_i\ne(\ct_i)^-_c$, and where 
the categories $\ca_i,\cb_i\in\{\ct_i^{c,b},\tsbb_i\}$ are also excluded. Assume further that
\begin{enumerate}
\item[\rm (i)] $\cb_i$ (and thus $\ca_i$) has an enhancement in $\HoStabInfA$, for $i=1,2$.
\item[\rm (ii)] $\ca_i$ has a unique enhancement
in $\HoStabInfA$, for $i=1,2$.
\item[{\rm (iii)}]
In the case where $\ca_i=(\ct_i)^b_c$,
we assume that either   
{\rm (a)} holds for both
$\ct_1$ and $\ct_2$, or
else that {\rm (b)} holds for both.
And when both $\ct_1$ and $\ct_2$
satisfy {\rm (b)} we furthermore
suppose that
${}^\perp(\ct_i)^b_c\cap(\ct_i)^-_c=\{0\}$,
for $i=1,2$.
\end{enumerate}
If $\ca_1\iso\ca_2$, then $\cb_1\iso\cb_2$.
\ethm

\prf
The Theorem is an immediate consequence
of the ``Moreover'' part of
\autoref{prop:applTwa}. After all, the unique
enhancement of $\ca_1\iso\ca_2$ permits us
to explicitly construct, up to
equivalence, the unique enhancement of
both $\cb_1$ and $\cb_2$. Hence $\cb_1\iso\cb_2$.

The only point worth mentioning is that
the recipe, which passes from
the unique enhancement of $\ca_i$ to
the unique enhancement of $\cb_i$,
is mostly the same for case {\rm (a)}
and for case {\rm (b)}. The only difference
occurs when $\ca_i=(\ct_i)^b_c$,
and therefore in this case we have to
impose condition (iii).
\eprf

Note that the result does not prove that the equivalence $\cb_1\iso\cb_2$ restricts to the given equivalence $\ca_1\iso\ca_2$. It would be interesting to investigate when this happens.

An immediate application of \autoref{prop:applextTwa} is the following generalization of \cite[Theorem C]{Canonaco-Neeman-Stellari24}:

\cor{cor:applextTwa}
Let $\ct_1$ and $\ct_2$ be two enhanceable, weakly approximable triangulated categories,
with each one satisfying one of the following two assumptions:
\begin{enumerate}
\item[{\rm (a)}] either $\ct_i$ is coherent,
\item[{\rm (b)}] or $\ct_i^c\subset(\ct_i)^b_c$.
\end{enumerate}
Let $\ca_1\subset\ct_1$ and $\ca_2\subset\ct_2$ be matching inclusions of subcategories in the diagram \eqref{eq:incl}, with
$\ca_i\notin\{(\ct_i)^-_c,\ct_i^{b,c},\tsbb_i\}$.
Assume further that $\ca_1\iso\ca_2$,
and this triangulated category has
a unique enhancement.
Furthermore:
\begin{enumerate}
\item[{\rm (i)}]
In the case where $\ca_i=(\ct_i)^b_c$,
then either {\rm (a)} holds for both
$\ct_1$ and $\ct_2$ or
{\rm (b)} holds for both.
And if {\rm (b)} holds for both,
then we also suppose that 
${}^\perp(\ct_i)^b_c\cap(\ct_i)^-_c=\{0\}$,
for $i=1,2$
\end{enumerate}
Then the entire tower of subcategories
of the diagram \eqref{eq:incl}
matches up, it can be identified (up to triangle equivalences) 
for $\ct_1$ and for $\ct_2$.
\ecor

\prf
We are given first that $\ca_1\iso\ca_2$, second
that this category has a unique enhancement, and
third
that $\ct_1$ and $\ct_2$  both have enhancements.
We simply apply \autoref{prop:applextTwa} to
the matching pair $(\ca_1,\ct_1)$ and $(\ca_2,\ct_2)$
to
deduce that $\ct_1\iso\ct_2$.
Hence the whole tower of
intrinsic subcategories matches up in pairs.
\eprf

Following the same stream of ideas that led us to formulate \autoref{qn:1}, we can ask whether \autoref{prop:applextTwa} may be extended to work for $\ct^-_c$ as well.

\section{Applications}\label{sect:applications}

In this section we discuss several applications of geometric/topological nature of the results in the previous section. In particular, in \autoref{subsect:uegeo} we apply \autoref{thm:main2} to geometric categories and (as a special case) we give a positive answer to the open question by Antieau mentioned in the introduction. \autoref{subsect:extspectra} concerns applications of \autoref{thm:main3} and provides two main results: a complete generalization of a result by Rickard (see \autoref{cor:D}) and a generalization of the main result in \cite{Schwede07} (see \autoref{margolis2}). Finally, \autoref{subsect:autoequiv} deals with the comparison of autoequivalence groups.

\subsection{(Strong) uniqueness in geometric contexts}\label{subsect:uegeo}

In the geometric setting, we can easily apply the results in \autoref{subsec:waen}, thus improving the existing previous results.

\pro{cor:geomappli}
Let $X$ be a quasi-compact and quasi-separated scheme and let $Z\subset X$ be a closed subset such that $X\setminus Z$ is quasi-compact.
\be
\item
If $X=Z$, then all the natural triangulated subcategories of $\Dqc(X)$ in \autoref{ex:waexample}, except $\Dqcp(X)$, have unique enhancements in $\HoStabInfA$.
\item
If $X$ is noetherian, then all the natural triangulated subcategories of $\Dqcs Z(X)$ in \autoref{ex:waexample} have unique enhancements in $\HoStabInfA$.
\ee
\epro

\prf
First we treat the case $X=Z$.
The  weakly approximable
triangulated
category $\ct=\Dqc(X)$ obviously has an enhancement,
while the subcategory $\ct^c=\dperf X$ has
a unique enhancement by \autoref{thm:uniqenhangeom} (ii).
Also, by \cite[Proposition~10.1.3]{Canonaco-Neeman-Stellari24},
we have that $^\perp(\ct^b_c)\cap\ct^-_c=\{0\}$.
\autoref{cor:finTwa} therefore applies to assert
that the other subcategories except $\Dqcp(X)$ also
have unique enhancements.

If $X$ is noetherian then $\ct=\Dqcs Z(X)$ is a coherent
weakly approximable triangulated category having
an enhancement. By \autoref{cor:uniqensupp} $\ct^b_c=\Dqcpbs Z(X)$
has a unique enhancement, and from \autoref{lem:unienTbcTc}
we learn that so does $\ct^c=\dperfs ZX$. Applying
\autoref{cor:finTwa} proves that all the natural
triangulated subcategories other than $\D^-_{\coh,Z}(X)(=\Dqcps Z(X))$
have unique enhancements in $\HoStabInfA$. This leaves
us with the case of $\D^-_{\coh,Z}(X)$ which, under our assumptions,
is equivalent to $\D^-(\coh_Z(X))$. The uniqueness of
enhancement for this category comes from \autoref{thm:uniqenhangeom} (i).
\eprf

\rmk{rmk:uniqueaffine}
If $R$ is a ring a similar proof, using \autoref{thm:uniqenhring}, provides the uniqueness of enhancements for all the full triangulated subcategories of $\D(R)$ in \autoref{ex:waexample}, other than $K^-(\proj{R})$.
\ermk

The machinery can also be applied to the homotopy category of spectra, yielding the following
improvement on \autoref{thm:uniqenhspectra}. 

\pro{cor:spectrappli}
All the natural triangulated subcategories of $\Ho{\Spe}$ in \eqref{eq:incl}, apart from $\Ho{\Spe}^-_c$ and ${\Ho{\Spe}}\SBb$, have unique enhancements in $\HoStabInfNL$.
\epro

\prf
As explained in \autoref{ex:coherent}, the triangulated category $\ct=\Ho{\Spe}$ is a coherent, weakly approximable triangulated category.
Clearly $\ct$ has an enhancement, while
the subcategory
$\ct^c=\Ho{\Spe}^{c}$ has a unique
enhancement by \autoref{thm:uniqenhspectra}.
Hence \autoref{cor:finTwa} applies and tells us  
that all the subcategories 
except $\Ho{\Spe}^{c,b},\Ho{\Spe}^-_c,{\Ho{\Spe}}\SBb$
have unique enhancements.
This leaves us with showing that
$\Ho{\Spe}^{c,b}$ has a unique enhancement, but this is
obvious since $\Ho{\Spe}^{c,b}=\{0\}$.
\eprf

For strong uniqueness of enhancements we present two results.

\cor{cor:struniqexcomp}
If $X$ is a scheme proper over a field $\kK$, such that $X$ has depth $\geq 1$, then both $\D^b(\coh(X))$ and $\dperf{X}$ have strongly unique enhancements in $\HoStabInf$.
\ecor

\prf
The triangulated category $\ct=\Dqc(X)$
is weakly approximable and coherent, and the
hypotheses on $X$ are such that
\autoref{prop:struniqex}
informs us that the subcategory
$\ct^b_c=\dbcoh(X)\cong\D^b(\coh(X))$
has a strongly unique enhancement.
By \autoref{lem:struniqgen} it automatically
follows that  $\ct^c=\dperf{X}$
also has a strongly unique enhancement.
\eprf

\cor{cor:struniqexcompsupp}
Let $X$ be a scheme quasi-projective over a field $\kK$, containing a projective subscheme $Z\subset X$ such
that $Z$ has depth $\ge1$, and also satisfying
$\OO_{iZ}\in\dperf{X}$ for all $i> 0$. Then
both $\dperfs Z{X}$ and $\D^b_{\coh,Z}(X)$ have strongly unique enhancements in $\HoStabInf$.
\ecor

\prf
The triangulated category $\ct=\Dqcs Z(X)$
is weakly approximable and coherent, and the
hypotheses on $Z\subset X$ are such that
\autoref{thm:suppCS}
informs us that the subcategory
$\ct^c=\dperfs ZX$
has a strongly unique enhancement.
By \autoref{lem:struniqgen} it again
automatically
follows that  $\ct^b_c=\D^b_{\coh,Z}(X)$
also has a strongly unique enhancement.
\eprf

\subsection{Extending equivalences and spectra}\label{subsect:extspectra}

The general results in \autoref{subsec:lifting} have important applications in the geometric setting. In particular, \autoref{cor:applextTwa} clearly applies to the case $\ct_i=\Dqc(X_i)$, where $X_i$ is a quasi-compact and quasi-separated scheme. In this way we recover Corollary D in \cite{Canonaco-Neeman-Stellari24}, which we repeat here both for the convenience of the reader and because we can slightly improve it by adding item (8); it treats the category $\tsb$ studied in \autoref{subsec:functorTsb}, which in the special case of $\ct=\Dqcs Z(X)$ turns out to be the $\Dloc_Z(X)$ of \autoref{subsect:geomcase}.

\cor{cor:D}
Let $X$ and $X'$ be quasi-compact and quasi-separated schemes with closed subschemes $Z\subset X$ and $Z'\subset X'$ such that $X\setminus Z$ and $X'\setminus Z'$ are quasi-compact.
Assume further that one of the two conditions below
holds:
\be
\item
Either  $Z=X$ or $Z'=X'$.
\item
One of $X$ or $X'$ is noetherian.
\ee
Then the following are equivalent:
\begin{enumerate}[{\rm (1)}]
\item There exists a triangle equivalence $\Dqcs Z(X)\iso\Dqcs{Z'}(X')$.
\item There exists a triangle equivalence $\Dqcpls Z(X)\iso\Dqcpls{Z'}(X')$.
\item There exists a triangle equivalence $\Dqcmis Z(X)\iso\Dqcmis{Z'}(X')$.
\item There exists a triangle equivalence $\Dqcbs Z(X)\iso\Dqcbs{Z'}(X')$.
\item There exists a triangle equivalence $\Dqcps Z(X)\iso\Dqcps{Z'}(X')$.
\item There exists a triangle equivalence $\Dqcpbs Z(X)\iso\Dqcpbs{Z'}(X')$.
\item There exists a triangle equivalence $\dperfs ZX\iso\dperfs{Z'}{X'}$.
\item There exists a triangle equivalence $\Dloc_Z(X)\iso\Dloc_{Z'}(X')$.
\end{enumerate}
\ecor

\prf
First of all:
for any pair $Z\subset X$ the category
$\ct=\Dqcs Z(X)$ 
is  weakly approximable, it
satisfies $\ct^c\subset\ct^b_c$, and
by \cite[Proposition~10.1.3]{Canonaco-Neeman-Stellari24}
we have that $^\perp(\ct^b_c)\cap\ct^-_c=\{0\}$.
By
\autoref{thm:passage},
applied to the
pair $\ct=\Dqcs Z(X)$ and
$\ct'=\Dqcs{Z'}(X')$,
the equivalence of any pair in the
tower implies the equivalence of the
(smallest) intrinsic subcategories
$\dperfs ZX\iso\dperfs{Z'}{X'}$.
By \autoref{cor:geomappli}
this category has a unique enhancement
under any of the four hypotheses.
\autoref{cor:applextTwa} therefore
tells us that all the categories
in the tower match up.
\eprf

Moving on to the topological example of $\Ho{\Spe}$ we recall that, traditionally, the
interest in
results like \autoref{thm:main3} has come from the following conjecture of \cite[Chapter 2, Section 1]{Margolis}:

\begin{conjecture}[Margolis Uniqueness Conjecture]\label{conj:margolis}
Let $\ct$ be a compactly generated triangulated category with small coproducts. If there is an exact equivalence $\Ho{\Spe}^c\iso\ct^c$, then there is an exact equivalence $\Ho{\Spe}\iso\ct$.
\end{conjecture}

This conjecture is still wide open, but a variant was proved in the beautiful paper \cite{Schwede07}. In the special case where $\ct$ admits an enhancement the conjecture is true; the proof 
(as in the original \cite{Schwede07}) is by a combination of \autoref{thm:uniqenhspectra} and \autoref{lem:restTTc}.

Let $\ct$ be some compactly generated $\ct$ with $\ct^c\iso\Ho\Spe^c$.
By \autoref{prop:ax2} $\ct$ must be weakly approximable and by
\autoref{excellencebycompacts} it must be coherent,
both properties depend only on the category
$\ct^c$. We may therefore formulate a more general version of \autoref{conj:margolis}. Let $\ct$ be a weakly approximable triangulated category and consider the following two sets of subcategories:
\begin{align*}
\cc_\Spe:=&\left\{\Ho{\Spe},\Ho{\Spe}^b,\Ho{\Spe}^-,\Ho{\Spe}^+,\Ho{\Spe}\SB,\Ho{\Spe}^-_c,\Ho{\Spe}^b_c,\Ho{\Spe}^c\right\},\\
\cc_\ct:=&\left\{\ct,\ct^b,\ct^-,\ct^+,\tsb,\ct^-_c,\ct^b_c,\ct^c\right\}.
\end{align*}
Then the generalized conjecture states:

\begin{conjecture}[Extended Uniqueness Conjecture]\label{conj:margolis2}
Let $\ct$ be a coherent, weakly approximable triangulated category and let $\ca_1\subset\cb_1$ in $\cc_\Spe$ and $\ca_2\subset\cb_2$ in $\cc_\ct$ be matching subcategories. If there is an exact equivalence $\ca_1\iso\ca_2$, then there is an exact equivalence $\cb_1\iso\cb_2$.
\end{conjecture}

Using \autoref{cor:finTwa} we prove the following special but interesting case of \autoref{conj:margolis2}.

\pro{margolis2}
\autoref{conj:margolis2} holds true if $\cb_2\neq\ct^-_c$, and $\cb_2$  has an enhancement.
\epro

\prf
By \autoref{thm:uniqenhspectra} and \autoref{cor:finTwa} all the subcategories in $\cc_\Spe$ except $\Ho{\Spe}^-_c$ have unique enhancements in $\HoStabInfNL$. Unless $\ca_1=\ct^-_c$  the result follows from \autoref{prop:applextTwa}, with $\ct$ and $\Ho{\Spe}$ both satisfying assumption (a).

It remains to deal with the case where $\ca_1=\Ho\Spe^-_c$ and $\ca_2=\ct^-_c$. We are assuming that $\Ho\Spe^-_c\iso\ct^-_c$, and \autoref{thm:main1} guarantees that
$\Ho\Spe^c\iso\ct^c$. But these equivalent categories
have a unique enhancement, and
\autoref{prop:applextTwa}
applies again to guarantee that $\cb_1\iso\cb_2$.
\eprf

This generalizes the result in \cite{Schwede07}, which is the case where $\ca_1=\Ho{\Spe}^c$, $\cb_1=\Ho{\Spe}$, $\ca_2=\ct^c$ and $\cb_2=\ct$.

\subsection{Autoequivalence groups}\label{subsect:autoequiv}

Over recent decades there has been much interest
in the study of autoequivalence groups
of triangulated categories and of their
enhancements. It is therefore worth noting
the following.

\rmk{Rmkautequivgrps}
The following constructions induce natural
group homomorphisms:
\be
\item
Given a triangulated category 
$\ct$ and a characteristic triangulated
subcategory $\cs\subset\ct$, then
restriction gives a group homomorphism
$\Res\colon\Aut(\ct)\la\Aut(\cs)$.
\item
Given a triangulated category 
$\ct$ with a
characteristic good metric $\{\cm_i\}$,
then the construction
$\fs$ yields a natural group homomorphism
$\Aut(\ct)\la\Aut(\fs(\ct))$.
\setcounter{enumiv}{\value{enumi}}
\ee
In the world of $\infty$-categories,
there are various notions
of automorphism groups. The one we will work with
is the following:
if $\TT$ is a stable
$\infty$-category, $\AutL(\TT)$ will denote
the automorphism group of the object $\TT$
in the category $\HoStabInfA$.
In the special case of derived categories
of (good enough) schemes, this agrees with the 
group of Fourier-Mukai transforms.
Suppose next that we are given a 
triangulated category 
$\ct$ together with an enhancement
$\TT\in\Enhq(\ct)$.
\be
\setcounter{enumi}{\value{enumiv}}
\item
If $\cs\subset\ct$ is a
characteristic triangulated
subcategory,
then
restriction gives a group homomorphism
$\Res\colon\AutL(\TT)\la\AutL(\Phi^?_{\ct,\cs}(\TT))$.
\item
If the triangulated category 
$\ct$ has a classical generator $G$,
then with $\Delta(\TT)$ as
in \autoref{constrtenhTsb} there is
an obvious group homomorphism
$\wt\Delta\colon\AutL(\TT)\la\AutL(\Delta(\TT))$.
\item
If the triangulated category 
$\ct$ has a characteristic very good metric $\cm_i$,
then we have two natural maps
\[
\wt\LL:\AutL(\TT)\la\AutL(\LL(\TT))
\quad
\text{and}
\quad
\wt\Psi\colon\AutL(\TT)\la\AutL(\Psi(\TT))\ .
\]
with $\LL$ as 
\autoref{theconstructionL} and
with $\Psi$ as in
\autoref{theconstructionpsi}.
\ee
\ermk

And now the result we want to prove is the following.

\pro{Proallisomorphisms}
Let $\ct$ be a weakly approximable triangulated
category, and let $\ca\subset\cb$ be
a pair of subcategories in
the diagram \eqref{eq:incl}, with
$\ca,\cb\notin\{\ct^{c,b},\tsbb,\ct^b_c,\ct^-_c\}$.
If $\BB$ is any enhancement of
$\cb$, then the restriction map
\[
\Res\colon\AutL(\BB)\la\AutL(\Phi^?_{\cb,\ca}(\BB))
\]
is an isomorphism. Moreover, if $\ct$ satisfies one of the hypotheses
\be
\item
Either $\ct$ is coherent,
\item
Or we have $\ct^c\subset\ct_c^b$ and the inclusion  
makes $\ct^c$ a characteristic subcategory,
\ee
then we can also allow $\ca,\cb=\ct^b_c$.
\epro

\prf
It suffices to consider the case where there
is no intermediate category in the diagram
\eqref{eq:incl}; if we have a triple
$\ca\subset\cb\subset\cc$, then the case of the
pair $\ca\subset\cc$ follows by
transitivity from the pairs $\ca\subset\cb$
and $\cb\subset\cc$.

In the case $\ct^c\subset\tsb$, if we start with
an enhancement $\BB$ of $\cb=\tsb$,
then 
\autoref{lem:uniqenhTsb} and
\autoref{actualmap} give a natural
map $\alpha\colon\BB\la\Delta\comp\Phi^?_{\cb,\ca}(\BB)$
and establish that $\alpha$ is an isomorphism
in $\HoStabInfA$. 
By \autoref{Rmkautequivgrps} (iii) and (iv)
we have group homomorphisms
\[
\xymatrix@C-15pt{
\AutL(\BB)
\ar[rr]^-{\Res}
&&
\AutL(\Phi^?_{\cb,\ca}(\BB))
\ar[rr]^-{\mathrm{\wt\Delta}}
&&
\AutL(\Delta\comp\Phi^?_{\cb,\ca}(\BB))
}
\]
and, if $\gamma\colon\BB\la\BB$ is any
element of $\AutL(\BB)$, the naturality
of $\alpha\colon\BB\la\Delta\comp\Phi^?_{\cb,\ca}(\BB)$
gives the commutativity of the square
\[
\xymatrix@C+40pt{
\BB
\ar[r]^\gamma
\ar[d]_\alpha
&
\BB
\ar[d]^\alpha
\\
\Delta\comp\Phi^?_{\cb,\ca}(\BB)
\ar[r]^-{\wt\Delta\comp\Res(\gamma)}
&
\Delta\comp\Phi^?_{\cb,\ca}(\BB)
}
\]
The commutativity of this square, for every
$\gamma\in\AutL(\BB)$,
establishes that the composite
$\wt\Delta\comp\Res$ is an isomorphism
$\AutL(\BB)
\la
\AutL(\Delta\Phi^?_{\cb,\ca}(\BB))$.
But the composite 
\[\xymatrix@C-15pt{
\AutL(\Phi^?_{\cb,\ca}(\BB))
\ar[rr]^-{\mathrm{\wt\Delta}}
&&
\AutL(\Delta\comp\Phi^?_{\cb,\ca}(\BB))
\ar[rr]^-{\Res}
&&
\AutL(\Phi^?_{\cb,\ca}(\BB))
}\]
is obviously equal to the identity, and
the fact that $\Res$ is an
isomorphism immediately follows.

The other cases are similar: if
$(\ca,\cb)\in\{(\ct^-,\ct),(\ct^b,\ct^+),(\ct^+,\ct),(\ct^b,\ct^-),(\tsb,\ct^-)\}$ 
then the way to recover
$\BB\in\Enhq(\cb)$ from
$\Phi^?_{\cb,\ca}(\BB)$ is the
isomorphism
$\BB\la\LL\comp\Phi^?_{\cb,\ca}(\BB)$
or the isomorphism
$\BB\op\la\LL\comp\Phi^?_{\cb,\ca}(\BB)\op$,
depending on whether we apply
\autoref{L1.5} and \autoref{P15.5}
to the pair $(\ca,\cb)$ or to
the pair $(\ca\op,\cb\op)$. In either case
\autoref{actualmapagain} tells us that the
isomorphism is natural, and an argument
just like the one for the pair
$\ct^c\subset\tsb$ proves that the composite
\[
\xymatrix@C-15pt{
\AutL(\BB)
\ar[rr]^-{\Res}
&&
\AutL(\Phi^?_{\cb,\ca}(\BB))
\ar[rr]^-{\mathrm{\wt\LL}}
&&
\AutL(\LL\comp\Phi^?_{\cb,\ca}(\BB))
}
\]
is an isomorphism, while the composite
\[
\xymatrix@C-15pt{
\AutL(\Phi^?_{\cb,\ca}(\BB))
\ar[rr]^-{\mathrm{\wt\LL}}
&&
\AutL(\LL\comp\Phi^?_{\cb,\ca}(\BB))
\ar[rr]^-{\Res}
&&
\AutL(\Phi^?_{\cb,\ca}(\BB))
}
\]
is the identity.

In the case where we have $\ct^c\subset\ct^b$
there are more natural inclusions: the subcategory
$\tsb\subset\ct^b$ is characteristic and
satisfies $\ct^b=\fs(\tsb)$, and (if (ii) also holds)
the subcategory $\ct^c\subset\ct^b_c$ is characteristic and satisfies $\ct^b_c=\fs(\ct^c)$.
Thus if $(\ca,\cb)\in\{(\ct^c,\ct^b_c),(\tsb,\ct^b)\}$
and $\BB\in\Enhq(\cb)$,
the proof of
\autoref{Leasyanalog} gives an
isomorphism $\BB\la\Psi\comp\Phi^?_{\cb,\ca}(\BB)$,
and \autoref{actualmapagainagain} says that
the map is natural.
As before this tells us that the composite
\[
\xymatrix@C-15pt{
\AutL(\BB)
\ar[rr]^-{\Res}
&&
\AutL(\Phi^?_{\cb,\ca}(\BB))
\ar[rr]^-{\mathrm{\wt\Psi}}
&&
\AutL(\Psi\comp\Phi^?_{\cb,\ca}(\BB))
}
\]
is an isomorphism, while the composite
\[
\xymatrix@C-15pt{
\AutL(\Phi^?_{\cb,\ca}(\BB))
\ar[rr]^-{\mathrm{\wt\Psi}}
&&
\AutL(\Psi\comp\Phi^?_{\cb,\ca}(\BB))
\ar[rr]^-{\Res}
&&
\AutL(\Phi^?_{\cb,\ca}(\BB))
}
\]
is the identity.

Now the metric on $\cs_1=\tsb$ is always excellent, while
the metric on $\cs_2=\ct^c$ is excellent if and only if
$\ct$ is coherent. Both categories
are idempotent-complete, hence for $i=1,2$ we have that $\cs_i\subset\fl(\cs_i)$ is closed with respect to direct summands. Now for excellent
metrics on $\cs=\Hc(\SSS)$,
\autoref{P4.5} provides
an explicit natural transformation
$\SSS\op\la\Psi(\Psi(\SSS)\op)$ and proves
it to be an isomorphism in $\HoStabInfA$,
while \autoref{actualmapagainagainagain}
establishes the naturality of the map.
This makes the composite
\[
\xymatrix@C-15pt{
\AutL(\SSS)
\ar[rr]^-{\mathrm{\wt\Psi}}
&&
\AutL(\Psi(\SSS)\op)
\ar[rr]^-{\mathrm{\wt\Psi}}
&&
\AutL(\Psi(\Psi(\SSS)\op)\op)
}
\]
an isomorphism, but as
$\AutL(\Psi(\Psi(\SSS)\op)\op)\iso\AutL(\SSS)$
this tells us that $\AutL(\SSS)$
and 
$\AutL(\Psi(\SSS))$
are isomorphic.
\eprf

An immediate consequence of
\autoref{Proallisomorphisms} is:

\cor{corollaryautomorphisms}
Suppose $\ct$ is a weakly approximable triangulated category, such that the subcategories
$\ct^c,\tsb,\ct^b,\ct^-,\ct^+,\ct$ all
have unique enhancements. Then the automorphism
groups of these enhancements are all isomorphic.

If we furthermore assume that one of the following
holds
\be
\item
Either $\ct$ is coherent,
\item
Or we have $\ct^c\subset\ct_c^b$ and the inclusion  
makes $\ct^c$ a characteristic subcategory,
\ee
then $\ct^b_c$ also has a unique enhancement, and
the automorphism group of this unique enhancement
is the same.
\ecor

\prf
Under the added hypotheses, $\ct^b_c$ has a unique
enhancement by
\autoref{cor:finTwa}. The assertion about the 
automorphism groups being all isomorphic
comes directly from \autoref{Proallisomorphisms},
which tells us that all the restriction homomorphisms
are isomorphisms.
\eprf

\exm{exmallisomorphisms}
The theorem applies to the following
three examples:
\be
\item
$\ct=\Ho\Spe$, the homotopy category of
spectra.  
\item
$\ct=\Dqcs Z(X)$, where 
$X$ is a noetherian scheme and $Z$ is a closed subset.
\item
$\ct=\Dqc(X)$, where $X$ is a quasi-compact and quasi-separated scheme.
\ee
By this we mean that if
$\BB\in\HoStabInfA$
is any object, with $\Hc(\BB)=\ct^\star_\Box$,
where $\ct$ is as in (i), (ii) or (iii) and
the decoration is any of
$\ct^c,\ct^b_c,\tsb,\ct^b,\ct^-,\ct^+,\ct$,
then the automorphism group
$\AutL(\BB)$ is independent of
the decoration. 
\eexm

There are of course other group homomorphisms
one could study.
First of all, given any object $\TT\in\HoStabInfA$,
there is an obvious comparison map
$\Comp\colon\AutL(\TT)\la\Aut(\Hc(\TT))$.
By \autoref{Rmkautequivgrps} (i), given a triangulated
category $\ct$ and a characteristic
triangulated subcategory $\cs\subset\ct$, there is a
restriction homomorphism
$\Res\colon\Aut(\ct)\la\Aut(\cs)$.
And if $(\TT,\fF)$ is an enhancement of $\ct$ in $\HoStabInfA$, then the square below obviously commutes
\[
\xymatrix@C+30pt{
\AutL(\TT)
\ar[r]^-{\Res_\infty}
\ar[d]_-{\Comp_1}
&
\AutL(\Phi^?_{\ct,\cs}(\TT))
\ar[d]^-{\Comp_2}
\\
\Aut(\ct)
\ar[r]^-{\Res}
&
\Aut(\cs),
}
\]
where $\Comp_1$ in the above diagram is a shorthand for the compositions $\fF\comp\Comp\comp\fF^{-1}$, while
$\Comp_2$ stands for $\fF|_{\Phi^?_{\ct,\cs}(\TT)}\comp\Comp\comp(\fF|_{\Phi^?_{\ct,\cs}(\TT)})^{-1}$. In the light of
\autoref{Proallisomorphisms} it becomes interesting
to note the following:

\rmk{Corgeneralmonomorphism}
With the notation as in the commutative square
above, assume that the map
$\Res_1\colon\AutL(\TT)\la\AutL(\Phi^?_{\ct,\cs}(\TT))$
is an isomorphism.
Then
\be
\item
If the map 
$\Comp_2$ is a monomorphism
then so is $\Comp_1$.
\item
If the 
map $\Comp_2$ is a
split monomorphism, then
$\Comp_1$ is a split monomorphism.
\ee
\ermk

Combining \autoref{Corgeneralmonomorphism} with \autoref{cor:finTwa} and \autoref{corollaryautomorphisms} gives the following.

\cor{Rmkallgeneralmonomorphism}
Let $\ct$ be a weakly approximable
triangulated category satisfying 
one of the following
conditions:
\be
\item
Either $\ct$ is coherent.
\item
Or we have $\ct^c\subset\ct_c^b$ and the inclusion  
makes $\ct^c$ a characteristic subcategory.
\setcounter{enumiv}{\value{enumi}}
\ee
We suppose 
moreover that 
\be
\setcounter{enumi}{\value{enumiv}}
\item
The category $\ct$
has an enhancement in $\HoStabInfA$. 
\item
The category $\ct^c$
has a unique enhancement
$\TT^c\in\HoStabInfA$.
\item
The comparison map
$\Comp\colon\AutL(\TT^c)\la\Aut(\ct^c)$
is (split) injective.
\setcounter{enumiv}{\value{enumi}}
\ee
Then the groups $\AutL(\TT^\star_\Box)$ are all
isomorphic, where $\TT^\star_\Box$ denotes an enhancement (unique by \autoref{cor:finTwa}) of $\ct^\star_\Box\in\{\ct^c,\ct^b_c,\tsb,\ct^b,\ct^-,\ct^+,\ct\}$,
and all the natural restriction
maps among them are isomorphisms. Furthermore
all the comparison maps
$\Comp\colon\AutL(\TT^\star_\Box)\la\Aut(\ct^\star_\Box)$
are (split) injective.
\ecor

\prf
The fact that the natural
restriction maps
$\Res_\infty\colon\AutL(\TT^\star_\Box)\la\AutL(\TT^\dagger_{\dagger\dagger})$
are all isomorphisms was seen in
\autoref{corollaryautomorphisms}.
And now because every one of the categories
in the list contains either
$\ct^c$ or (in case $\ct$ is coherent) the category
$\ct^b_c$, the (split) injectivity in (v)
coupled with \autoref{Corgeneralmonomorphism}
implies the (split) injectivity of all the maps
$\Comp\colon\AutL(\TT^\star_\Box)\la\Aut(\ct^\star_\Box)$.
\eprf

In the geometric setting we get the following:

\pro{Literature}
Assume that $X$ is a scheme
proper over a field $\kK$, and suppose further that
it has depth $\geq 1$.
Let $\ct=\Dqc(X)$, and therefore $\ct^c=\dperf{X}$.
Then the hypotheses of
\autoref{Rmkallgeneralmonomorphism} are satisfied, in fact the map
$\Comp\colon\AutL(\TT^c)\la\Aut(\ct^c)$ is an isomorphism---in
particular split injective. Thus, all the maps
$\Comp\colon\AutL(\TT^\star_\Box)\la\Aut(\ct^\star_\Box)$
are split injective, and their images
are all isomorphic.
\epro

\prf
First we claim that there is a natural injective map
\[
\Gamma\colon\Aut(\ct^c)\la\Iso(\Dqc(X\times X)),
\]
where $\Iso(\cc)$ denotes the isomorphism classes of objects in a category $\cc$. In fact $\Gamma$ maps $\fF\in\Aut(\ct^c)$ to the isomorphism class of an object $E\in\Dqc(X\times X)$ such that $\fF$ is isomorphic to the Fourier-Mukai functor with kernel $E$. Thus the claim amounts to the fact that every such $\fF$ is of Fourier-Mukai type with unique (up to isomorphism) kernel. Now, if $X$ is smooth and projective (over $\kK$) this is a classical result going back to \cite[2.2 Theorem]{Orlov97}. Moreover, $\fF$ is known to be of Fourier-Mukai type when $X$ is either projective with depth $\ge1$ (by \cite[Corollary 9.13]{Lunts-Orlov10}) or smooth and proper (by \cite[Theorem 1]{OlanderOrlov}). The same is true when $X$ is proper with depth $\ge1$ because the same argument of \cite{OlanderOrlov}, extensively described in \cite[Chapter 3]{OlanderThesis}, works in this more general setting, since the set of objects in \cite[Section 3.3]{OlanderThesis} is still almost ample. As for uniqueness of the kernel, the proof of \cite[Theorem 3.0.1]{OlanderThesis} in the smooth and proper case (which is formally the same as the one in \cite[Section 4.3]{Canonaco-Stellari07}) does not really need the smoothness assumption.

On the other hand, by \cite[Corollary 8.12]{To}, there is also an injective map
\[
\Gamma_\infty\colon\AutL(\TT^c)\la\Iso(\Dqc(X\times X))
\]
(for a suitable choice of the enhancement $\TT$ of $\ct$) such that $\Gamma$ and $\Gamma_\infty$ have the same image. Therefore, to conclude that $\Comp\colon\AutL(\TT^c)\la\Aut(\ct^c)$ is an isomorphism, it is enough to know that $\Gamma\comp\Comp=\Gamma_\infty$. This equality would follow from an unproved claim in \cite[Section 8.3]{To}. While a full proof of the claim will appear in \cite{COS3}, a proof of a weaker version of the claim is already available in \cite[Theorem 1.1]{LS}, and allows to obtain the desired equality $\Gamma\comp\Comp=\Gamma_\infty$ under the extra hypothesis that every object of $\dperf{X}$ is isomorphic to a strictly perfect complex.
\eprf


\bigskip

{\small\noindent{\bf Acknowledgements.} Part of this work was carried out while the third author was visiting the Laboratoire de Math\'ematiques d'Orsay (Universit\'e Paris-Saclay) and the Institut des Hautes \'Etudes Scientifiques (Paris) whose warm hospitality is gratefully acknowledged. We are very grateful to Benjamin Antieau, Federico Binda, Denis-Charles Cisinski and Alberto Vezzani for patiently answering our questions. It is a pleasure to thank Evgeny Shinder, for suggesting to us the possibility of applying our results to autoequivalence groups in the geometric setting.}

\bibliographystyle{abbrv}

\end{document}